\documentclass[10pt]{article}

\usepackage{amsmath, amsthm,amssymb,mathscinet}

\usepackage[section]{algorithm}
\usepackage{algorithmic}

\usepackage{tikz}

\newcommand{\Sl}{\mathfrak{sl}_2}
\newcommand{{\N}}{{\rm N}}
\newcommand{{\M}}{{\rm M}}

\newcommand{\f}{\mathbb{F}}

\newcommand{\ad}[1]{\mathsf{ad}_{#1}}

\newcommand{\0}{\mathbf{0}}

\theoremstyle{definition}
\newtheorem{Theorem}{Theorem}[section]
\newtheorem{Corollary}[Theorem]{Corollary}
\newtheorem{Lemma}[Theorem]{Lemma}
\newtheorem{Definition}[Theorem]{Definition}
\newtheorem{Example}[Theorem]{Example}
\newtheorem{Remark}[Theorem]{Remark}

\newtheorem{Notation}[Theorem]{Notation}
\newtheorem{Acknowledgement}[Theorem]{Acknowledgement}
\numberwithin{equation}{section}
\usepackage{color}
\usepackage{multirow}
\newcounter{IssueCounter}
\newtheorem{Issue}[IssueCounter]{Issue}
\newcommand{\issue}[2]{
\begin{Issue}[\textcolor{red}{#1}]{\textcolor{blue}{#2}}\end{Issue}}

\newcommand{\tr}[1]{\mathrm{tr}(#1)}

\def\be{\begin{equation}}
\def\ee{\end{equation}}
\def\ba{\begin{eqnarray}}
\def\ea{\end{eqnarray}}
\def\bpr{\begin{proof}}
\def\epr{\end{proof}}
\def\bes{\begin{equation*}}
\def\ees{\end{equation*}}
\def\bas{\begin{eqnarray*}}
\def\eas{\end{eqnarray*}}
\def\im{\mathrm{im}\,}

\begin{document}
\bibliographystyle{alpha}
\baselineskip=18pt
\pagestyle{empty}
\tikzstyle{dynkinnode}=[draw, color=cyan, shape=circle,minimum size=3.5 pt,inner sep=0]
\tikzstyle{udynkinnode}=[draw, color=green, shape=circle,minimum size=3.5 pt,inner sep=0]

\tikzstyle{ldynkinnode}=[draw, color=cyan, shape=circle,minimum size=3.5 pt,inner sep=0]

\tikzstyle{sdynkinnode}=[draw, color=red, shape=circle,minimum size=3.5 pt,inner sep=0]

\tikzstyle{fdynkinnode}=[draw, color=cyan, shape=circle,minimum size=3.5 pt,inner sep=0,fill=black]

\tikzstyle{fldynkinnode}=[draw, color=cyan, shape=circle,minimum size=3.5 pt,inner sep=0,fill=black]

\tikzstyle{fsdynkinnode}=[draw, color=red, shape=circle,minimum size=3.5 pt,inner sep=0,fill=black]

\title{The Lie algebraic structure of  linear colored network dynamics}

\author{
Fahimeh Mokhtari\\Jan A. Sanders \\
Department of Mathematics, Faculty of Sciences\\
Vrije Universiteit, De Boelelaan 1111, 1081 HV Amsterdam, The Netherlands
}

\maketitle

\pagestyle{plain}
\vspace{-0.10in}

\baselineskip=15pt

\:\:\:\:\ \ \rule{5.88in}{0.012in}
\def\art{1}
\def\aa{a}
\def\bb{b}
\def\cc{c}
\def\dd{o}
\def\int{30}

\def\aa{a}
\def\bb{b}
\def\cc{c}
\newcommand{\FA}[2]{\mathfrak{{\aa}}^{#2}_{#1}} % elements of Lie subalgebra
\newcommand{\GA}[2]{{{\aa}}^{#2}_{#1}} 
\newcommand{\FAb}[2]{\mathfrak{\bar{\aa}}^{#2}_{#1}} % elements of Lie subalgebra
\newcommand{\sFA}[2]{\mathfrak{{\aa}}^{}_{}}    % Lie subalgebra
\newcommand{\ssFA}[2]{\mathfrak{\bar{\aa}}^{}_{}}    % Lie subalgebra
\newcommand{\cFA}[2]{\mathsf{{\aa}}^{#1}_{#2}}  % coefficients of elements of Lie subalgebra (dual)
\newcommand{\FB}[2]{\mathfrak{{\bb}}^{#2}_{#1}}
\newcommand{\GB}[2]{{{\bb}}^{#2}_{#1}}
\newcommand{\FBb}[2]{\mathfrak{\bar{\bb}}^{#2}_{#1}}
\newcommand{\sFB}[2]{\mathfrak{{\bb}}^{}_{}}
\newcommand{\ssFB}[2]{\mathfrak{\bar{\bb}}^{}_{}}
\newcommand{\cFB}[2]{\mathsf{{\bb}}^{#1}_{#2}}
\newcommand{\FC}[2]{\mathfrak{{\cc}}^{#2}_{#1}}
\newcommand{\GC}[2]{{{\cc}}^{#2}_{#1}}
\newcommand{\FCb}[2]{\mathfrak{\bar{\cc}}^{#2}_{#1}}
\newcommand{\sFC}[2]{\mathfrak{{\cc}}^{}_{}}
\newcommand{\ssFC}[2]{\mathfrak{\bar{\cc}}^{}_{}}
\newcommand{\cFC}[2]{\mathsf{{\cc}}^{#1}_{#2}}
\newcommand{\FO}[2]{\mathfrak{{\dd}}^{#2}_{#1}}
\newcommand{\GO}[2]{{{\dd}}^{#2}_{#1}}
\newcommand{\sFO}[2]{\mathfrak{{\dd}}^{}_{}}
\def\BB{{\expandafter\MakeUppercase\expandafter{\bb}}}
\def\CC{{\expandafter\MakeUppercase\expandafter{\cc}}}
\newcommand{\cFH}[2]{\bar{\mathsf{{\cc}}}^{#1}_{#2}}
\def\cBB{\mathcal{\BB}}
\def\cCC{\mathcal{\CC}}
\def\dBB{\mathsf{\BB}}
\def\dCC{\mathsf{\CC}}
\def\dI{\mathsf{I}}
\def\dX{\mathsf{X}}
\def\dim{\mathrm{dim}\ }
\newcommand{\cu}[1]{\mathsf{u}_{#1}}
\newcommand{\bu}[1]{\mathsf{v}_{#1}}
\newcommand{\cU}[1]{\mathsf{u}^{#1}}
\newcommand{\bU}[1]{\mathsf{v}^{#1}}
\newcommand{\cx}[1]{\mathsf{x}_{#1}}
\newcommand{\dcx}[1]{\dot{\mathsf{x}}_{#1}}
\newcommand{\dbx}[1]{\dot{\mathsf{y}}_{#1}}
\newcommand{\bx}[1]{\mathsf{y}_{#1}}
\newcommand{\cX}[1]{\mathsf{x}^{#1}}
\newcommand{\bX}[1]{\mathsf{y}^{#1}}
\newcommand{\cp}[1]{\mathsf{e}_{#1}}
\newcommand{\cP}[1]{\mathsf{E}^{#1}}
\newcommand{\bp}[1]{\mathsf{f}_{#1}}
\newcommand{\bP}[1]{\mathsf{F}^{#1}}
\newcommand{\lu}[2]{\underline{\mathbf{u}}^{#2}_{#1}}
\usetikzlibrary{calc}
\newcommand{\E}[2]{\mu^{#2}_{#1}}
\newcommand{\F}[2]{\nu^{#1}_{#2}}
\newcommand\ev[1]{\mathsf{V}_{#1}}
\newcommand\dvv[2]{\dot{\mathbf{u}}_{#1}^{#2}}
\newcommand\vv[2]{\mathbf{u}_{#1}^{#2}}
\newcommand\VV[2]{\bar{\mathbf{u}}^{#1}_{#2}}
\newcommand\xx[2]{\mathbf{x}_{#1}^{#2}}
\newcommand\XX[2]{\bar{\mathbf{x}}^{#1}_{#2}}
\newcommand\eV[1]{\mathsf{V}^{\star}_{#1}}
\newcommand\Dx[3]{{\color{c#1}#2}_{\color{co#3}#1}}
\newcommand\Ux[2]{{\color{cc#1}#2}^{\color{cc#1}#1}}

\maketitle
\noindent 
\providecommand{\keywords}[1]
{
	\small	
	\textbf{\textit{Keywords---}} #1
}

\begin{abstract}

This paper explores the category of colored network dynamical systems, a class of network systems characterized by specific structural features. In these networks, components of the same color share identical functions within the differential equations governing their dynamics. This models systems with layers of identical components, but different inputs, grouped by color.

Our primary objective is to explore the linear structure inherent in these network dynamical systems, with the ultimate goal
of the computation of their normal form to facilitate the local study of their dynamics. This is important to us if the organizing center
is nonsemisimple. If it is semisimple, the standard normal form theory applies without problem.
To compute the normal form of a colored network vector field, we employ the semigroup(oid) approach introduced in \cite{rink2015coupled,MR3071397}. We aim to elucidate the structure of the Lie algebra of linear colored network vector fields.
%, contrasting with the purely abstract approach outlined in \cite{MR3606590}.

We present a concrete algorithm for deriving the Levi decomposition
of the Lie algebra \(\mathfrak{net}_{\CC,N}\) of all \(N\)-dimensional linear vector fields of \(N\) dimensional with \(\CC\) colors (representing distinct functions describing various types of cells/nodes/components in the network). It has a Levi subalgebra isomorphic to the direct sum of  
\(\mathfrak{sl}_{\CC}\) and \(\mathfrak{sl}_{\BB}\) (with \(\BB=N-\CC\)). Moreover, the solvable part consists of elements representing the identity \(\dCC\) in \(\FC{}{}\simeq\mathfrak{gl}_{\CC}\) 
and \(\dBB\) in \(\sFB{}{}\simeq\mathfrak{gl}_{\BB}\), along with an abelian algebra \(\sFA{}{}\simeq\mathfrak{Gr}(\CC,N)\) (the Grassmannian). 
%This algebra comprises \(\CC\)-dimensional subspaces of \(\mathbb{R}^N\). 
We give an algorithm to explicitly write any given colored matrix in the corresponding block form.
If the nilpotent part of the linear organizing center of a given system lies in the Levi subalgebra,
standard normal form theory applies. If not, 
this creates a new type of problem, which we intend to study in the near future. 

As a byproduct, we show in several examples how our approach delivers the {\em multipliers} determining the spectrum,
and this will be very useful in the bifurcation analysis of critical situations and their control.
In this paper, we will take a rather pragmatic view of these multipliers, in the sense that we make sure that they deliver the eigenvalues
when questioned, but maybe bigger than what is common from the representation point of view. 
This approach is motivated by the consideration that although our examples are rather low dimensional,
in practice, the problems may well be high dimensional, and it pays off to minimize the amount of computing.

%\issue{FM}{the first question that you can ask any Lie algebra is the Levi decomposition this Levi decomposition important?
%Suppose that the Levi decomposition has only the semisimple part then finding the normal form would be easy in this case, or at least one knows how to 
%the element in the solvable part always makes the nilpotent.  
%}
%\issue{FM}{
%If we apply this algorithm to a given subalgebra, it does put the elements in the subalgebra in the block form given by the Levi %decomposition, 
%but this need not be the Levi decomposition of the given subalgebra. 
%The methods in this paper can be immediately applied to study the linear maps of colored networks.}
\end{abstract}

%\tableofcontents

\section{Introduction}

In science and technology, many systems can be represented as networks comprising interconnected nodes. Examples of such systems span various domains, including biology (metabolic systems), economics, sociology, computer science, information theory, neuroscience, electronic engineering (power grid), and
ecology (food chains). These networks are frequently described by sets of coupled non-linear differential equations, forming a network dynamical system However, the network structure often introduces specific conditions in the differential equations, such as feedforward connections, identical nodes, or symmetry, which invalidate conventional methods. 

From a dynamical system perspective, the objective is to develop approaches that analyze dynamics in critical scenarios while preserving the network structure and capturing emerging phenomena within the normalized model equations. 

This paper focuses on the preparation of the linear part of a colored network for study in normal form. These networks consist of subgroups with identical nodes but potentially different inputs \cite{wu2002synchronization,sevilla2015enhancing,boccaletti2014structure,in2006complex}. See below an example of a colored network with five nodes and one color. 
\def\NN{5}
	\begin{Notation}\label{not:pic}
		Two kinds of colors play a role in the pictures: the colors of the cells in the network (that is, in the differential equation, the \(i\) in \(f^i\), which we will call \(\cc\)-colors, and the colors of the arrows, determined by the position in each \(f^i\). These we call \(\sigma\)-colors. These colors are only needed to make sure that one can recover the differential equation (up to the permutation of the arguments). This is slightly confusing since the arrows also have a \(\cc\)-color,
  determined by the color of their origin. But this color is clear from the picture and does not discern between the positions.
		In our pictures of the networks, we use the following conventions. 
		The big balls are the cells (or nodes), with \(f^1\) red-, \(f^2\) violet- and \(f^3\) blue-colored, 
		where the \(\cc\)-colors are of course almost completely arbitrary and chosen for their suitability to convey the following information. 
		The small balls \(\tikz \fill[red] (1ex, 1ex) circle (1ex);\) denote the exit point of an arrow, in this case from a red cell, and the \(\sigma\)-color of the arrow
		indicates the position in \(f^i\), while the thickness of the arrow is determined by the \(\cc\)-color \(i\):
		 \(\tikz \draw[violet,line width=1pt] (20pt,0pt) to [out=0,in=180] node [sloped] {$>$} (0pt,6pt);\),
		 \(\tikz \draw[violet,line width=1.5pt] (20pt,0pt) to [out=0,in=180] node [sloped] {$>$} (0pt,6pt);\),
		 \(\tikz \draw[violet,line width=2pt] (20pt,0pt) to [out=0,in=180] node [sloped] {$>$} (0pt,6pt);\),
		the \(\tikz \fill[red!\int] (1ex,1ex) circle (1ex);\) denote the arrival point of the arrow
		and the \(\tikz\draw[blue,line width=1pt] (0,0) to[loop below] node [sloped] {$$} (0,0);\) denotes a blue arrow going out and back in again (selfinteraction).
		The differential equations and the picture are completely equivalent if the order of the \(\sigma\)-colors is known.
	\end{Notation}

In this paper, our objective is to investigate the colored network dynamical system by elucidating its linear structure. 
Apart from giving insight into the structure of linear differential equations, this effort represents a crucial step toward the computation and classification of the normal form of the nonlinear terms near equilibrium and the local study of bifurcation phenomena. Before exploring the details in the Introduction, we provide some preliminaries on the normal form theory and discuss the challenges it presents.

%\issue{JS}{Should we remove this part, now that there is a section on normal form theory?}
Normal form theory is a fundamental concept in the study of vector fields and dynamical systems, generalizing the concepts of Jordan normal form and rational normal form in the linear theory to the nonlinear case.
It simplifies complex systems by locally transforming them into a more manageable form through coordinate changes. This normalization process facilitates the analysis of stability, bifurcations, and other key properties moving in the process
parameters that play no essential role in a higher order. 
It also provides insights into the underlying structure of dynamical systems, making it a powerful tool for understanding their behavior.

\iffalse
In the study of normal forms for the given vector field
\(\dot{x}={\mathsf A}x+f(x)\), where \(f(x)\) is the non-linear part, \(x\in \mathbb{R}^n\), one should determine the type of linear component, whether the matrix \(\mathsf A\)  is nilpotent \(\mathsf{N}\), semisimple \(\mathsf{S}\), or nonsemisimple \(\mathsf{S}+\mathsf{N}\). Depending on this classification, different normal-form styles emerge. For example, if the linear component is nilpotent, one can employ the
\(\Sl\)-style of normal form. The Jacobson–Morozov theorem guarantees the existence of a nilpotent $\mathsf{M}$ and semi-simple operator $\mathsf{H}$ such that $\langle \mathsf{N},\mathsf{H},\mathsf{M} \rangle$ forms a $\Sl$ triple.
Recall that the \(\Sl\)-triple is generated  by \(\langle\mathsf{N},\mathsf{H},\mathsf{M}\rangle\)  with the following commutator relations \([\mathsf{H}, \mathsf{N}]=-2\mathsf{N}, 
 [\mathsf{H},\mathsf{M}]=2 \mathsf{M},\) and \([\mathsf{M}, \mathsf{N}]= \mathsf{H}\), since the space of vector fields of fixed degree can be
\(\ker{\rm ad} (\mathsf{M})\oplus\im{\rm ad} (\mathsf{N}).\)
The normal form is generated by $\ker{\rm ad} (\mathsf{M})$, where ${\rm ad}(x)y=[x,y]$,
\cite{baider1992further,sanders2002spectral,cushman1988normal,cushman1988splitting,cushman56nilpotent,SVM2007}.
%Normal form for nonlinear vector fields equivalent to the Jordan forms for the matrixes. 
\fi

Normal form theory, especially when there is linear nilpotency involved,
 relies on some nontrivial Lie algebra results, such as Chevalley decomposition, splitting an element in its semisimple and nilpotent part, and the Jacobson-Morozov theorem, which extends a nilpotent to an \(\Sl\). This last result, in turn, relies on the underlying Lie algebra of linear vector fields being reductive, that is, it should be the direct sum of
 its semisimple part and its center. Notice that a semisimple element does not necessarily have to be in the semisimple component (the identity matrix in \(\mathfrak{gl}_N\) is the standard example), nor does an element in the semisimple component need to be semisimple.
Confusing as this may be at first, the two usages of the term semisimple come together in the Cartan subalgebra of a semisimple Lie algebra,
 which consists of commuting semisimple elements.

%In various scientific fields, including neuroscience, physics, ecology, and sociology, dynamic systems with network structures serve as valuable models. Among these networks, {\em Colored networks} represent a significant class where the interaction between layers of nodes shares identical input types.

%The objective of this project is to conduct a comprehensive dynamic analysis of colored network dynamical systems. We will accomplish this by leveraging and developing mathematical theories for linear and nonlinear normal forms, thereby producing a collection of readily applicable results. The rich phenomenology of these networks, attributed to the presence of identical nodes, further motivates this study.
This paper started as an attempt to find out how much of ordinary normal form theory could still
be used in the context of colored networks. Normal form theory was the original motivation for the semigroupoid approach in \cite{rink2015coupled}, but computing normal forms for concrete problems turned out
not to be so easy. Since the structure of the Lie algebra was not known to us initially, it was difficult to check the reductiveness, in general. 

\begin{Notation}
		We let, for a given equilibrium \(\mathbf{x}_0\),
	\bas
		f^i_j=\frac{\partial f^i}{\partial x_j}(\mathbf{x}_0),\quad i=1,\ldots,\CC.
	\eas
\end{Notation} 
\begin{Remark}\label{rem:vector}
This notation implicitly assumes that \(f^i\) takes its values in \(\mathbb{R}\) and, as a consequence, the \(f^i_j\) are commuting. 
In \S\ref{sec:2osc} we give an example 
where the \(f^i\) take their values in \(\mathbb{R}^2\) and \(f^i_j\) is a \(2\times 2\)-matrix. The main consequence of
this, is that  we can no longer assume the \(f^i_j\) to be commuting. This implies that we can still use the results in this paper as long as they are linear in the \(f^i_j\) and that we have to rethink the nonlinear consequences, for instance when we compute determinants (as we will do in the next example) or in the Levi decomposition, as in all other Lie algebraic results.
\end{Remark}
\begin{Example}\label{exm:2Dnil}
The application of Jacobson-Morozov is an issue that is already playing a role in the following \(2\)-dimensional, \(1\)-color, case,
describing the \(\mathfrak{net}_{1,2}\) situation:
\bas
\dot{x}_1&=&f(x_1,x_1,x_2,x_2),\\
\dot{x}_2&=&f(x_1,x_2,x_1,x_2).
\eas
The Jacobi-matrix of this system is \(f_i=\frac{\partial f}{\partial x_i}\),
\bas
J&=&\begin{pmatrix} f_1+f_2&f_3+f_4\\f_1+f_3&f_2+f_4\end{pmatrix}.
\eas
Remark that the family of Jacobi matrices is at most \(3\) dimensional since both row sums are equal.
For this matrix to be nilpotent, we need its trace and determinant to be zero.
Solving the equations by eliminating \(f_1, f_3\), we see that a general nilpotent matrix in this algebra looks like this:
\bas
(f_2+f_4)\begin{pmatrix} -1&1\\-1&1\end{pmatrix}.
\eas
This implies that the Jacobson-Morozov construction is bound to fail, since with only one nilpotent generator, there cannot
be two linearly independent nilpotents necessary for the existence of an \(\Sl\) in which our first nilpotent should be embedded.

In this case the {\em semigroup} is generated by the maps  \(\sigma_i, i=1,i\ldots,4\) from the index set \(\langle 1,2\rangle\) to itself,
given by rewriting the equation as
\bas
\dot{x}_1&=&f(x_{\sigma_1(1)}, x_{\sigma_{2}(1)}, x_{\sigma_{3}(1)}, x_{\sigma_{4}(1)}),\\
\dot{x}_2&=&f(x_{\sigma_1(2)}, x_{\sigma_{2}(2)}, x_{\sigma_{3}(2)}, x_{\sigma_{4}(2)}).
\eas
One verifies that \(\sigma_2\) is the identity of this semigroup, where the multiplication is given by the associative composition of the maps. Since there is an identity, this also goes by the name of the monoid {\em}.
If there is more than one color, not all \(\sigma\)s can be composed. In that case, they generate a {\em semigroupoid}.

By representing the \(\sigma_i\)s as matrices, we introduce the semigroup algebra and the associated Lie algebra as follows.
By writing a general element of the semigroup algebra as
\bas
f_1\begin{pmatrix}1&0\\1&0\end{pmatrix}+f_2\begin{pmatrix}1&0\\0&1\end{pmatrix}+f_3\begin{pmatrix}0&1\\1&0\end{pmatrix}+f_4\begin{pmatrix}0&1\\0&1\end{pmatrix},
\eas
one can then introduce a Lie bracket in the usual way: \([A,B]=AB-BA\).
This example is continued in Example \ref{exm:dim2col1exm1}, but with \(\sigma_1\) and \(\sigma_2\) interchanged
and in \S\ref{sec:2osc} in a somewhat different setting.
\end{Example}

We started by trying to put the commutation relations of some concrete systems in the Jordan-Chevalley decomposition form \cite{humphreys2012introduction}. We collected information on the dimensions of the components of the Levi decomposition
using \cite{MR2189632}. This was a bit of a frustrating activity since there was an exceptional case, namely ordinary differential equations without any color structure, that is, \(\CC=N\), where these dimensions follow a slightly different formula, making extrapolation very difficult (see the proof of Corollary \ref{cor:hom1} for details). However, with the correct general formula, it was not difficult to guess the
structure of the Lie algebra given the number of cells and colors. What remained was a search for a unifying proof for all cases, and this is given here. That the final answer only depends on the number of cells and the number of colors was something that we had not expected at all when we started.

We will show in this paper that the dimension of the network Lie algebra \(\mathfrak{net}_{\CC,N}\), that is the algebra of the Total Network with \(N\) cells and \(\CC\) colors, \(1\leq \CC\leq N\),
equals \(\BB^2+\BB\CC+\CC^2\), with \(\BB=N-\CC\) (Some readers may want to simplify this expression to \(N^2-\BB\CC\), others may want to conjecture the structure of the Lie algebra from it).
We notice that this result is as simple as it could be: it does not depend on how many cells have a given color; only the totals \(N\) and \(\CC\) play a role; 

The Levi decomposition (cf. \cite[p.101]{MR0123464})
is given by solvable part
\bas
\begin{bmatrix} \dCC &0\\ \sFA{}{} & \dBB\end{bmatrix},
\eas
where \( \sFA{}{}\) is abelian and has dimension \(\BB\CC\), \(\dCC\) and \(\dBB\) are elements 
such that \(\dBB+\dCC=\dI\in\mathcal{Z}(\mathfrak{net}_{\CC,N})\), the center of \(\mathfrak{net}_{\CC,N}\), and the semisimple part, the aforementioned Levi subalgebra,
\bas
	\begin{bmatrix} 
		\hat{\FC{}{}} & 0 \\ 
		0& \hat{\FB{}{}}
	\end{bmatrix}
	&\equiv&
	\begin{bmatrix} 
		\mathfrak{sl}_\CC & 0 \\ 
		0& \mathfrak{sl}_{\BB}
	\end{bmatrix}.
\eas
\begin{Remark}\label{rem:quotient1}
If we identify all variables with the same color, then this quotient network has no cocolors, so it
linearizes to \(\FC{}{}\).
\end{Remark}
The proof of all this consists of several steps.
First, we choose a basis for the matrices of maps from one color space to another.
This choice is different from the classical choice of maps from one color to itself, which is based on the diagonal
and is no longer very natural in the multicolored context.
Writing out the basis in terms of two-tensors suggests to us a new choice of basis, where we group the colors and the cocolors.
Using this new basis, we can now explicitly construct the Lie subalgebras \(\sFA{}{}, \sFB{}{}\) and \(\sFC{}{}\)
and show that they are isomorphic to \(\mathfrak{Gr}_{\CC,N}\), \(\mathfrak{gl}_{\BB}\) and \(\mathfrak{gl}_{\CC}\), respectively.
\iffalse
If \(\mathfrak{g}\) is a subalgebra of \(\mathfrak{net}_{\CC,N}\), one can generate the transformations on \(\mathfrak{g}\) by  the {\bf normalizer }
\bas
\mathsf{norm}_{\mathfrak{net}_{\CC, N}}(\mathfrak{g})&=&\{x\in\mathfrak{net}_{\CC, N}|[x, y]\in\mathfrak{g}, \forall y\in\mathfrak{g}\}.
\eas
This is a matter of taste: one can be restrictive and generate the transformations on \(\mathfrak{g}\) by \(\mathfrak{g}\), or one can be less restrictive and allow more transformations
(and usually a normal form with fewer terms).
This is comparable to the situation in Hamiltonian mechanics, where one might require the transformations to be symplectic.
But for some questions, it is allowed to use more general transformations.
The important thing is to always explicitly mention which choices have been made in any particular theory or computation.
\fi
At the initial stage of our research, some choices (like the bases and the linear map) were inspired by symbolic calculations in Form \cite{Form00} and Maple \cite{MR2189632}.
	
In \S \ref{sec:motivexs} we give two examples with three cells and one, respectively two, colors.
In \S \ref{sec:structure} we start with the semigroupoid formalism to determine a basis for the Lie algebra \(\mathfrak{net}_{\CC,N}\) of linear network vector fields with dimension \(N\) and \(\CC\) colors.
We do this by appointing for every color a coordinate, which we will call the {\em color}; the remaining coordinates with that color will be called {\em cocolors}. To prove our Structure 
Theorem \ref{thm:structure}, the choice of the color coordinate is completely arbitrary, all coordinates are treated as equals. But when it comes to computing a nice representation of the linearized vector field, it will turn out that some coordinates are more equal than others, depending on one's taste and/or the structure of the example.

We formulate Algorithm \ref{alg:ff} to fix the choice of colors and cocolors, which tries to keep a possible feedforward structure intact.
In \S \ref{sec:xybasis} we introduce a simple invertible transformation (Algorithm \ref{alg:xbasis}) that puts the Jacobi-matrix of any given linear network differential equation in the desired block form.
In \S \ref{sec:Lie} we compute the structure constants of \(\mathfrak{net}_{\CC,N}\) and define an involution \(\theta: \mathfrak{net}_{\CC,N}\rightarrow\mathfrak{net}_{\BB,N}\).
In \S \ref{sec:structureLie} we describe \(\mathfrak{net}_{\CC,N}\) in terms of its subalgebras and we determine the Levi decomposition.
In \S \ref{sec:dualpair} we show the existence of a dual pair within \(\mathfrak{net}_{\CC,N}\).
Finally, we give in \S \ref{sec:examples} several examples, with \(N\) ranging from two to eight, followed by the conclusions in \S \ref{sec:conclusions}.

\begin{Remark}
	General background and motivation for the study of linear network vector fields with \(\CC=1\), the homogeneous case, can be found in \cite{MR2481276}.
	Although we follow the semigroupoid approach (cf. \cite{rink2015coupled}), we want to emphasize that the linear transformation 
	that we use to obtain the feedforward block form can be applied without any knowledge of this approach.
	Readers who are only interested in the application of the given algorithmic approach can get quite far by going back and forth between the crucial Algorithm \ref{alg:xbasis},
	Definition \ref{def:xbasis} 
	and the Examples in \S \ref{sec:examples}.
 That no knowledge of semigroup theory is required
 reflects the lack of knowledge of the authors of this approach.
 This in turn can be blamed on the semigroup experts, whose work is very difficult to read for non-experts.
 A good place to start, however, is \cite[\S 5.3]{MR3525092}.
 We took our \(8\)-dimensional example from \cite{MR4183886}, which paper sets out to apply the representation theory of monoids (that is, semigroups with identity) to the computation of the spectrum of the linearized equation.
 Our paper can be seen as the next step in this approach, since it allows us to explicitly transform the linear system
 to block-diagonal form, with the {\em multipliers} on the diagonal, in contrast to loc. cit., where this computation is avoided and 
 possible multipliers are constructed by taking quotients and dimension arguments, based on the representation theory, which is used to
 compute the decomposition of the characteristic invariants.
\end{Remark}
\begin{Remark}
	We remark here that in the literature (cf. \cite{MR4059374})
	the homogeneous case (\(\CC=1\)) is seen as opposite to the fully inhomogeneous case (\(\CC=N\)), that is, general ODEs. We show in this paper that they
	can also be considered as neighbors, since the \(\CC=1\) case turns out to be isomorphic to the \(\CC=N-1\)-case, cf. Theorem \ref{thm:BCdual}.
\end{Remark}

\begin{Acknowledgement}
The authors want to thank Vincent Knibbeler for his careful reading of an earlier version and for pointing out errors in the theoretical setup. After correcting these,
we could finally give proof of what we previously hoped to be obvious.
\end{Acknowledgement}
\section{Preliminaries }
This section aims to clarify the fundamental definitions and terminology established in \cite{rink2015coupled}, to facilitate understanding of the subsequent results presented here.
\subsection{Normal form theory}
The results of this paper are motivated by questions arising in normal form theory.
Here we explain how they arise.
The discussion will be somewhat simplified. For a more extensive treatment, see \cite{SVM2007} and references therein.
In this paper we are not going to give results on nonlinear normal form theory of colored networks, we just prepare the way.

Consider a  differential equation 
\bas
\dot{x}&=& \mathsf{A}x+f(x), \quad x\in\mathbb{R}^n.
\eas
We identify this with the first order differential operator \(F^{(0)}+F^{(1)}\), with \(\partial^i=\frac{\partial}{\partial x_i}\),
\bas
F^{(0)}&=&\sum_{i, j=1}^n \mathsf{A}_i^j x_j \partial^i, \quad 
F^{(1)}=\sum_{i=1}^n f_i(x)\partial^i.
\eas
\begin{Remark}
In the nonlinear context, we prefer the \(\partial^i\)-notation, to emphasize its derivative property, in a linear context we will use \(x^i\) instead since it allows us to express a change of coordinates more easily.
\end{Remark}
We now want to apply transformations to this operator, fixing the linear part determined by the {\em organizing center} \(\mathsf{A}\).
We do this by letting \(G^{(1)}=\sum_{k=1}^n g_k(x)\partial^k\) and applying the formal transformation
\bas
\exp(\ad{G^{(1)}})(F^{(0)}+F^{(1)}),
\eas
where \(\ad{G^{(1)}}F^{(0)}=[G^{(1)},F^{(0)}]^{(1)}\) is well-defined, since first-order operators form a Lie algebra.
The upper index indicates a filtration, that is, we assume \([G^{(k)},F^{(l}]=H^{(k+l)}\).
If the functions are formal power series, then the usual filtration is given by the lowest power in \(x\) minus \(1\),
but we may also allow linear terms in \(F^{(1)}\), cf. \cite{MR3953028}.
Think of operators starting with linear terms of having filtering degree \(0\) and starting with quadratic terms degree \(1\).
%But a similar effect can also be reached by introducing a small parameter.
With the filtering in place, the \(\exp\) can be computed up to any desired degree.
For instance, if we let denote by \(\mathcal{F}^{(k)}\) the operators of filtering degree \(k\), then
\bas
\exp(\ad{G^{(1)}})(F^{(0)}+F^{(1)})=F^{(0)}+F^{(1)}-[F^{(0)}, G^{(1)}]\mod{\mathcal{F}^{(2)}}.
\eas
If we have a direct (as modules, not necessarily as Lie algebras) summand \(\bar{\mathcal{F}}^{(1)}\) to \(\im\ad{\mathcal{F}^{(k)}}\), 
we can now write \({F}^{(1)}=\bar{F}^{(1)}+\ad{F^{(0)}}{G^{(1)}}\) and we say that \(F^{(0)}+\bar{F}^{(1)}\) is the 
{\em first order normal form with respect to \(\mathsf{A}\)}.
While the image of \(\ad{F^{(0)}}\) is given, the direct summand is subject to choice.
Any choice determines a {\em style} of normal form, a terminology, introduced by Jim Murdock, indicating both taste and fashion.

If \(\mathsf{A}\) is semisimple, say \(\mathsf{S}\), the universal choice is to take \(\bar{\mathcal{F}}^{(1)}=\ker\ad{\mathsf{S}}\).
If \(\mathsf{A}\) is nilpotent, say \(\mathsf{N}\), one might consider the kernel of the adjoint, which is quite natural in the context
of partial differential equations, but in the context of Lie algebras the natural choice seems to be the application
of the Jacobson-Morozov theorem, constructing a triple (or triad) \(\langle \mathsf{N,H,M}\rangle \approx \Sl\) and
define \(\bar{\mathcal{F}}^{(1)}=\ker\ad{\mathsf{M}}\), as suggested by the fact that in the (locally) finite-dimensional situation
\(\mathcal{F}=\im\ad{\mathsf{N}}\oplus\ker\ad{\mathsf{M}}\).
If \(\mathsf{A}=\mathsf{S}+\mathsf{N}\), with commuting semisimple and nilpotent, then we have
\(\mathcal{F}=\im\ad{\mathsf{S+N}}\oplus(\ker\ad{\mathsf{M}}\cap\ker\ad{\mathsf{S}})\).
We remark here that if one has to explicitly compute \(G^{(1)}\) from \(F^{(1)}\), this can be done by solving in style
\bas
\ad{\mathsf{S+M}}(F^{(1)}-[F^{(0)},G^{(1)}])=0\mod{\mathcal{F}^{(2)}}.
\eas
At this point, the description, and computation of the first-order normal form concerning \(\mathsf{A}\), is
well-defined and what remains is to carry out the program in any concrete problem setting.

But the whole discussion, at least when \(\mathsf{A}\) is nonsemisimple, depends on the existence of a nilpotent linear operator,
either the adjoint or the \(\Sl\)-adjoint \(\mathsf{M}\), and as we have seen in Example \ref{exm:2Dnil}, where there is only one nilpotent available,
this is not always the case in the context of colored networks. This paper aims to make the obstruction explicit for general
colored networks. If \(\mathsf{A}\) is semisimple, the situation does not need the theory developed here. In that case, the
only thing one needs to worry about is small divisors, since the action of a semisimple operator on tensor products is again semisimple.

We have illustrated the normal form theory of formal power series vector fields, but it should be clear from this discussion
that the same thing can be done for locally finite filtered Lie algebras in general.
\subsection{Semigroup(oid)}
  The semigroup(oid)s we consider in this paper consist of maps from subsets (colors) of the indices \(1,\ldots,N\).
  The only property they have is that the composition of two elements if it is at all possible, is an associative operation.
  Obviously, a map from color \(1\) to color \(2\) cannot be composed with a map from color \(3\) to color \(4\),
  but a map from color \(1\) to color \(2\) can be composed with a map from color \(2\) to color \(4\), and this composition should be associative. When there is only one color, the semigroupoid becomes a semigroup, or, when it has an identity, a monoid.
  The monoid can very well be a group, but this is something we completely ignore.
  If it is a group, this may show up in the spectrum.
  
  The composition shows up in the composition of transformations and is therefore essential for normal form theory.
  We will not use any of the theory of semigroup(oid)s, so no knowledge of this subject is required to follow the discussion.
  In practical situations, the differential equation may not contain the full semigroup(oid). In this case,
  the given maps are extended by adding all possible compositions.
  The result is a network where all shortcuts are added,
  the {\em Completed Network}. In its simplest form,
  consider
  \bas
  \dot{x}_1&=&f(x_2)=f(x_{\sigma_1(x_1)}),\\
  \dot{x}_2&=&f(x_1)=f(x_{\sigma_1(x_2)}).
  \eas
  Then the semigroup generator \(\sigma_1=(12)\) and \(\sigma_2:=\sigma_1^2\) is the identity.
  So the completed description is
  \bas
  \dot{x}_1&=&\hat{f}(x_2,x_1)=\hat{f}(x_{\sigma_1(x_1)},x_{\sigma_2(x_1}),\\
  \dot{x}_2&=&\hat{f}(x_1,x_2)=\hat{f}(x_{\sigma_1(x_2)},x_{\sigma_2(x_2}).
  \eas
  In a normal form calculation, starting with a given \(f\), this could be done step by step,
  just computing those compositions necessary to compute the brackets in the normal form calculation at each filtering degree
  and leading to an expression of the normal form in terms of some \(\hat{f}\).
  The semigroupoid formalism was introduced in \cite{rink2015coupled}, in \cite{MR4183886} this formalism is explained with many examples
  and the application of representation theory, leading to results that seem to be very useful for the present paper.
\begin{Remark}
    In many publications, one assumes that the function \(f\) is symmetric under the permutation of its arguments
    (this is usually indicated by \(f(\overline{x_1,...,x_N})\)).
    Whether this assumption is based on the properties of some existing model or on the fact that certain calculations
    will be much simplified is not relevant here; what is relevant is that if one wants to keep this symmetry in the normal form transformation, the space of allowed transformations will be reduced, in many cases to the identity transform.
\end{Remark}
	\section{Some motivating examples}\label{sec:motivexs}
\def\NN{3}
\subsection{{$N=3, \CC=1$}}\label{sec:motivex}
In the first example, we intend to explain two things: First, we
want to apply the results from \cite{rink2015coupled} to show how the semigroup appears.
We identify the generators \(\sigma_i\) and show that they form a semigroup in this case, so the Generating Network happens to be the Completed Network and the normal form can be expressed using the given \(\sigma\)'s. Second, we show how the main result, Algorithm \ref{alg:xbasis} works here, and later on, we provide more details about this problem.

\begin{Example}\label{exm:3x1}
In \cite[\S 11.2]{rink2015coupled} the following \(3\)-cell example is treated:
	\colorlet{co1}{black}
	\colorlet{co2}{orange}
	\colorlet{co3}{violet}
	\colorlet{co4}{blue}
	\colorlet{co5}{green}
\colorlet{c1}{red}
\colorlet{cc1}{red}
\colorlet{c2}{red}
\colorlet{c3}{red}
\ba\label{eq:n3c1}
\nonumber
\Dx{1}{\dot{x}}{1}&=&\Ux{1}{f}(\Dx{1}{x}{1},\Dx{1}{x}{2},\Dx{1}{x}{3}),\\
\nonumber
\Dx{2}{\dot{x}}{2}&=&\Ux{1}{f}(\Dx{2}{x}{1},\Dx{1}{x}{2},\Dx{1}{x}{3}),\\
\nonumber
\Dx{3}{\dot{x}}{3}&=&\Ux{1}{f}(\Dx{3}{x}{1},\Dx{2}{x}{2},\Dx{1}{x}{3}).\\
\ea

Or, equivalently (cf. Notation \ref{not:pic}),
\begin{center}
	\colorlet{c1}{red}
	\colorlet{c2}{red}
	\colorlet{c3}{red}
	\begin{tikzpicture}
	\node (n1) at (0,0) {1};
	\node (n2) at (3,0) {2};
	\node (n3) at (1.5,-2.5) {3};
	\node (z1) at (3,0) {};
	\node (z2) at (6,0) {};
	\node (z3) at (4.5,-2.5) {};
	\foreach \i in {0,...,15} 
		\foreach \x in {1,...,\NN} 
			\foreach \k in {x}	
				\node (n\x\k\i) at ($ (n\x)!0.23!360*\i/16:(z\x) $) {};
	\draw[co2,line width=1pt] (n1x1) to[out=0,in=180] node [sloped] {$>$} (n2x7);
	\draw[co3,line width=1pt] (n1x0) to[out=0,in=180] node [sloped] {$>$} (n2x8);
	\draw[co3,line width=1pt] (n1x13) to[out=300,in=120] node [sloped] {$>$} (n3x6);
	\draw[co2,line width=1pt] (n2x11) to[out=240,in=60] node [sloped] {$<$} (n3x2);
	\foreach \i in 
		{n1x0,n1x1,n1x13,n2x11}
			\fill[c1] (\i) circle [radius=4pt];
	\foreach \i in 
		{n2x7,n2x8,n3x2,n3x6,n1x7,n1x8,n1x9,n3x12,n2x1}
			\fill[c1!\int] (\i) circle [radius=4pt];
	\foreach \i in {n3x12}
		\draw[co1,line width=1pt] (\i) to[loop below] node [sloped] {$$} ();
	\foreach \i in {n2x1}
		\draw[co1,line width=1pt] (\i) to[loop right] node [sloped] {$$} ();
	\foreach \i in {n1x9}
		\draw[co1,line width=1pt] (\i) to[loop left] node [sloped] {$$} ();
	\foreach \i in {n1x8}
		\draw[co2,line width=1pt] (\i) to[loop left] node [sloped] {$$} ();
	\foreach \i in {n1x7}
		\draw[co3,line width=1pt] (\i) to[loop left] node [sloped] {$$} ();
	\foreach \i in {1,...,\NN}{
		\fill[c\i!\int] (n\i) circle [radius=20pt];
		\draw[c\i!\int,line width=1pt]	(n\i) to[out=180,in=0] node [color=black] {$x_\i$}		(n\i);}
\end{tikzpicture}
\end{center}	

\begin{Remark}\label{rem:quotient2}
If we identify cell \(1\) with cell \(2\), we obtain a new network, a quotient network.
It may be an interesting exercise to see how the theory that we develop in this paper behaves under taking quotients.
This identification relies on the specific form of the differential equation.
What we can always do, is to identify all cells with the same color, cf. Remark  \ref{rem:quotient1}.
This leads to what we will later call the \(\sFC{}{}\)-block, in this example \(\bigwedge^1=
\begin{pmatrix} f^1_1+f^1_2+f^1_3\end{pmatrix}\).
\end{Remark}
We let, according to the theory developed in \cite{rink2015coupled}, \(\sigma_1=(1,2,3)\),  \(\sigma_2=(1,1,2)\) and  \(\sigma_3=(1,1,1)\).
One verifies that the \(\sigma_i, i=1,2,3\) define a monoid \(\Sigma\), that is a semigroup with identity,  with identity \(\sigma_1\) and zero \(\sigma_3\).

The Completed Network with monoid \(\hat{\Sigma}=\Sigma\) is described by \(f^1(X_1,X_2,X_3)\) and the linearized system is
\bas
\dot{\mathbf{x}}&=&J\mathbf{x},
\eas
where \(J=\sum_{i=1}^3 f^1_i \nu(\sigma_i)\) with
\bas
\nu(\sigma_1)&=&\begin{pmatrix} 1&0&0\\0&1&0\\0&0&1\end{pmatrix},\quad
\nu(\sigma_2)=\begin{pmatrix} 1&0&0\\1&0&0\\0&1&0\end{pmatrix},\quad
\nu(\sigma_3)=\begin{pmatrix} 1&0&0\\1&0&0\\1&0&0\end{pmatrix}
\eas
and \(f^\cc_i=\frac{\partial f^\cc}{\partial X_i}(0,0,0)\), where \((0,0,0)\) is supposed to be a zero of \(f^\cc\).
We call the Lie algebra of all such matrices (with the usual commutator) \(\mathfrak{net}_{1,3}\), or, for \(N\)-cell networks with \(\CC\) colors,  \(\mathfrak{net}_{\CC,N}\).

The Jacobi-matrix of \eqref{eq:n3c1} is
\bas
J=\begin{pmatrix}
{f^1_3} + {f^1_2} + {f^1_1}& 0& 0\\
{f^1_3} + {f^1_2}&{f^1_1}& 0\\
{f^1_3}&{f^1_2}&{f^1_1}
\end{pmatrix}.
\eas

Using the Chevalley method and assuming \(f^1_2+f^1_3\neq 0\), we compute the above mentioned polynomial \(p_\mathsf{N}(x)\) such that \(p_\mathsf{N}(J)\) is the nilpotent part of its
\(\mathsf{S}+\mathsf{N}\)-decomposition:
\bas
p_\mathsf{N}(x)=-\frac{2 (x-f_1^1)^{3}-(f^1_{2}+f^1_{3})( 3 (x-f^1_1)^{2}  +f^1_{1} (f^1_{2}+f^1_3))}{\left(f^1_{2}+f^1_{3}\right)^{2}}.
\eas

Then the result of this complicated formula is simply \(J=\mathsf{S}+\mathsf{N}\), given by
\bas
{\mathsf{N}}=\begin{pmatrix}
0 & 0 & 0 
\\
 0 & 0 & 0 
\\
 -f^1_{2} & f^1_{2} & 0 \end{pmatrix}
=f^1_2(\nu(\sigma_2)-\nu(\sigma_3)), \quad
\mathsf{S}=\begin{pmatrix}
f^1_{1}+f^1_{2}+f^1_{3} & 0 & 0 
\\
 f^1_{2}+f^1_{3} & f^1_{1} & 0 
\\
 f^1_{2}+f^1_{3} & 0 & f^1_{1} 
\end{pmatrix}=f^1_1\nu(\sigma_1)+(f^1_2+f^1_3)\nu(\sigma_3),
\eas
which also holds if \(f^1_2+f^1_3= 0\).
Sofar the usual theory. 

We now turn to the methods developed in this paper.
The relation between the coordinates (using Algorithm \ref{alg:ff}) is given by
\bas
\begin{pmatrix}
x_1\\
x_3\\
x_2\\
\end{pmatrix}=\begin{pmatrix}
\vv01
\\
\vv11
\\
\vv21
\\
\end{pmatrix}=\begin{pmatrix}
\cu1\\
\bu1
\\
\bu2
\\
\end{pmatrix}.
\eas
%The choice of \(x_1\) as {\em color}-coordinate, can be motivated by the fact that
%the trace of \(\mathsf{J}=3f^1_1+f^1_2+f^1_3\). Subtracting from this \(f^1_1+f^1_2+f^1_3\),
%which we will later see as \(\tr_{\mathfrak{c}} \mathsf{J}\) or, according to \cite{MR4183886}, the multiplier %\(\bigwedge^1_{1,1}\), we are left with \(\tr_{\mathfrak{b}} \mathsf{J} =2f^1_1\).
Using Algorithm \ref{alg:xbasis} we obtain the following Jacobi matrix (in the coordinates \(\cu1, \bu1 \) and \(\bu2\)),
given  in its (obvious) \(\mathsf{S}+\mathsf{N}\)-decomposition:
\bas
\mathsf{J}=\begin{pmatrix}
{f^1_3} + {f^1_2} + {f^1_1}& 0& 0\\
 0&{f^1_1}& 0\\
 0&0&{f^1_1}
\end{pmatrix}+
\begin{pmatrix}
0& 0& 0\\
 0&0& 0\\
 0&{f^1_2}&0
 \end{pmatrix}.
\eas
Using the notation of \cite{MR4183886}, we identify the multipliers 
\bas
\bigwedge^1&=&\begin{pmatrix}{f^1_3} + {f^1_2} + {f^1_1}\end{pmatrix},\\
\bigwedge^{2,3}&=& \begin{pmatrix} {f^1_1}\end{pmatrix}.
\eas
This shows that the method provides us in this case with instant access to the \(\mathsf{S}+\mathsf{N}\)-decomposition
and the multipliers. We did not define 
\bas
\bigwedge^{2}&=& \begin{pmatrix}
 {f^1_1}& 0\\
 {f^1_2}&{f^1_1}
 \end{pmatrix},
\eas
\iffalse
the dual of \(\bigvee^1\), 
\fi
as a multiplier, because if we did that, we would have \(n_2^2=4 > 3 =\CC N\).

It also shows that there is no hope to find an \(\Sl\) within the given Lie algebra. We should emphasize at this point that this is not the end
of a discussion, but rather the start of it. One could, for instance, embed in a reductive Lie algebra and then check whether
the kernel of \(\ad{\mathsf{M}}\) can be seen as a Lie subalgebra of the original network Lie algebra. If that is the case,
then it is normal form business as usual.
%\issue{JS}{Use feedforward method}
\end{Example}
\subsection{$N=3=2\oplus 1$, \protect$\CC=2$}
\def\NN{3}
\begin{Example}\label{exm:3x2}
We consider the following equation, taken from
\cite[\S5.2]{MR4059374}, see also \cite{MR3606590}, but interpreted differently, that is, not fully inhomogeneous, which leads to a slightly restricted case of Example \ref{exm:bas3x2}:
\bas
	\dot{X}&=&f^1(X),\quad X\in\mathbb{R}^2\\
	\dot{y}&=&f^2(X;y)\quad y\in\mathbb{R}
\eas
We write down the general form
\colorlet{co1}{black}
\colorlet{co2}{orange}
\colorlet{co3}{violet}
\colorlet{co4}{blue}
\colorlet{co5}{green}
\colorlet{c1}{red}
\colorlet{cc1}{red}
\colorlet{c2}{red}
\colorlet{c3}{violet}
\colorlet{cc2}{violet}
\bas
	\Dx{1}{\dot{x}}{1}&=&\Ux{1}{f}(\Dx{1}{x}{1},\Dx{1}{x}{2},\Dx{2}{x}{3},\Dx{2}{x}{4}),\\
	\Dx{2}{\dot{x}}{2}&=&\Ux{1}{f}(\Dx{2}{x}{1},\Dx{1}{x}{2},\Dx{1}{x}{3},\Dx{2}{x}{4}),\\
	\Dx{3}{\dot{x}}{3}&=&\Ux{2}{f}(\Dx{1}{x}{1},\Dx{2}{x}{2};\Dx{3}{x}{3}).
\eas
Or, equivalently (cf. Notation \ref{not:pic}),
\begin{center}
	\begin{tikzpicture}
	\node (n1) at (0,0) {1};
	\node (n2) at (3,0) {2};
	\node (n3) at (1.5,-2.5) {3};
	\node (z1) at (3,0) {};
	\node (z2) at (6,0) {};
	\node (z3) at (4.5,-2.5) {};
	\foreach \i in {0,...,15} 
		\foreach \x in {1,...,\NN} 
			\foreach \k in {x}	
				\node (n\x\k\i) at ($ (n\x)!0.23!360*\i/16:(z\x) $) {};
		\draw[co2,line width=1pt] (n1x1) to[out=0,in=180] node [sloped] {$>$} (n2x7);
		\draw[co3,line width=1pt] (n1x14) to[out=0,in=180] node [sloped] {$>$} (n2x10);
		\draw[co1,line width=1.5pt] (n1x13) to[out=300,in=120] node [sloped] {$>$} (n3x6);
		\draw[co3,line width=1pt] (n2x8) to[out=180,in=0] node [sloped] {$<$} (n1x0);
		\draw[co4,line width=1pt] (n2x9) to[out=180,in=0] node [sloped] {$<$} (n1x15);
		\draw[co2,line width=1.5pt] (n2x11) to[out=240,in=60] node [sloped] {$<$} (n3x2);
	\foreach \i in {n1x1,n1x14,n1x13,n2x8,n2x9,n2x11}
		\fill[c1] (\i) circle [radius=4pt];
	\foreach \i in {n2x15,n2x1,n1x0,n1x15,n2x7,n2x10,n3x2,n3x6,n1x9,n1x7}
		\fill[c1!\int] (\i) circle [radius=4pt];
	\foreach \i in {n3x12}
		\fill[c3!\int] (\i) circle [radius=4pt];
	\foreach \i in {n2x1}
		\draw[co4,line width=1pt] (\i) to[loop right] node [sloped] {$$} ();
	\foreach \i in {n2x15}
		\draw[co1,line width=1pt] (\i) to[loop right] node [sloped] {$$} ();
	\foreach \i in {n1x9}
		\draw[co2,line width=1pt] (\i) to[loop left] node [sloped] {$$} ();
	\foreach \i in {n1x7}
		\draw[co1,line width=1pt] (\i) to[loop left] node [sloped] {$$} ();
	\foreach \i in {n3x12}
		\draw[co3,line width=1pt] (\i) to[loop below] node [sloped] {$$} ();
	\foreach \i in {1,...,\NN}{
		\fill[c\i!\int] (n\i) circle [radius=20pt];
		\draw[c\i!\int,line width=1pt]	(n\i) to[out=180,in=0] node [color=black] {$x_\i$}		(n\i);}
	\end{tikzpicture}
\end{center}	
We continue this example, slightly generalized, in Example \ref{exm:bas3x2}.
%  We note that 
% 	in \cite[\S5.2]{MR3606590} the equation is seen as a \(3\)-color system
% 	\colorlet{co1}{black}
% 	\colorlet{co2}{orange}
% 	\colorlet{co3}{violet}
% 	\colorlet{co4}{blue}
% 	\colorlet{co5}{green}
% \colorlet{c1}{red}
% \colorlet{c2}{violet}
% \colorlet{c3}{blue}
% \colorlet{cc1}{red}
% \colorlet{cc2}{violet}
% \colorlet{cc3}{blue}
% \bas
% 	\Dx{1}{\dot{x}}{1}&=&\Ux{1}{f}(\Dx{1}{x}{1},\Dx{1}{x}{2};\Dx{2}{x}{3},\Dx{2}{x}{4}),\\
% 	\Dx{2}{\dot{x}}{2}&=&\Ux{2}{f}(\Dx{2}{x}{1};\Dx{1}{x}{2},\Dx{1}{x}{3};\Dx{2}{x}{4}),\\
% 	\Dx{3}{\dot{x}}{3}&=&\Ux{3}{f}(\Dx{1}{x}{1};\Dx{2}{x}{2};\Dx{3}{x}{3}).
% \eas
\end{Example}
\begin{Remark}
If one has a concrete differential equation, how does one see whether it is a colored network equation?
This question is difficult to answer since there might be many possibilities.
But is it a relevant question?
Probably not, since the network description should be part of the modeling.
	If there is no colored network to begin with (except \(N=\CC\)), there is also no need to stay in the colored network framework.
Despite this, if one can give a nontrivial description of the equation as a colored network equation,
the transformations to be described in this paper can still be used to simplify the Jacobi-matrix without knowledge of the eigenvalues.

%This is comparable to Gau\ss-elimination, but more elementary,
%since no assumptions are made on any of the coefficients.
%The whole process solely depends on the colors.

The situation can be compared to the one in Hamiltonian mechanics: a given equation might be identifiable as a Hamiltonian equation
and one might be able to immediately conclude from this fact, like conservation of energy, even if there is no physical reason why the system would be Hamiltonian, cf. the Volterra-Lotka equation.
\end{Remark}
\section{Structure theory in the colored case}\label{sec:structure}
\begin{Remark}\label{rem:sigmaFun}
The \(\sigma\)-notation only plays a role in the theoretical formulation. 
For the application of the theory to any given example, it is enough to apply the Algorithms \ref{alg:ff} and \ref{alg:xbasis} and Definition \ref{def:xbasis}.
	The Algorithms \ref{alg:ff} and \ref{alg:xbasis} work for general colored networks, they are not required to be closed under taking brackets.
\end{Remark}
While in \cite{rink2015coupled} the approach is to start with a limited given number of \(\sigma\)s and extend this to a semigroup(oid) (the Completed Network),
the approach here is to work from the other side and consider all possible \(\sigma\)s. We then try to reduce the number by going to a matrix representation of the \(\sigma\)s and 
constructing a basis for the representation. Once we have a basis, we construct explicitly a new basis in which the structure of the Lie algebra can be easily determined. 
The transformations that are used can also be used if we work with a limited number of \(\sigma\)s, either in the Generating Network
or the Completed Network; in the latter case one works in a subalgebra of the general network Lie algebra.

We consider differential equations of the following type.
For every color (as in \S \ref{sec:motivexs}) \(\cc\in\{1,\ldots,\CC\}\), we assign a dimension \(n_{\cc}\in \mathbb{N}\),
the number of cells with color \(\cc\) (that is, \(x_i\) with \(\dot{x}_i=f^\cc(x_{\sigma_1(i)},\cdots,x_{\sigma_N(i)}))\) and write \(\bar{n}_\cc=n_\cc-1\) to simplify notation
and define index sets \(R_{\cc}=\{0,\ldots,\bar{n}_{\cc}\}\) and \(\bar{R}_{\cc}=\{1,\ldots,\bar{n}_{\cc}\}\). 
We have coordinates \(\vv{i}{\cc}\) and differential equations
\ba\label{eq:CU}
	\dvv{i}{\cc}&=&f^\cc(\mathsf{U}^{\cc,k_1}_{i,l_1};\mathsf{U}^{\cc,k_2}_{i,l_2};\cdots;\mathsf{U}^{\cc,k_{M_\cc}}_{i,l_{M_\cc}}),\quad i=0,\ldots,\bar{n}_\cc,\quad 1\leq\cc,k_1,\ldots,k_{M_\cc}\leq\CC,
\ea
where each \(\mathsf{U}^{\cc,k_j}_{i,l_j}\) denotes a list of coordinates with color \(k_j\) of length \(l_j\),
fixing the color of every possible input.
So the first \(k_1\) arguments of \(f^\cc\) should have color \(k_1\), that is their differential equation is 
\bas
\dvv{q}{k_1}&=&f^{k_1}(\mathsf{U}^{k_1,k_1}_{q,l_1};\mathsf{U}^{k_1,k_2}_{q,l_2};\cdots;\mathsf{U}^{k_1,k_{M_{k_1}}}_{q,l_{M_{k_1}}}), \quad q\in R_{k_1}.
\eas
These formulas illustrate why some authors prefer to restrict their publications to the case \(\CC=1\).

Let \(\Delta_\cc=\sum_{j=1}^{M_\cc}l_j\) be the {\bf degree} of the color \(\cc\), that is the number of arguments of \(f^\cc\).
In the pictures, this is the total number of incoming arrows; \(\Delta=\sum_{\cc=1}^\CC \Delta_\cc\) is the number of parameters
\(f^\cc_q\) in the Jacobi matrix.

We remark that 
the \(\mathsf{U}\)-notation is only intended to give a general definition of a colored network differential equation.
If a given equation does not satisfy the definition, applying our algorithms may give unpredicted results. We say this because it is very easy
to make mistakes when one changes an existing equation.

Our \(N\)-dimensional space, with \(N=\sum_{\cc=1}^\CC n_\cc\) has then coordinates \(\mathcal{N}={[}\vv{j}{\cc}{]}^{\cc=1,\ldots,\CC}_{ j\in R_\cc}\).
We order this list, for example lexicographically: \(\vv{j}{i}\prec\vv{j^\prime}{i^\prime}\) if \(i<i^\prime\)  or, if \(i=i^\prime, j<j^\prime\).
We now split 
\[\mathcal{N}=\cCC\cup\cBB\] as follows:
\ba\label{eqs:split}
\cCC&=&{[}\vv{0}{1},\vv{0}{2},\cdots,\vv{0}{\CC}{]},\\\nonumber
	\cBB&=&{[}\vv{j}{\cc}{]}^{\cc=1,\ldots,\CC}_{ j\in \bar{R}_\cc}.
\ea
We chose the subindex in the definition of \(\cCC\) to be \(0\) here, because we know beforehand that it is always present;
as a consequence, we usually have sums starting from \(0\) or \(1\).
We identify \(\cCC\) with the list of colors and call the elements in \(\cBB\) {\em cocolors}.
Here, the {\em co} in the {\em cocolors} can be first seen as the {\em co}
of the complement, later we will think of this {\em co} in its usual dual meaning, but this will require some work.
At the moment, we can only give a posteriori explanation for this splitting in colors and cocolors,
which turns out to be very natural and reflects the symmetry induced on the differential equation by having only one \(f^\cc\) per color.

In any given example with coordinates \(x_1,\cdots,x_N\) and a standard dual basis \(x^1,\cdots,x^N\), we have to identify these with \(\vv{j}{\cc}\) and \(\VV{j}{\cc}\). This may involve some rather arbitrary choices,
but if we use Algorithm \ref{alg:ff} much of this arbitrariness will disappear. While the choices are arbitrary in the sense that they do not influence the 
promised \(\FA{}{},\FB{}{},\FC{}{}\)-block form, the algorithm may improve the result. For instance, in Examples \ref{exm:3x1} and \ref{exm:dim5col1exm1a}
it preserves the feedforward
structure, with nice consequences for the application of both the Jordan-Chevalley decomposition and the Jacobson-Morozov theorem.
%\issue{JS}{See whether this still holds true in the final version}

Let \(\BB=N-\CC\) be the cardinality of \(\cBB\).
	The following definition introduces new notation for the coordinates (\(\cu{}{},\bu{}{}\)) and will be used extensively in \S \ref{sec:Lie}, where it simplifies matters considerably, 
	but it turns out that some of the proofs get very complicated in this new notation,
	since the position in the list \(\cBB\) does not give any color information, we will use the old notation (\(\vv{}{}\)) till we get there. 
	Let the \(\vv{i}{k}\) be a basis of \(\mathbb{R}^{n_k}\) and \(\VV{i}{k}\) its standard dual basis, a basis of \(\mathbb{R}^{n_k\star}\), the dual of \(\mathbb{R}^{n_k}\).
	\begin{Definition}\label{def:xynotation}
		We define \(\cu{k}=\vv{0}{k}, \cU{k}=\VV{0}{k}\) and \(\bu{l}=\vv{i}{k}, \bU{l}=\VV{i}{k}, i\in\bar{R}_k\), 
		where \(l\) is the position of \(\vv{i}{k}\) in the list \(\cBB\).
	\end{Definition}

\begin{Example}\label{exm:motivexa}
In our motivating Example \ref{exm:3x1}, this leads to:  \(\cCC=\{\vv{0}{1}\}\) and \(\cBB=\{\vv{1}{1},\vv{2}{1}\}\),
with \(N=3, \CC=1, \BB=2\), \(n_{1}=3\) and \(R_{1}=\{0,1,2\}\), \(\bar{R}_1=\{1,2\}\). 
	We now have to identify \(x_1,x_2,x_3\) with \(\vv{0}{1},\vv{1}{1},\vv{2}{1}\) and \(\cu{1},\bu{1},\bu{2}\).
 \end{Example}
 \begin{Remark}\label{rem:choice}
		We remark here that we make two kinds of choices here.
		First the choice of the \(\vv{0}{1}\) color-coordinate, second the ordering of the cocolor coordinates for each color.
		In this particular case, we can say \(x_1=\vv{0}{1}, x_2=\vv{1}{1}\) and \(x_3=\vv{2}{1}\) and this is an excellent choice.
		But if we make another choice, we might easily lose the upper triangular structure of the Jacobi-matrix, which is determined by the feedforward structure.
\end{Remark}
		
		To this end, we suggest  Algorithm \ref{alg:ff} for \(\cc=1,\ldots,\CC\). In this algorithm, we assign \(\vv{}{\cc}\) coordinates to the \(x_i\) coordinates
and in the meantime redefine the \(\cu{},\bu{}\) coordinates as in Definition \ref{def:xynotation}. 
	The remaining part of this section is devoted to theoretical matters. 
	If one is just interested in putting a particular example in block form, it can be skipped.
	
	\begin{Definition}\label{def:nu}
		Let \(k\) and \(l\) be colors, and \(\Sigma_{k}^{l}\) the set of maps with domain \(R_l\) taking their values from \(R_{k}\).
	We can think of \(\sigma\in\Sigma_{k}^{l}\) as an \(n_l\)-tuple \((k^\sigma_0,\ldots,k^\sigma_{\bar{n}_l})_k^l\), with \(k^\sigma_j\in R_k\) for \( j\in R_l\).
	Then the cardinality  of \(\Sigma_{k}^{l}\) is \(n_{k}^{n_{l}}\). We let \(\tr{\sigma}\) the number of \(j\in R_l\) such that
 \(j=k^\sigma_j\).
		We define  a map \(\nu:\Sigma_{k}^{l}\rightarrow\mathbb{R}^{n_k}\otimes\mathbb{R}^{n_l\star}=\mathrm{Hom}(\mathbb{R}^{n_l},\mathbb{R}^{n_k}) \) as follows. 
	To each \(\sigma\in\Sigma_{k}^{l}\) we assign a \(2\)-tensor 
		\bas
		\nu(\sigma)&=&\sum_{j=0}^{\bar{n}_l}\vv{k^\sigma_j}{k}\otimes\VV{j}{l}\in\mathbb{R}^{n_k}\otimes\mathbb{R}^{n_l\star}\subset\mathfrak{gl}_N,
  \eas
  mapping \(\vv{q}{l}\) to \(\vv{\sigma(q)}{k}\):
  \bas
  \nu(\sigma)\cdot\vv{q}{l}&=&\sum_{j=0}^{\bar{n}_l}\vv{k^\sigma_j}{k}\otimes\VV{j}{l}\cdot\vv{q}{l}=\sum_{j=0}^{\bar{n}_l}
  \vv{k^\sigma_j}{k}\delta^j_q=\vv{k^\sigma_q}{k}=\vv{\sigma(q)}{k}.
  \eas
  We now let \(\mathrm{tr}: V\otimes V^\star\rightarrow \mathbb{R}\) be defined by \(\tr{x\otimes y}=y(x)\). Then for 
  \(\sigma\in\Sigma_k^l\), we find
 \bas \tr{\nu(\sigma)}&=&\sum_{j=0}^{\bar{n}_l}\VV{j}{l}\cdot \vv{k^\sigma_j}{k}=\delta^k_l\sum_{j=0}^{\bar{n}_l} \delta^{j}_{k^\sigma_j}=\delta^k_l\tr{\sigma},
		\eas
		that is, if \(k=l\), then \(\tr{\nu(\sigma)}\) counts the number of self interactions.
	\end{Definition}
	% \begin{Example}\label{exm:2Cexb}
	% 	Let \(n_1=2\) and \(n_2=1\). Then 
 %  \bas
 % \Sigma^1_1&=&\{(0,0)^1_1,(0,1)^1_1,(1,0)^1_1,(1,1)^1_1\}, 
 % \\
 % \Sigma_1^2&=&\{(0)^2_1,(1)^2_1\},
 % \\
 % \Sigma_2^1&=&\{(0,0)^1_2\} 
 % \\
 % \Sigma^2_2&=&\{(0)^2_2\}.
 %  \eas
	% \bas
	% 	\nu((0,0)^1_1)&=&
	% 	\vv{0}{1}\otimes\VV{0}{1}+\vv{0}{1}\otimes\VV{1}{1}
	% 	,\quad
	% 	\nu((0,1)^1_1)=
	% 	\vv{0}{1}\otimes\VV{0}{1}+\vv{1}{1}\otimes\VV{1}{1}
	% 	,\\
	% 	\nu((1,0)^1_1)&=&
	% 	\vv{1}{1}\otimes\VV{0}{1}+\vv{0}{1}\otimes\VV{1}{1}
	% 	,\quad
	% 	\nu((1,1)^1_1)
	% 	=\vv{1}{1}\otimes\VV{0}{1}+\vv{1}{1}\otimes\VV{1}{1}
	% 	,\quad\# \Sigma_2^2=n_1^{n_1}=2^2=4,\\
	% 	\nu((0)_1^2)&=&
	% 	\vv{0}{1}\otimes\VV{0}{2}
	% 	,\quad
	% 	\nu((1)_1^2)=
	% 	\vv{1}{1}\otimes\VV{0}{2}
	% 	,\quad\# \Sigma_1^2=n_1^{n_2}=2^1=2,\\
	% 	\nu((0,0)_2^1)&=&
	% 	\vv{0}{2}\otimes\VV{0}{1}
	% 	+\vv{0}{2}\otimes\VV{1}{1}
	% 	,\quad\#\Sigma_2^1=n_2^{n_1}=1^2=1,\\
	% 	\nu((0)^2_2)&=&
	% 	\vv{0}{2}\otimes\VV{0}{2}
	% 	,\quad\#\Sigma^1_1=1^1=1.
	% \eas
	% 	We continue this example in Example \ref{exm:2Cexc}, where we reduce the number of cases from \(8=4+2+1+1\) to \(7\) by choosing a basis for each \(\Sigma_k^l\).
	% \end{Example}
	 \begin{Lemma}\label{lem:antirep}
		The \(\left\{\Sigma_k^l\right\}_{k,l=1}^\CC\) together form a semigroupoid \(\Sigma\) and \(\nu\) is a faithful antirepresentation of \(\Sigma\) in \(\mathfrak{gl}_N\).
	\end{Lemma}
	\begin{proof}
	Let \(\sigma^l_k\in\Sigma_k^l\) and \(\bar{\sigma}^m_l\in\Sigma_l^m\), so that their composition exists. Then we compute \(\bar{\sigma}^m_l\cdot\sigma^l_k\in\Sigma_k^m\):
	\bas
	\bar{\sigma}\cdot{\sigma}&=&
	(l^{\bar{\sigma}^m_l}_0,\ldots,l^{\bar{\sigma}^m_l}_{\bar{n}_m})
	\cdot
	(k^{\sigma^l_k}_0,\ldots,k^{\sigma^l_k}_{\bar{n}_l})
	\\&=&(k^{\sigma^l_k}_{l^{\bar{\sigma}^m_l}_0},\ldots,k^{\sigma^l_k}_{l^{\bar{\sigma}^m_l}_{\bar{n}_m}}),
	\eas
	that is,
		\(k^{\bar{\sigma}^m_l\cdot{\sigma}^l_k}_i=k^{\sigma^l_k}_{l^{\bar{\sigma}^m_l}_i}=k^{\sigma^l_k}_{\bar{\sigma}^m_l(i)}
  , i=0,\cdots,\bar{n}_m\).
	Then
	\bas
		\nu(\sigma^l_k)\cdot\nu(\bar{\sigma}^m_l)&=&
		\sum_{i_1=0}^{\bar{n}_l}\vv{k^{\sigma^l_k}_{i_1}}{k}\otimes\VV{i_1}{l}
		\cdot
		\sum_{i_2=0}^{\bar{n}_m}\vv{l^{\bar{\sigma}^m_l}_{i_2}}{l}\otimes\VV{i_2}{m}
=\sum_{i_1=0}^{\bar{n}_l}\sum_{i_2=0}^{\bar{n}_m}\delta_{l^{\bar{\sigma}^m_l}_{i_2}}^{i_1}
		\vv{k^{{\sigma}^k_l}_{i_1}}{k}\otimes \VV{i_2}{m}
  \\
&=&
		\sum_{i=0}^{\bar{n}_m}
		\vv{k^{{\sigma}^l_k}_{l^{\bar{\sigma}^m_l}_{i}}}{k}\otimes \VV{i}{m}
		=\sum_{i=0}^{\bar{n}_m}
		\vv{k^{\bar{\sigma}^m_l\cdot\sigma^l_k}_{i}}{k}\otimes \VV{i}{m}
		\\&=&\nu(\bar{\sigma}^m_l\cdot\sigma^l_k).
		\eas
	\end{proof}
\begin{Corollary}
	We define the (Lie) algebra \(\mathfrak{net}_{C,N}\) as follows. It is the \(\mathbb{R}\)-span of the matrices \(\nu(\sigma)\in\mathfrak{gl}_N, \sigma\in\Sigma_k^l, k,l=1,\ldots,\CC\).
	We have seen in Lemma \ref{lem:antirep} that the product of two of these is again in \(\mathfrak{net}_{C,N}\) and therefore \(\mathfrak{net}_{C,N}\) is a
	(Lie) subalgebra of \(\mathfrak{gl}_N\). In the sequel, we restrict our attention to the Lie algebra \(\mathfrak{net}_{C,N}\). 
	We leave it to the reader, if necessary, to derive the structure
	constants for the algebra, following the methods of Section \ref{sec:Lie}.
\end{Corollary}
%\issue{JS}{If \(\mathfrak{net}_{C,N}\) is solvable, is that also true for the nonlinear Lie algebra?}
	\begin{Definition}\label{def:netbasis}
		We fix a basis (cf. Theorem \ref{thm:Basis}) \(\cBB_k^l\) for \(\nu(\Sigma_{k}^{l})\) as follows:
  \[
0 < k, l \leq \mathbb{C}, \quad
\left\{
\begin{array}{lrcl}
j=0, & \F{l,0}{k,0}(\vv{}{}) &=& \nu(([0]_{n_{l}})^{l}_{k}) = \vv{0}{k} \otimes \sum_{q\in R_{l}} \VV{q}{l},\\
&\tr{\F{l,0}{k,0}(\vv{}{})}&=&\delta^k_l,\\
&&
\\
j \in \bar{R}_{k}, & \F{l,i}{k,j}(\vv{}{}) &=& \nu(([0]_{i}, j, [0]_{\bar{n}_{l}-i})^{l}_{k}) =\F{l,0}{k,0}(\vv{}{}) + (\vv{j}{k} - \vv{0}{k}) \otimes \VV{i}{l}, \quad i \in R_{l},\\
&\tr{\F{l,i}{k,j}(\vv{}{})}&=&\delta^l_k(1+\delta^i_j-\delta^i_0),
\end{array}
\right.
\]
	where \([0]_p\) is a row of zeros of length \(p\).
	\end{Definition}
 %\begin{Remark}
 %\end{Remark}

	\begin{Theorem}\label{thm:Basis}
		\(\cBB_k^l\) is a basis, that is, every \(\nu(\sigma), \sigma\in\Sigma_k^l\), can be expressed in terms of \(\cBB_k^l\):
	\bas
			\nu(\sigma)&=& 
			\sum_{i\in Q} \F{l,i}{k,k^\sigma_i}(\vv{}{})+(1-|Q|)\F{l,0}{k,0}(\vv{}{}),
	\eas
		where \(Q\), with cardinality \(|Q|\), is the set of all \(i\in R_l\) such that \(k^\sigma_i>0\), and the elements in \(\cBB_k^l\) are linearly independent. 
  Since if \(0\in Q\) then  \(\delta^0_{k^\sigma_0}=0\) does not contribute to the trace, while if \(0\notin Q\) then \(k^\sigma_0=0 \), and \(\delta^0_{k^\sigma_0}\) contributes \(1\) to the trace, we find
  \bas
  \tr{\nu(\sigma)}&=&\sum_{i\in Q} \tr{\F{l,i}{k,k^\sigma_i}(\vv{}{})}+(1-|Q|)\tr{\F{l,0}{k,0}(\vv{}{})}\\
  &=&\delta^k_l \sum_{i\in Q} (1+\delta^i_{k^\sigma_i}-\delta^i_0)+(1-|Q|)\delta^k_l\\
   &=&\delta^k_l \sum_{i\in Q} (\delta^i_{k^\sigma_i}-\delta^i_0)+\delta^k_l\\
  &=&\left\{ \begin{array}{c} \mbox{ if } 0\in Q \mbox{ then }=-\delta^k_l+\delta^k_l \sum_{i\in Q\setminus {0}} \delta^i_{k^\sigma_i}+\delta^k_l\\ \mbox{ if } 0\notin Q \mbox{ then } k^\sigma_0=0 \mbox{ and } \delta^k_l \sum_{i\in Q} \delta^i_{k^\sigma_i}+\delta^k_l\end{array}\right\}=\delta^k_l\tr{\sigma},
  \eas
	as it should be.
 \end{Theorem}
	\begin{proof}
	Let \(\sigma=(k^\sigma_0,\ldots,k^\sigma_{\bar{n}_l})^k_l\).
		If \(k^\sigma_i>0\),
		\bas
		\vv{k^\sigma_i}{k}\otimes\VV{i}{l}&=&\F{l,i}{k,k^\sigma_i}(\vv{}{})-\F{l,0}{k,0}(\vv{}{})+\vv{0}{k}\otimes\VV{i}{l} .
		\eas
		Let \(Q\) be the set of all \(i\) such that \(k^\sigma_i>0\), with cardinality \(|Q|\), and let \(\complement Q\) be its complement  in \(R_l\). Then
		\bas
			\nu(\sigma)&=& \sum_{i=0}^{\bar{n}_l}
			\vv{k^\sigma_i}{k}\otimes\VV{i}{l}
			\\&=&\sum_{i\in Q} 
			\vv{k^\sigma_i}{k}\otimes\VV{i}{l}
			+\sum_{i\in \complement{Q} }
			\vv{k^\sigma_i}{k}\otimes\VV{i}{l}
			\\&=&\sum_{i\in Q} \left(\F{l,i}{k,k^\sigma_i}(\vv{}{})-\F{l,0}{k,0}(\vv{}{})+\vv{0}{k}\otimes\VV{i}{l} \right)
			+\sum_{i\in \complement{Q}} \vv{0}{k}\otimes\VV{i}{l}
			\\&=&\sum_{i\in Q} \left(\F{l,i}{k,k^\sigma_i}(\vv{}{})-\F{l,0}{k,0}(\vv{}{})\right)
			+\sum_{i\in R_l} \vv{0}{k}\otimes\VV{i}{l}
			\\&=&\sum_{i\in Q} \F{l,i}{k,k^\sigma_i}(\vv{}{})+(1-|Q|)\F{l,0}{k,0}(\vv{}{}).
\eas
	This shows that an arbitrary \(\nu(\sigma)\) can be expressed as a linear combination of the elements in \(\cBB_k^l\).
	The linear independence of these elements is easy to see.
	\end{proof}
	% \begin{Example}\label{exm:2Cexc}
	% 	The basis \(\cBB^2_2\) in Example \ref{exm:2Cexb} is given by 
	% \bas
	% 	\nu((0,0)^1_1)&=&
	% 	\vv{0}{1}\otimes\VV{0}{1}+\vv{0}{1}\otimes\VV{1}{1}
	% \\
	% 	\nu((0,1)^1_1)&=&
	% 	\vv{0}{1}\otimes\VV{0}{1}+\vv{1}{1}\otimes\VV{1}{1}
	% 	,\\
	% 	\nu((1,0)^1_1)&=&
	% 	\vv{1}{1}\otimes\VV{0}{1}+\vv{0}{1}\otimes\VV{1}{1}.
	% \eas
	% 	Together with \(\Sigma_1^2, \Sigma_2^1\) and \(\Sigma_2^2\) this  gives us the basis of all linear network maps in this case.
	% \end{Example}
	\begin{Lemma}\label{lem:dim}
		The dimension of \(\mathfrak{net}_{\CC,N} \) is \(\CC^2+\BB \CC+\BB^2=N^2-\BB\CC\).
	\end{Lemma}
	\begin{proof} For this proof, we look at Definition \ref{def:netbasis}. The term \(1\) corresponds to \(i=j=0\), the \(\bar{n}_{k}\) to the number of \(j\in\bar{R}_k\) and 
		\(n_l\) to the number of \(i\in R_l\):
		\bas
		\dim\mathfrak{net}_{\CC, N}=\sum_{k, l=1}^\CC \left( 1+\bar{n}_{k}n_{l}\right)=
		 \sum_{k=1}^\CC\left( \CC+\bar{n}_{k}N\right)=
		 \CC^2 + (N-\CC)N
		= \BB^2+ \BB \CC+\CC^2.
		\eas
	Notice that the dimension is minimal if \(\CC\approx N/2\), where \(\dim\mathfrak{net}_{\CC, N}\approx \frac{3}{4}\dim\mathfrak{gl}_N\).
 \end{proof}
	\begin{Remark}
		In the case that the number of colors \(\CC\) equals the dimension \(N\) (implying \(n_{i}=1\) for all \(i=1,\ldots,\CC\)), this reduces to 
	\bas
		\F{l,0}{k,0}&=&\vv{0}{l}\otimes\VV{0}{k}=\cu{l}\otimes\cU{k}.
	\eas
		There are \(N^2\) elements and \(\mathfrak{net}_{N,N}=\mathfrak{gl}_N\). 

	This indicates that the case \(\CC=N\) is a bit exceptional, even though it is the case we are most familiar with, that of the ordinary differential equations (or maps).
	This will show up when we compute the Levi decomposition of \(\mathfrak{net}_{\CC,N}\).
	\end{Remark}
  \begin{Example}\label{exm:2Cexa}
Let \(n_{1}=2\) and \(n_{2}=1\), as in Example \ref{exm:3x2} and \cite[Example 12.4]{rink2015coupled}, that is, \(N=3\) and \(\CC=2,\BB=1\), 
with \(\cCC=\{\vv{0}{1},\vv{0}{2}\} \) and \(\cBB=\{\vv{1}{1}\}\). \(R_{1}=\{0,1\}\), \(\bar{R}_1=\{1\}\),  \(R_2=\{0\}\).
	The coordinates are, following Algorithm \ref{alg:ff},
\bas
	x_1=\vv{0}{1}=\cu{1}, x_2=\vv{1}{1}=\bu{1}, x_3=\vv{0}{2}=\cu{2}.
\eas
Since we have
\bas
	\dot{x}_1&=&f^1(x_1,x_1,x_2,x_2),\\
	\dot{x}_2&=&f^1(x_2,x_1,x_1,x_2),\\
	\dot{x}_3&=&f^2(x_1,x_2;x_3),
\eas
	we find that \(\Delta_1=4\) and \(\Delta_2=3\) and, since
each \(\mathsf{U}^{\cc,k_j}_{i,l_j}\) denotes a list of coordinates with color \(k_j\) of length \(l_j\),
	\bas
	\mathsf{U}^{1,1}_{0,4}&=&[x_1,x_1,x_2,x_2]=[\vv{0}{1},\vv{0}{1},\vv{1}{1},\vv{1}{1}]=[\cu{1},\cu{1},\bu{1},\bu{1}],\\
	\mathsf{U}^{1,1}_{1,4}&=&[x_2,x_1,x_1,x_2]=[\vv{1}{1},\vv{0}{1},\vv{0}{1},\vv{1}{1}]=[\bu{1},\cu{1},\cu{1},\bu{1}],\\
	\mathsf{U}^{2,1}_{0,2}&=&[x_1,x_2]=[\vv{0}{1},\vv{1}{1}]=[\cu{1},\bu{1}],\\
	\mathsf{U}^{2,2}_{0,1}&=&[x_3]=[\vv{0}{2}]=[\cu{2}].
	\eas

 Now we apply Definition \ref{def:nu}, we find 
  \bas
 \Sigma^1_1&=&\{(0,0)^1_1,(0,1)^1_1,(1,0)^1_1,(1,1)^1_1\}, 
 \\
 \Sigma_1^2&=&\{(0)^2_1,(1)^2_1\},
 \\
 \Sigma_2^1&=&\{(0,0)^1_2\},
 \\
 \Sigma^2_2&=&\{(0)^2_2\}.
  \eas
	\bas
		\nu((0,0)^1_1)&=&
		\vv{0}{1}\otimes\VV{0}{1}+\vv{0}{1}\otimes\VV{1}{1}
		,\quad
		\nu((0,1)^1_1)=
		\vv{0}{1}\otimes\VV{0}{1}+\vv{1}{1}\otimes\VV{1}{1}
		,\\
		\nu((1,0)^1_1)&=&
		\vv{1}{1}\otimes\VV{0}{1}+\vv{0}{1}\otimes\VV{1}{1}
		,\quad
		\nu((1,1)^1_1)
		=\vv{1}{1}\otimes\VV{0}{1}+\vv{1}{1}\otimes\VV{1}{1}
		,\quad |\Sigma_1^1|=n_1^{n_1}=2^2=4,\\
		\nu((0)_1^2)&=&
		\vv{0}{1}\otimes\VV{0}{2}
		,\quad
		\nu((1)_1^2)=
		\vv{1}{1}\otimes\VV{0}{2}
		,\quad |\Sigma_1^2|=n_1^{n_2}=2^1=2,\\
		\nu((0,0)_2^1)&=&
		\vv{0}{2}\otimes\VV{0}{1}
		+\vv{0}{2}\otimes\VV{1}{1}
		,\quad |\Sigma_2^1|=n_2^{n_1}=1^2=1,\\
		\nu((0)^2_2)&=&
		\vv{0}{2}\otimes\VV{0}{2}
		,\quad \Sigma^2_2|=n_2^{n_2}=1^1=1.
	\eas
	where we reduce the number of cases from \(8=4+2+1+1\) to \(7=\CC^2+\BB\CC+\BB^2\) by choosing a basis for each \(\Sigma_k^l\).
	
	By applying Theorem \ref{thm:Basis}, the basis \(\cBB^2_2\) is given by 
	\bas
		\nu((0,0)^1_1)&=&
		\vv{0}{1}\otimes\VV{0}{1}+\vv{0}{1}\otimes\VV{1}{1},
	\\
		\nu((0,1)^1_1)&=&
		\vv{0}{1}\otimes\VV{0}{1}+\vv{1}{1}\otimes\VV{1}{1}
		,\\
		\nu((1,0)^1_1)&=&
		\vv{1}{1}\otimes\VV{0}{1}+\vv{0}{1}\otimes\VV{1}{1}.
	\eas
		Together with \(\Sigma_1^2, \Sigma_2^1\) and \(\Sigma_2^2\) this gives us the basis of all linear network maps in this case.
  We continue with this example in Section \ref{ex:n3c2}.
  \end{Example}
	\section{ {Subalgebra decomposition of $\mathfrak{net}_{\protect\CC,N}$}}\label{sec:xybasis}
	In this section, we define an invertible linear transformation of \(\mathbb{R}^N\), mapping each colored subspace \(\mathbb{R}^{n_\cc}, 1\leq \cc\leq\CC\) onto itself.
	After this transformation, it is not so difficult to define subalgebras of \(\mathfrak{net}_{\CC,N}\) and to determine their Lie algebraic type.
	We start our discussion with the definition of a Lie bracket on the \(2\)-tensors and corresponding representations.
	In the linear case, this comes down to the usual Lie bracket of square matrices and their action on (co)vectors,
	but the construction also applies to the nonlinear case, cf. Remark \ref{rem:nonlinear}.
	\begin{Definition}\label{def:liebracket}
		Let \(\mathsf{e}_i\) be a basis of \(\mathsf{V}\), and let the \(\mathsf{e}^j\) be the standard dual basis, that is, \(\mathsf{e}^j\cdot\mathsf{e}_i=\delta^j_i\). 
		Let \(\mathsf{x},
	\mathsf{y}\in \mathsf{V}\otimes\mathsf{V}^\star\).
		On the space of two-tensors we can define a product \(\trianglerighteq\) (matrix multiplication) by:
	\bas	\mathsf{e}_{i_1}\otimes\mathsf{e}^{i_2}\trianglerighteq\mathsf{e}_{j_1}\otimes\mathsf{e}^{j_2}
&=&(\mathsf{e}^{i_2}\cdot\mathsf{e}_{j_1})\,\mathsf{e}_{i_1}\otimes\mathsf{e}^{j_2}
	=\delta^{i_2}_{j_1}\mathsf{e}_{i_1}\otimes\mathsf{e}^{j_2}
	\eas
	and 
		\bas
		\mathsf{x}\trianglerighteq\mathsf{y}&=&
		\sum_{i_1,i_2} x^{i_1}_{i_2} \mathsf{e}_{i_1}\otimes\mathsf{e}^{i_2}\trianglerighteq
		\sum_{j_1,j_2}y^{j_1}_{j_2}\mathsf{e}_{j_1}\otimes\mathsf{e}^{j_2}
  \\
		&=&\sum_{i_1,i_2,j_1,j_2}x^{i_1}_{i_2}y^{j_1}_{j_2}(\mathsf{e}^{i_2}\cdot\mathsf{e}_{j_1})\mathsf{e}_{i_1}\otimes\mathsf{e}^{j_2}
		\\&=&\sum_{i_1,i_2,j_1,j_2}x^{i_1}_{i_2}y^{j_1}_{j_2}\delta^{i_2}_{j_1}\mathsf{e}_{i_1}\otimes\mathsf{e}^{j_2}
		=\sum_{i, j}\sum_k x^{i}_{k}y^{k}_{ j}\mathsf{e}_{i}\otimes\mathsf{e}^{j}.
	\eas
 \begin{Corollary}
	\(\tr{ \mathsf{x}\trianglerighteq \mathsf{y}}=\tr{\mathsf{y}\trianglerighteq\mathsf{x}}\).
 \end{Corollary}
	The {\bf associator } (cf. \cite{MR0161898})
		of \(\trianglerighteq\) is a 3-tensor \(\alpha(\mathsf{x},\mathsf{y},\mathsf{z})=\mathsf{x}\trianglerighteq(\mathsf{y}\trianglerighteq\mathsf{z})-(\mathsf{x}\trianglerighteq\mathsf{y})\trianglerighteq\mathsf{z}\).
		This is also written as \([\mathsf{x},\mathsf{y},\mathsf{z}]\) in the literature. 
  The concept was introduced to show that associativity of the product was not a necessary condition to prove the Jacobi identity for Lie algebras.
  The proof of the following Lemma
is left out since it  is a matter
		of writing out the definitions and some trivial mathematics.
	\end{Definition}
	
 \begin{Lemma}\label{lem:assoc_lie}
    The associator of $\trianglerighteq$ is symmetric in its second and third arguments. Additionally, the bracket defined by $[\mathsf{x},\mathsf{y}]=\mathsf{x}\trianglerighteq\mathsf{y} - \mathsf{y}\trianglerighteq\mathsf{x}$ is antisymmetric and,
    by the symmetry of the associator, satisfies the Jacobi identity. In other words, it defines a Lie algebra, which is isomorphic to $\mathfrak{gl}(\mathsf{V})$.
\end{Lemma}
\begin{Corollary}
	\(\tr{ [\mathsf{x},\mathsf{y}]}=0\).
 \end{Corollary}\begin{Remark}\label{rem:nonlinear}
		The use of the associator is a bit of overkill in the linear case, where associativity is clear, but will prove useful when we consider nonlinear vector fields, replacing \(\mathsf{V}\otimes\mathsf{V}^\star\)
		by \(\mathsf{S(V)}\otimes\mathsf{V}^\star\), the space of polynomial vector fields, extending the action of \(\mathsf{V}^\star\) to \(\mathsf{S(V)}\) by derivation. 
	\end{Remark}
	% \begin{Lemma}\label{lem:symmass}
	% 	The associator of \(\trianglerighteq\) is symmetric in its second and third argument (in the nonlinear case this follows from the symmetry of the second derivatives).
	% \end{Lemma}
	% \begin{Lemma}\label{lem:ass}
	% 	\([\mathsf{x},\mathsf{y}]=\mathsf{x}\trianglerighteq\mathsf{y}
	% 	-\mathsf{y}\trianglerighteq\mathsf{x}\) defines an antisymmetric bracket that satisfies the Jacobi identity, in other words: it defines a Lie algebra,
	% 	isomorphic to \(\mathfrak{gl}(\mathsf{V})\).
	% \end{Lemma}
	\begin{Definition}\label{def:xaction}
		We define a representation (cf. Definition \ref{def:liebracket}) of \(\mathsf{V}\otimes\mathsf{V}^\star\) on \(\mathsf{V}\) and \(\mathsf{V}^\star\) as follows:
	\bas
		\left\{ \begin{matrix}
			\mathsf{e}_{i_1}\otimes\mathsf{e}^{i_2}\triangleright\mathsf{e}_j&=&(\mathsf{e}^{i_2}\cdot\mathsf{e}_{j})\ \mathsf{e}_{i_1}=\delta^{i_2}_j\ \mathsf{e}_{i_1},
			\\
			\mathsf{e}_{i_1}\otimes\mathsf{e}^{i_2}\triangleright\mathsf{e}^j&=&-(\mathsf{e}^{j}\cdot\mathsf{e}_{i_1})\ \mathsf{e}^{i_2}=-\delta^j_{i_1}\ \mathsf{e}^{i_2} ,
		\end{matrix}\right.
		\eas
	in other words, a matrix multiplying a (co)vector.
	\end{Definition}
 \begin{Corollary}
     If we now act in the usual algebra fashion with \(\mathsf{e}_{i_1}\otimes\mathsf{e}^{i_2}\) on
     \(\mathsf{e}_{j_1}\otimes\mathsf{e}^{j_2}\), that is,
     \bas
     \mathsf{e}_{i_1}\otimes\mathsf{e}^{i_2}\triangleright(\mathsf{e}_{j_1}\otimes\mathsf{e}^{j_2})&=&
     (\mathsf{e}_{i_1}\otimes\mathsf{e}^{i_2}\triangleright\mathsf{e}_{j_1})\otimes\mathsf{e}^{j_2}
     +\mathsf{e}_{j_1}\otimes(\mathsf{e}_{i_1}\otimes\mathsf{e}^{i_2}\triangleright\mathsf{e}^{j_2})
     \\&=&\delta^{i_2}_{j_1} \mathsf{e}_{i_1}\otimes\mathsf{e}^{j_2}
     -\delta_{i_1}^{j_2}\mathsf{e}_{j_1}\otimes\mathsf{e}^{i_2}
     \\&=&[\mathsf{e}_{i_1}\otimes\mathsf{e}^{i_2},\mathsf{e}_{j_1}\otimes\mathsf{e}^{j_2}],
     \eas
     and we recover the Lie bracket.
 \end{Corollary}
	In Algorithm \ref{alg:xbasis} we formulate the crucial step to obtain the subalgebra decomposition. It is motivated by the choice of the basis in Definition \ref{def:netbasis}.

The exponential formula will be used in the examples in Section \ref{sec:examples}, since it allows us to compute the Jacobi-matrix, given in terms of the \(\vv{}{}\)-coordinates, to one given in \(\xx{}{}\)-coordinates,
and then in \(\cx{},\bx{}\)-coordinates (as defined in Algorithm \ref{alg:ff}), to finally arrive at the \(\FA{}{},\FB{}{},\FC{}{}\) notation to be introduced next.

\begin{Definition}
\label{def:xbasis}
	Let \(k\) and \(l\) be colors and  \(1\leq j\leq \bar{n}_k\) and \( 1\leq i\leq \bar{n}_l\). Then we define,
 after applying Algorithm \ref{alg:xbasis}, the following subalgebras (cf. Lemma \ref{lem:iso}):
 
		\begin{itemize}
			\item The space spanned by the \(\GA{k,j}{l,0}=\xx{j}{k}\otimes\XX{0}{l}\) will be denoted by \(\sFA{\CC}{\BB}\subset\mathfrak{net}_{\CC,N}\).
			\item The space spanned by the \(\GB{k,j}{l,i}=\xx{j}{k}\otimes\XX{i}{l}\) will be denoted by \(\sFB{}{\BB}\subset\mathfrak{net}_{\CC,N}\).
			\item The space spanned by the \(\GC{k,0}{l,0}=\xx{0}{k}\otimes\XX{0}{l}\)  will be denoted by \(\sFC{\CC}{}\subset\mathfrak{net}_{\CC,N}\).
			\item The space spanned by the \(\GO{k,0}{l,i}=\xx{0}{k}\otimes\XX{i}{l}\) will be denoted by \(\sFO{\BB}{\CC}\subset\mathfrak{gl}_{N}\).
		\end{itemize}

  It will turn out that \(\sFA{\CC}{\BB}\) is contained in the solvable part of the Lie algebra \(\mathfrak{net}_{\CC,N}\). The space \(\sFO{\BB}{\CC}\) is the complement of
		\(\mathfrak{net}_{\CC,N}\) in \(\mathfrak{gl}_N\), that is, \(\mathfrak{gl}_N=\mathfrak{net}_{\CC,N}\oplus\sFO{\BB}{\CC}\).
	\begin{Lemma}\label{lem:iso}
		The \(\FA{}{},\FB{}{},\FC{}{}\) form the span of the linear network vector fields:
		\bas
		\F{l,0}{k,0}(\vv{}{})&=&\GC{k,0}{l,0},\\
		\F{l,i}{k,j}(\vv{}{})&=&\GC{k,0}{l,0}+\GB{k,j}{l,i},\\
		\F{l,0}{k,j}(\vv{}{})&=&\GC{k,0}{l,0}+\GA{k,j}{l,0}-\sum_{i=1}^{\bar{n}_k} \GB{k,j}{l,i}.
		\eas
		One checks that this relation is invertible:
\bas
        \GC{k,0}{l,0}&=&\F{l,0}{k,0}(\vv{}{}),\\
		\GB{k,j}{l,i}&=&\F{l,i}{k,j}(\vv{}{})-\F{l,0}{k,0}(\vv{}{}),\\
		\GA{k,j}{l,0}&=&\sum_{i=0}^{\bar{n}_k} (\F{l,i}{k,j}(\vv{}{})-\F{l,0}{k,0}(\vv{}{})).
\eas
  It follows that the traces are consistent with what one would expect from the definitions:
  %\tr{\F{l,i}{k,j}(\vv{}{})}&=&\delta^k_l(1+\delta^i_j-\delta^i_0),
%  		\(\GA{k,j}{l,0}=\xx{j}{k}\otimes\XX{0}{l}\)
%  \(\GB{k,j}{l,i}=\xx{j}{k}\otimes\XX{i}{l}\)
 % \(\GC{k,0}{l,0}=\xx{0}{k}\otimes\XX{0}{l}\) 
 \bas
  	\tr{\GC{k,0}{l,0}}&=&\tr{\F{l,0}{k,0}(\vv{}{})}=\delta^l_k=\XX{0}{l}\cdot\xx{0}{k},\\
		\tr{\GB{k,j}{l,i}}&=&\tr{\F{l,i}{k,j}(\vv{}{})}-\tr{\F{l,0}{k,0}(\vv{}{})}=\delta^k_l\delta^i_j=\XX{i}{l}\cdot\xx{j}{k},\\
		\tr{\GA{k,j}{l,0}}&=&\sum_{i=0}^{\bar{n}_k} (\tr{\F{l,i}{k,j}(\vv{}{})}-\tr{\F{l,0}{k,0}(\vv{}{}))}=
  \sum_{i=0}^{\bar{n}_k}( \delta^l_k(1+\delta^i_j-\delta^i_0)-\delta^l_k)
  \\&=&\delta^l_k\sum_{i=0}^{\bar{n}_k}(\delta^i_j-\delta^i_0)=0=\XX{0}{l}\cdot\xx{j}{k}.
  \eas
  Combined with {Theorem \ref{thm:Basis}} this allows us to express \(\nu(\sigma)\) in terms of 
		\(\GA{}{}, \GB{}{}\) and \(\GC{}{}\) explicitly.
\end{Lemma}
\begin{proof}
		Recall that for \(1\leq k,l\leq\CC\),
	\bas
		\F{l,0}{k,0}(\vv{}{})&=&\vv{0}{k}\otimes\sum_{q=0}^{\bar{n}_l} \VV{q}{l},\\
		\F{l,0}{k,j}(\vv{}{})&=&\F{l}{k,0}(\vv{}{})+(\vv{j}{k}-\vv{0}{k})\otimes\VV{0}{l}, \quad j\in \bar{R}_k,\\
		\F{l,i}{k,j}(\vv{}{})&=&\F{l}{k,0}(\vv{}{})+(\vv{j}{k}-\vv{0}{k})\otimes\VV{i}{l}, \quad i\in \bar{R}_l, j\in \bar{R}_k.
	\eas
		Then, using \(\exp(\mathsf{w}\trianglerighteq)(\mathsf{x}\otimes\mathsf{y})=\exp(\mathsf{w}\triangleright)\mathsf{x}\otimes \exp(\mathsf{w}\triangleright)\mathsf{y}\),
		\bas
		\F{l,0}{k,0}(\vv{}{})&=& 
		\vv{0}{k}\otimes\VV{0}{l}+\sum_{i=1}^{\bar{n}_l} \vv{0}{k}\otimes\VV{i}{l}
		=\xx{0}{k}\otimes(\XX{0}{l}-	\sum_{i=1}^{\bar{n}_l} \XX{i}{l})+\sum_{i=1}^{\bar{n}_l} \xx{0}{k}\otimes\XX{i}{l}
        =\xx{0}{k}\otimes\XX{0}{l}=\GC{k,0}{l,0},
		\\\F{l,i}{k,j}(\vv{}{})&=&
		\F{l,0}{k,0}(\vv{}{})+(\vv{j}{k}-\vv{0}{k})\otimes\VV{i}{l}
		=\GC{k,0}{l,0}+\xx{j}{k}\otimes\XX{i}{l}=\GC{k,0}{l,0}+\GB{k,j}{l,i},\quad i\in \bar{R}_l, j\in \bar{R}_k,
    \\\F{l,0}{k,j}(\vv{}{})&=&
		\F{l}{k,0}(\vv{}{})+(\vv{j}{k}-\vv{0}{k})\otimes\VV{0}{l} 
		=\GC{k,0}{l,0}+\xx{j}{k}\otimes(\XX{0}{l}-\sum_{i=1}^{\bar{n}_l} \XX{i}{l})=\GC{k,0}{l,0}+\GA{k,j}{l,0}-\sum_{i=1}^{\bar{n}_l}\GB{k,j}{l,i} ,\quad j\in \bar{R}_l,
  \eas
  and this proves the Lemma.
	\end{proof}
	\end{Definition}
	%\begin{Remark}
		%The semigroup(oid) approach was taken in \cite{rink2015coupled}  with the explicit goal of working within the network and not %allowing transformations with generators outside the network.
		%Yet the generators \(\vv{0}{k}\otimes 	\sum_{q=1}^{\bar{n}_k}\VV{q}{k}\) are not in the network. Are we cheating here, in the %sense that we ignore
		%our categorical approach here?
		%The subalgebras \(\FA{}{},\FB{}{},\FC{}{}\) that we construct can be expressed in the original network basis by 
  %
%applying the %inverse transformation,
	%and are as such objects within the network. So we are not cheating, but we follow the analogy of solving real differential equations using complex numbers.
	%	We (implicitly) use the fact that the exponential is a Lie algebra homomorphism.
	%\end{Remark}
	\begin{Theorem}\label{thm:abc}
		\(\mathfrak{net}_{\CC,N}=\sFA{\CC}{\BB}\oplus\sFB{}{\BB}\oplus\sFC{\CC}{}\) and \(\mathfrak{gl}_N=\mathfrak{net}_{\CC,N}\oplus\sFO{\BB}{\CC}\)
		{\em {\em (}as vector spaces{\em)}}.
	\end{Theorem}
	\begin{proof}
		See Lemma \ref{lem:iso} and check that the dimension is correct.
	\end{proof}

	%	If there is only one function involved (\(\CC=1\)), one may drop the super index.
	
%	\begin{Theorem}
%	\dcx{1}=f^1(\cx{1}, \cdots, \cx{1}),
	%	\eas
	%	then the Jacobi-matrix of the full system has zeros in the \(\FA{}{}\)-part of the matrix and the coefficient of \(\FC{1}{1}\) %is \(\sum_{j=1}^{\Delta_1} f^1_j\).
	%	In general, if the number of self interaction terms is less than \(\Delta_1\), let \(\iota_1\) be the number of \(\bx{}\)-variables in the equation for \(\cx{1}\).
	%	Then there are at most \(\iota_1\) nonzero terms in \(\FA{}{}\).
	%\end{Theorem}
	%\issue{JS}{This coefficient is always there when \(\CC=1\)}
% \begin{proof}
	%This follows from Algorithm \ref{alg:xbasis}.
	%\end{proof}
	%See Examples \ref{exm:3x1}, \ref{exm:dim3col1exm1}  and \ref{exm:dim5col1exm1} for an illustration of this theorem.
	%In Example \ref{exm:dim8col1exm1} one has \(\iota_1=2\).

	\section{The Lie algebra $\mathfrak{net}_{\protect\CC,N}$}\label{sec:Lie}
	We are almost done now that the notation and definitions are fixed. 
	The problem is reduced to simple matrix calculations to compute the structure constants, observe that they are familiar-looking
	, and then draw our conclusions. 
We will use the indices as defined in Definition \ref{def:xynotation} to simplify writing out the proofs.
Although simple to obtain, the computation of the structure constants will come in handy when we come to the dual pair in \S \ref{sec:dualpair}.

\subsection{The Lie subalgebras $\sFB{\BB}{}$ and $\sFC{\CC}{}$}
	% \begin{Lemma}\label{lem:calgebra}
	% 	Let \(\sFC{\CC}{}\) be the Lie algebra spanned by the \(\FC{\cc_1}{\cc_2}, 0<\cc_i\leq \CC\) as defined in Definition \ref{def:xbasis}, with Lie bracket 
	% 	as defined in Definition \ref{def:liebracket}.
	% 	Then the structure constants are given by 
	% 	\bas 
	% 		[\FC{\cc_1}{\cc_2},\FC{\cc_3}{\cc_4}]=\delta^{\cc_2}_{\cc_3} \FC{\cc_1}{\cc_4}-\delta^{\cc_4}_{\cc_1} \FC{\cc_3}{\cc_2}
	% 	\eas
	% 	and \(\sFC{\CC}{}\simeq\mathfrak{gl}_\CC\).
	% \end{Lemma}
	% \begin{proof}
	% 	We define \(\cx{k}=\xx{0}{k}, \cX{k}=\XX{0}{k}\) and \(\bx{l}=\xx{i}{k}, \bX{l}=\XX{i}{k}, i\in\bar{R}_k\), 
	% 	where \(l\) is the position of \(\xx{i}{k}\) in the list \(\cBB\). Then
 % %We use the notation introduced in Definition \ref{def:xynotation} to simplify writing out the proofs.
	% \bas
	% 	[\FC{\cc_1}{\cc_2},\FC{\cc_3}{\cc_4}]&=&[\cx{\cc_1}\otimes\cX{\cc_2},\cx{\cc_3}\otimes\cX{\cc_4}]
	% 	=\delta^{\cc_2}_{\cc_3} \cx{\cc_1}\otimes\cX{\cc_4} -\delta^{\cc_4}_{\cc_1} \cx{\cc_3}\otimes\cX{\cc_2}
	% 	=\delta^{\cc_2}_{\cc_3} \FC{\cc_1}{\cc_4}-\delta^{\cc_4}_{\cc_1} \FC{\cc_3}{\cc_2}.
	% \eas
	% \end{proof}
	% % \begin{Lemma}\label{lem:Chom}
	% 	Let \(\dCC=\sum_{\cc=1}^{\CC}\FC{\cc}{\cc}\). Then \([\dCC,\FC{\cc_1}{\cc_2}]=0\).
	% \end{Lemma}
	% \begin{proof}
	% 	\(
	% 	[\sum_{\cc=1}^{\CC}\FC{\cc}{\cc},\FC{\cc_1}{\cc_2}]=
	% 	\sum_{\cc=1}^{\CC}( \delta^{\cc}_{\cc_1} \FC{\cc}{\cc_2}-\delta^{\cc_2}_{\cc} \FC{\cc_1}{\cc})
	% 	=\FC{\cc_1}{\cc_2}-\FC{\cc_1}{\cc_2}
	% 	=0
	% 	\).
	% \end{proof}
 \begin{Definition}\label{def:xy}
     We define $\cx{k}=\xx{0}{k}$, $\cX{k}=\XX{0}{k}$, and $\bx{l}=\xx{i}{k}$, $\bX{l}=\XX{i}{k}$, where $i\in\bar{R}_k$ and $l$ is the position of $\xx{i}{k}$ in the list $\cBB$.
 \end{Definition}
 \begin{Notation}\label{not:xy}
		The elements in the space spanned by the \(\GC{k,0}{l,0}=\xx{0}{k}\otimes\XX{0}{l}\) will be denoted by \(\FC{k}{l}=\cx{k}\otimes\cX{l}, 1\leq k,l\leq \CC\).
  %The space spanned by the \(\GC{k,0}{l,0}=\xx{0}{k}\otimes\XX{0}{l}\)  will be denoted by \(\sFC{\CC}{}\subset\mathfrak{net}_{\CC,N}\).
	\end{Notation}
 \begin{Lemma}\label{lem:merged}
        Let $\sFC{\CC}{}$ be the Lie algebra spanned by the $\FC{\cc_1}{\cc_2}$, $0<\cc_i\leq \CC$ as defined in Notation \ref{not:xy}, with Lie bracket as defined in Definition \ref{def:liebracket}. 
         \begin{enumerate}
         \item\label{lem:calgebra}
        Then the structure constants are given by:
        \(
        [\FC{\cc_1}{\cc_2},\FC{\cc_3}{\cc_4}]=\delta^{\cc_2}_{\cc_3} \FC{\cc_1}{\cc_4}-\delta^{\cc_4}_{\cc_1} \FC{\cc_3}{\cc_2}
        \)
        and $\sFC{\CC}{}\simeq\mathfrak{gl}_\CC$.
        \item\label{lem:Chom}
        Let $\dCC=\sum_{\cc=1}^{\CC}\FC{\cc}{\cc}$. Then $[\dCC,\FC{\cc_1}{\cc_2}]=0$.
    \end{enumerate}
    \end{Lemma}
    \begin{proof}We prove
        \begin{itemize}
            \item[\eqref{lem:calgebra}]
            \(
            [\FC{\cc_1}{\cc_2},\FC{\cc_3}{\cc_4}]=[\cx{\cc_1}\otimes\cX{\cc_2},\cx{\cc_3}\otimes\cX{\cc_4}]=\delta^{\cc_2}_{\cc_3} \cx{\cc_1}\otimes\cX{\cc_4} -\delta^{\cc_4}_{\cc_1} \cx{\cc_3}\otimes\cX{\cc_2}=\delta^{\cc_2}_{\cc_3} \FC{\cc_1}{\cc_4}-\delta^{\cc_4}_{\cc_1} \FC{\cc_3}{\cc_2}
            \),
            
            \item[\eqref{lem:Chom}]
            \(
            [\sum_{\cc=1}^{\CC}\FC{\cc}{\cc},\FC{\cc_1}{\cc_2}]=\sum_{\cc=1}^{\CC}( \delta^{\cc}_{\cc_1} \FC{\cc}{\cc_2}-\delta^{\cc_2}_{\cc} \FC{\cc_1}{\cc})=\FC{\cc_1}{\cc_2}-\FC{\cc_1}{\cc_2}=0
            \).
               
    \end{itemize}
    \end{proof}

	\begin{Notation}
		The elements in the space spanned by the \(\GB{l,j}{k,i}=\xx{j}{l}\otimes\XX{i}{k}\) will be denoted by \(\bx{\bb_1}\otimes\bX{\bb_2}=\FB{\bb_1}{\bb_2}, 1\leq \bb_1,\bb_2\leq \BB\).
	\end{Notation}
	\begin{Lemma}\label{lem:Balgebra}
		Let \(\sFB{}{\BB}\) be the Lie subalgebra spanned by the \(\FB{\bb_1}{\bb_2}, 0<\bb_i\leq \BB\) as defined in Definition \ref{def:xbasis}, with Lie bracket 
		as defined in Definition \ref{def:liebracket}.
  \begin{enumerate}
      \item 	Then the structure constants are given by 
			\(
			[\FB{\bb_1}{\bb_2},\FB{\bb_3}{\bb_4}] =\delta^{\bb_2}_{\bb_3}\FB{\bb_1}{\bb_4} -\delta^{\bb_4}_{\bb_1}\FB{\bb_3}{\bb_2},
			\)
		and \(\sFB{}{\BB}\simeq\mathfrak{gl}_\BB\). Furthermore, \([\sFB{}{\BB},\sFC{\CC}{}]=0\).
  \item\label{lem:Bhom} Let \(\dBB=\sum_{\bb=1}^{\BB} \FB{\bb}{\bb}\). 
  Then,
  \([\dBB,\sFB{}{\BB}] =0\).
    \end{enumerate}
	\end{Lemma}
 
	\begin{proof} We prove
 \begin{enumerate}
     
	\item See the proof of Lemma  \ref{lem:merged}, item \ref{lem:calgebra}, replacing \(\cx{}\) by \(\bx{}\).
  \item See the proof of Lemma \ref{lem:merged}, item \ref{lem:Chom}, replacing \(\FC{}{}\) by \(\FB{}{}\).
  \end{enumerate}
  \end{proof}
  
	\begin{Theorem}\label{thm:structure}
We have now shown that:
\begin{enumerate}
\item The Lie subalgebra \(\FC{}{}\subset\mathfrak{net}_{\CC,N}\) is
	isomorphic as a Lie algebra to \(\mathfrak{gl}_\CC(\mathbb{R})\), with \(\dCC\) as the identity matrix. 
		The Lie subalgebra \(\hat{\sFC{\CC}{}}=\sFC{\CC}{}/\langle\dCC\rangle\simeq\mathfrak{sl}_\CC\) will turn out to be part of the semisimple
	component of \(\mathfrak{net}_{\CC,N}\).
	In the case \(\CC=N\), this gives a complete description of the network Lie algebra.
    \item The Lie subalgebra \(\sFB{\BB}{}\) of \(\mathfrak{net}_{\BB,N}\)  is
	isomorphic as a Lie algebra to \(\mathfrak{gl}_{\BB}(\mathbb{R})\), with \(\dBB\) as the identity matrix. 
		The Lie subalgebra \(\hat{\sFB{}{\BB}}=\sFB{}{\BB}/\langle\dBB\rangle\simeq\mathfrak{sl}_{\BB}\) will turn out to be part of the semisimple
	component of \(\mathfrak{net}_{\CC,N}\).
\end{enumerate}  
\end{Theorem}
	\begin{Corollary}\label{cor:comeh}
		It follows from  Corollary \ref{cor:dualxy} that \([\FC{}{},\FB{}{}]=0\); 
			we show in \S \ref{sec:dualpair} that \({\sFC{\CC}{}}\oplus\langle\dBB\rangle\) and \({\sFB{}{\BB}}\oplus\langle\dCC\rangle\) form a dual pair.
	\end{Corollary}

	\subsection{The Lie subalgebra $\sFA{\CC}{\BB}$}
  In this lemma, we establish key properties of subalgebras $\sFA{\CC}{\BB}$, $\sFB{}{\BB}$, and $\sFC{\CC}{}$, including their commutation relations and dimensions. 
	\begin{Notation}
		The elements in the space spanned by the \(\GA{l,j}{k}=\xx{j}{l}\otimes\XX{0}{k}\) will be denoted by 
	\(\bx{\bb_1}\otimes\cX{k}=\FA{\bb_1}{k}, 
		1\leq \bb_1\leq \BB,
		1\leq k\leq \CC\). 
	\end{Notation}

 	\begin{Lemma}
 \label{lem:combined}
		The following statements hold:
		\begin{enumerate}
			\item\label{lem:AA} The subalgebra $\sFA{\CC}{\BB}$ is abelian, that is, $[\sFA{\CC}{\BB},\sFA{\CC}{\BB}]=0$.
			\item \label{rem:BFA} Structure constants are given by $[\FB{\bb_1}{\bb_2},\FA{\bb_3}{\cc_1}]=\delta^{\bb_2}_{\bb_3}\FA{\bb_1}{\cc_1}$. We let $\FA{}{\cc_1}$ be the space spanned by $\FA{\bb_1}{\cc_1}$, then each $\FA{}{\cc_1}$ is a standard representation of $\sFB{}{\BB}$.
			\item \label{rem:CFA} Structure constants are given by $[\FC{\cc_1}{\cc_2},\FA{\bb_1}{\cc_3}]=-\delta^{\cc_3}_{\cc_1}\FA{\bb_1}{\cc_2}$, that is, if we let $\FA{\bb_1}{}$ be the space spanned by $\FA{\bb_1}{\cc_1}$, then $\FA{\bb_1}{}$ is a standard representation of $\sFC{\CC}{}$.
			\item \label{lem:commCA}  $[\dCC,\FA{\bb_1}{\cc_1}]=-\FA{\bb_1}{\cc_1}$.
			\item \label{cor:dima} Its dimension is $\dim\sFA{\CC}{\BB}=\sum_{i=1}^{\CC} \sum_{j=1}^{\BB}1= \BB\CC$.
			\item \label{lem:comba} $[\dBB,\FA{\bb_1}{\cc_1}]=\FA{\bb_1}{\cc_1}$.
			\item Let $\dI=\dCC+\dBB$. Then $[\dI,\mathfrak{net}_{\CC,N}]=0$, that is, $\dI\in\mathcal{Z}(\mathfrak{net}_{\CC,N})$.
		\end{enumerate}
	\end{Lemma}
	
 \begin{proof}
    We prove the following items:
\begin{itemize}
\item[\eqref{rem:BFA}]
  \(
    [\FB{\bb_1}{\bb_2},\FA{\bb_3}{\cc_1}]=[\bx{\bb_1}\otimes \bX{\bb_2} , \bx{\bb_3}\otimes\cX{\cc_1}]=
    (\bX{\bb_2}\cdot\bx{\bb_3}) \bx{\bb_1}\otimes \cX{\cc_1}=\delta^{\bb_2}_{\bb_3}\FA{\bb_1}{\cc_1}
    \),

\item[\eqref{rem:CFA}]
    \(
    [\FC{\cc_1}{\cc_2},\FA{\bb_1}{\cc_3}]=[\cx{\cc_1}\otimes\cX{\cc_2},\bx{\bb_1}\otimes\cX{\cc_3}]=-\delta^{\cc_3}_{\cc_1}\FA{\bb_1}{\cc_2}
    \),
\item[\eqref{lem:commCA}]
    \(
    [\dCC,\FA{\bb_1}{\cc_1}]=\sum_{\cc=1}^{\CC}[\FC{\cc}{\cc},\FA{\bb1}{\cc_1}]
    =-\sum_{\cc=1}^{\CC} \delta^{\cc_1}_{\cc}\FA{\bb_1}{\cc}=-\FA{\bb_1}{\cc_1}
    \),

\item[\eqref{lem:comba}]
    \(
    [\dBB,\FA{\bb_1}{\cc_1}]=\sum_{\bb=1}^{\BB}[\FB{\bb}{\bb},\FA{\bb_1}{\cc_1}]=
    \sum_{\bb=1}^{\BB}\delta^{\bb}_{\bb_1}\FA{\bb}{\cc_1}=\FA{\bb_1}{\cc_1}
    \).
\end{itemize}
\end{proof}

	\subsection{The Lie subalgebra $\sFO{\BB}{\CC}$}
  \begin{Notation}
		The elements in the space spanned by the \(\GO{k,0}{l,i}=\xx{0}{k}\otimes\XX{i}{l}\) will be denoted by 
	\(\bx{k}\otimes\cX{\bb_1}=\FO{k}{\bb_1}, 
		1\leq \bb_1\leq \BB,
		1\leq k\leq \CC\). 
	The space spanned by the \(\FO{k}{\bb_1}\) will be denoted by \(\sFO{\BB}{\CC}\subset\mathfrak{gl}_{N}\).
 \end{Notation}
 
  The following lemma is the dual of   Lemma \ref{lem:combined}.
 \begin{Lemma}
    The following statements hold:
    
    \begin{enumerate}
        \item\label{lem:aa1} The subalgebra \(\sFO{\CC}{\BB}\) is abelian, that is, \([\sFO{\CC}{\BB},\sFO{\CC}{\BB}]=0\).
        
        \item\label{lem:aa2} \( [\FB{\bb_1}{\bb_2},\FO{\cc_1}{\bb_3}]=-\delta^{\bb_3}_{\bb_1}\FO{\cc_1}{\bb_2} \). We let \(\FO{\cc_1}{}\) be the space spanned by \(\FO{\cc_1}{\bb_1}\), then each \(\FO{\cc_1}{}\) is a standard representation of \(\sFB{}{\BB}\).
        
        \item\label{lem:aa3} \([\FC{\cc_1}{\cc_2},\FO{\cc_3}{\bb_1}]=\delta^{\cc_2}_{\cc_3}\FO{\cc_1}{\bb_1}\), that is, if we let \(\FO{}{\bb_1}\) be the space spanned by \(\FO{\cc_1}{\bb_1}\), then \(\FO{}{\bb_1}\) is a standard representation of \(\sFC{\CC}{}\).
        
        \item\label{lem:aa4} \([\FO{\cc_1}{\bb_1},\FA{\bb_2}{\cc_2}]=\delta^{\bb_1}_{\bb_2}\FC{\cc_1}{\cc_2}-\delta^{\cc_2}_{\cc_1}\FB{\bb_2}{\bb_1}\).
    \end{enumerate}
    
    \begin{proof}We prove
    \begin{itemize}
      \item[\eqref{lem:aa2}]

        One has
        \(
        [\FB{\bb_1}{\bb_2},\FO{\cc_1}{\bb_3}]=[\bx{\bb_1}\otimes \bX{\bb_2} , \cx{\cc_1}\otimes\bX{\bb_3}]=
        -(\bX{\bb_3}\cdot\bx{\bb_1}) \cx{\cc_1}\otimes \bX{\bb_2}=-\delta^{\bb_3}_{\bb_1}\cx{\cc_1}\otimes \bX{\bb_2}=-\delta^{\bb_3}_{\bb_1}\FO{\cc_1}{\bb_2}
        \),

       \item[\eqref{lem:aa3}]   
        \(
        [\FC{\cc_1}{\cc_2},\FO{\cc_3}{\bb_1}]=[\cx{\cc_1}\otimes\cX{\cc_2},\cx{\cc_3}\otimes\bX{\bb_1}]=\delta^{\cc_2}_{\cc_3}\FO{\cc_1}{\bb_1}
        \),

         \item[\eqref{lem:aa4}] 
        \(
        [\FO{\cc_1}{\bb_1},\FA{\bb_2}{\cc_2}]=[\cx{\cc_1}\otimes\bX{\bb_1},\bx{\bb_2}\otimes\cX{\cc_2}]=\delta^{\bb_1}_{\bb_2}\cx{\cc_1}\otimes\cX{\cc_2}
        -\delta^{\cc_2}_{\cc_1}\bx{\bb_2}\otimes\bX{\bb_1}=\delta^{\bb_1}_{\bb_2}\FC{\cc_1}{\cc_2}-\delta^{\cc_2}_{\cc_1}\FB{\bb_2}{\bb_1}
        \).
     \end{itemize}
     \end{proof}
\end{Lemma}
	\subsection{The involution $\theta:\mathfrak{net}_{\protect\CC,N}\simeq\mathfrak{net}_{\BB,N}$}\label{sec:BCdual}
 In this section, we give a theoretical result, which is not going to help with 
 our concrete calculations, but will be useful if one wants to classify all 
 the Lie subalgebras of \(\mathfrak{net}_{\,\cdot,N}\), since one may always suppose \(\CC\leq \BB\).
\begin{Theorem}\label{thm:BCdual}
	Let \(\mathfrak{net}_{\CC,N}=\sFA{\CC}{\BB}\oplus\sFB{}{\BB}\oplus\sFC{\CC}{}\) and 
	\(\mathfrak{net}_{\BB,N}=\ssFA{\CC}{\BB}\oplus\ssFB{}{\BB}\oplus\ssFC{\CC}{}\), with \(\dim\ssFC{\CC}{}=\BB^2\), etc.
	Define \bas\theta:\mathfrak{net}_{\CC, N}\rightarrow\mathfrak{net}_{\BB,N}\eas by
	\bas
	\theta \FA{\bb_1}{\cc_1}&=&\FAb{\cc_1}{\bb_1},\\
	\theta \FB{\bb_1}{\bb_2}&=&-\FCb{\bb_2}{\bb_1},\\
	\theta \FC{\cc_1}{\cc_2}&=&-\FBb{\cc_2}{\cc_1},
	\eas
	with \(\cc_i=1,\ldots,\CC\) and \(\bb_i=1,\ldots,\BB\) for \(i=1,2\).
	Then \(\theta\) is an involutative Lie algebra isomorphism, which is to say that \(\theta^2=1\) and \(\theta[p,q]=[\theta p, \theta q]\) for all \(p,q\in \mathfrak{net}_{\CC,N}\).
\end{Theorem}
\begin{proof}
	Observe that \(\theta\) preserves the dimension \(N=\CC+\BB\). Furthermore,
	\bas
		\theta[\FC{\cc_1}{\cc_2},\FC{\cc_3}{\cc_4}]&=&\delta^{\cc_2}_{\cc_3} \theta\FC{\cc_1}{\cc_4}-\delta^{\cc_4}_{\cc_1} \theta\FC{\cc_3}{\cc_2}
		=-\delta^{\cc_2}_{\cc_3} \FBb{\cc_4}{\cc_1}+\delta^{\cc_4}_{\cc_1} \FBb{\cc_2}{\cc_3}
		=[\FBb{\cc_2}{\cc_1},\FBb{\cc_4}{\cc_3}]
		=[\theta\FC{\cc_1}{\cc_2},\theta\FC{\cc_3}{\cc_4}],
		\\\theta[\FB{\bb_1}{\bb_2},\FB{\bb_3}{\bb_4}]&=&\delta^{\bb_2}_{\bb_3} \theta\FB{\bb_1}{\bb_4}-\delta^{\bb_4}_{\bb_1} \theta\FB{\bb_3}{\bb_2}
		=-\delta^{\bb_2}_{\bb_3} \FCb{\bb_4}{\bb_1}+\delta^{\bb_4}_{\bb_1} \FCb{\bb_2}{\bb_3}
		=[\FCb{\bb_2}{\bb_1},\FCb{\bb_4}{\bb_3}]
		=[\theta\FB{\bb_1}{\bb_2}, \theta\FB{\bb_3}{\bb_4}],
		\\\theta[\FC{\cc_1}{\cc_2},\FA{\bb_1}{\cc_3}]&=&-\delta^{\cc_3}_{\cc_1} \theta\FA{\bb_1}{\cc_2}=-\delta^{\cc_3}_{\cc_1} \FAb{\cc_2}{\bb_1}
		=-[\FBb{\cc_2}{\cc_1},\FAb{\cc_3}{\bb_1}]=[\theta\FC{\cc_1}{\cc_2}, \theta\FA{\bb_1}{\cc_3}],
		\\\theta[\FB{\bb_1}{\bb_2},\FA{\bb_3}{\cc_1}]&=&\delta^{\bb_2}_{\bb_3} \theta\FA{\bb_1}{\cc_1}=\delta^{\bb_2}_{\bb_3} \FAb{\cc_1}{\bb_1}
		=-[\FCb{\bb_2}{\bb_1},\FAb{\cc_1}{\bb_3}]=[\theta\FB{\bb_1}{\bb_2}, \theta\FA{\bb_3}{\cc_1}].
	\eas
This takes care of all the nonzero brackets.
\end{proof}
\section{Structure of the Lie algebra \(\mathfrak{net}_{\CC,N}\)}\label{sec:structureLie}
\begin{Theorem}\label{eq:mainth}
	Assume \(0<\CC<N\). Then the Levi decomposition of \(\mathfrak{net}_{\CC,N}\), is given by a solvable part
	\bas
		\begin{bmatrix} \langle\dCC\rangle & 0\\ \sFA{\CC}{\BB} & \langle\dBB\rangle\end{bmatrix}
	\eas
	and the semisimple part
	\bas
	\begin{bmatrix} \FC{}{}/\langle\dCC\rangle& 0 \\ 0& \FB{}{}/\langle\dBB\rangle \end{bmatrix}
	\simeq\begin{bmatrix} \mathfrak{sl}_\CC & 0 \\ 0& \mathfrak{sl}_{\BB}\end{bmatrix}.
	\eas
 The semisimple part acts on the solvable part by
\bas
[\begin{bmatrix} \FC{}{}/\langle\dCC\rangle& 0 \\ 0& \FB{}{}/\langle\dBB\rangle \end{bmatrix},
 \begin{bmatrix} \langle\dCC\rangle & 0\\ \sFA{\CC}{\BB} & \langle\dBB\rangle\end{bmatrix}]
	\subset
 \begin{bmatrix} 0 & 0\\\FA{}{}&0\end{bmatrix}.
	\eas 
	The Lie algebra \(\mathfrak{net}_{\CC,N}\) is {\em reductive} if \(\FA{}{}=0\). 
 The following remark is outside the blanket scope of this paper, as it makes assumptions on the coefficients: If \(\FA{}{}\neq 0\), it can be removed by coordinate transformations under the   following {\em nonresonance} condition:
	the eigenvalues of the \(\sFB{}{}\)- and \( \sFC{}{}\)-block should be different. So this is something that can be checked
 in particular cases with a given organizing center.
 
	The dimension of \(\FA{}{}\) is {\em {\em (}cf.  {\em Corollary }{\em  \ref{cor:dima}} {\em)}}
\bas
	\BB\CC&=&\dim\mathfrak{Gr}_{\CC, N}.
\eas
The dimension of the semisimple part is \(\BB^2+\CC^2-2\) and of the solvable part \(\CC\BB+2\), 
	where the \(2\) is the dimension of the first homology \(H_1(\mathfrak{net}_{\CC,N},\mathfrak{net}_{\CC,N})=\mathfrak{net}_{\CC,N}/[\mathfrak{net}_{\CC,N},\mathfrak{net}_{\CC,N}]\simeq\langle\dBB,\dCC\rangle\), cf. Corollary \ref{cor:hom1}.
We can view an element of \(\mathfrak{net}_{\CC,N}\) as a block matrix:
	\ba\label{eq:mat}
	\begin{bmatrix}
		\begin{pmatrix}
			\cFC{1}{1}&\cdots&\cFC{1}{\CC} \\
			\vdots&\ddots&\vdots\\
			\cFC{\CC}{1}&\cdots&\cFC{\CC}{\CC}
		\end{pmatrix}
		&
		\begin{pmatrix}
			0&\cdots&0\\
			\vdots&\ddots&\vdots\\
			0&\cdots&0
		\end{pmatrix}
	\\
		\begin{pmatrix}
			\cFA{1}{1}&\cdots&\cFA{1}{\CC}\\
			\vdots&\ddots&\vdots\\
			\cFA{\BB}{1}&\cdots&\cFA{\BB}{\CC}
		\end{pmatrix}
		&\begin{pmatrix}
			\cFB{1}{1}&\cdots&\cFB{1}{\BB}\\
			\vdots&\ddots&\vdots\\
			\cFB{\BB}{1}&
			\cdots&
			\cFB{\BB}{\BB}
		\end{pmatrix}
	\end{bmatrix}
		\in\begin{bmatrix}\sFC{\CC}{}&0\\\sFA{\CC}{\BB}&\sFB{\BB}{}\end{bmatrix}
	=\begin{bmatrix}\mathfrak{gl}_\CC&0\\\mathfrak{Gr}_{\CC, N}&\mathfrak{gl}_{\BB}\end{bmatrix}.
	\ea
	There are at most \(\frac{N}{2}+1\) nonisomorphic \(\mathfrak{net}_{\CC,N}\) for any given \(N\in\mathbb{N}\),
	Moreover, \(\min(\CC,\BB)\) and \(N\) fix the isomorphy class.
	Every colored network has a Lie algebra that is isomorphic to a subalgebra of one of these \(\mathfrak{net}_{\CC,N}\).
		The classification of these subalgebras is in general nontrivial, cf. \cite{MR3059183}.
	\end{Theorem}
	\begin{Corollary}\label{cor:hom1}
	The first homology can now be computed:
	\bas
	H_1(\mathfrak{net}_{\CC, N},\mathfrak{net}_{\CC, N})=\langle\dBB,\dCC\rangle.
	\eas
	\end{Corollary}
	\begin{proof}
		The semisimple part is acyclic, and \(\sFA{\CC}{\BB}=[\dBB,\sFA{\CC}{\BB}]\).
	We see that  \(\dim H_1(\mathfrak{net}_{\CC,N},\mathfrak{net}_{\CC,N})=2\), `explaining' this factor in the formula for the dimension of the semisimple and solvable part.
	If \(\CC=N\), this factor would have been \(1\) since then
	\bas
		H_1(\mathfrak{net}_{N, N},\mathfrak{net}_{N, N})=\langle\dCC\rangle=\langle \mathrm{Id}_{\mathbb{R}^N}\rangle.
	\eas
	\end{proof}
 %\issue{Jan}{about Jordan form of the blocks}
 We conclude this section with a method illustration. In this example, we aim to guide readers through the algorithms and main theorem of the paper, offering them a clear understanding of how the results presented in the paper can be applied.

\begin{Example}\label{exm:dim3col1exm1}

    In \cite[\S 11.2]{rink2015coupled} the following \(3\)-cell example with \(\Delta_1=3\) is treated:
\bas
\dot{x}_1&=&f^1({x_1}, {x_1}, {x_1}),\\
\dot{x}_2&=&f^1({x_2}, {x_1}, {x_1}),\\
\dot{x}_3&=&f^1({x_3}, {x_2}, {x_1}).
\eas
 \begin{enumerate}
     \item [Step I]: Identification of the original coordinate \(x_1,\cdots,x_n\) to \(\vv{j}{\cc}\).
     %In fact, two set of colors and cocolors. 
     
     By following Equation \ref{eqs:split}, this leads to 
   $\cCC=\{\vv{0}{1}\}$ and $\cBB=\{\vv{1}{1},\vv{2}{1}\}$, with $N=3$, $\CC=1$, $\BB=2$, $n_{1}=3$, and $R_{1}=\{0,1,2\}$, $\bar{R}_1=\{1,2\}$.

   \item [Step II]: Applying Algorithm \ref{alg:ff} to identification of  \(\vv{}{\cc}\) coordinates to the \(x_i\) coordinates, and the middle notational coordinate  \(\cu{},\bu{}\) as defined in Definition  \ref{def:xynotation}. 

   We need to identify $x_1, x_2, x_3$ with $\vv{0}{1}, \vv{1}{1}, \vv{2}{1}$ and $\cu{1}, \bu{1}, \bu{2}$, by applying  Algorithm \ref{alg:ff}.
   
 %We have \(\cc=1\), \(\mathcal{S}=\emptyset\) and \(\mathcal{N}=\{1,\dots,3\}\).
	We see that \(\dot{x}_1\) has the maximal number of self interactions.
 Using the `quick' option we see that
	%	So \(x_1=\vv{0}{1}=\cu{1}\) and \(\mathcal{S}=\{1\}\) and \(\mathcal{N}=\{2,3\}\). 
	%	Then we put \(q=1\) and compute \(\varpi_2=1\) and \(\varpi_3=2\).
	%	Thus \(x_2=\vv{1}{1}=\bu{1}\)
		% and \(\mathcal{S}=\{1,2\}\) and \(\mathcal{N}=\{3\}\). 
		%Then we put \(q=2\) and 
		%\(x_3=\vv{2}{1}=\bu{2}\). Then 
		% \(\mathcal{S}=\{1,2,3\}\) and \(\mathcal{N}=\emptyset\) and we are done.
So, we have $x_1=\vv{0}{1}=\vv{0}{1}$, $x_2=\vv{1}{1}=\bu{1}$, and $x_3=\vv{2}{1}=\bu{2}$.

Hence, the relation between the coordinates (using Algorithm \ref{alg:ff} and  Definition \ref{def:xynotation}) is given by
\bas
\begin{pmatrix}
x_1\\
x_3\\
x_2\\
\end{pmatrix}=\begin{pmatrix}
\vv01
\\
\vv11
\\
\vv21
\\
\end{pmatrix}=\begin{pmatrix}
\cu1\\
\bu1
\\
\bu2
\\
\end{pmatrix}.
\eas
Since we have
\bas
\dot{x}_1&=&f^1(x_1,x_1,x_1),\\
\dot{x}_2&=&f^1(x_2,x_1,x_1),\\
\dot{x}_3&=&f^1(x_3,x_2,x_1),
\eas
	where \(\Delta_1=3\) and we use the notation given in \eqref{eq:CU}
	\bas
	\mathsf{U}^{1,1}_{0,3}&=&[x_1,x_1,x_1]=[\vv{0}{1},\vv{0}{1},\vv{0}{1}]=[\cu{1},\cu{1},\cu{1}],\\
	\mathsf{U}^{1,1}_{1,3}&=&[x_2,x_1,x_1]=[\vv{1}{1},\vv{0}{1},\vv{0}{1}]=[\bu{1},\cu{1},\cu{1}],\\
	\mathsf{U}^{1,1}_{2,3}&=&[x_3,x_2,x_2]=[\vv{2}{1},\vv{1}{1},\vv{1}{1}]=[\bu{2},\bu{1},\bu{1}].
	\eas

% We return to this in Example \ref{exm:dim3col1exm1}.
  \item [Step III]: Now, by applying Algorithm \ref{alg:xbasis} and Definition \ref{def:xy}, the relationship between the coordinates is given by:
		\ba\label{eqn:3x1}
			\begin {pmatrix}x_1\\x_2\\x_3\\x^1\\x^2\\x^3\end{pmatrix}
   =\begin {pmatrix}\vv{0}{1}\\\vv{1}{1}\\\vv{2}{1}\\\VV{0}{1}\\\VV{1}{1}\\\VV{2}{1}\end{pmatrix}
			&=&	\exp(\xx{0}{1}\otimes 	\sum_{q=1}^{2} \XX{q}{1}\triangleright)
			\begin {pmatrix}\xx{0}{1}\\\xx{1}{1}\\\xx{2}{1}\\\XX{0}{1}\\\XX{1}{1}\\\XX{2}{1}\end{pmatrix}
			=\begin {pmatrix}\xx{0}{1}\\\xx{1}{1}+\xx{0}{1}\\\xx{2}{1}+\xx{0}{1}\\\XX{0}{1}-	(\XX{1}{1}+ \XX{2}{1})\\\XX{1}{1}\\\XX{2}{1}\end{pmatrix}
			=\begin {pmatrix}\cx{1}\\\bx{1}+\cx{1}\\\bx{2}+\cx{1}\\\cX{1}-	(\bX{1}+ \bX{2})\\\bX{1}\\\bX{2}\end{pmatrix}.
		\ea
   \item [Step IV:] 
In this step, we employ the new coordinates in the Jacobian matrix of the system \eqref{eq:n3c1}
\bas
J=\begin{pmatrix}
{f^1_3} + {f^1_2} + {f^1_1}& 0& 0\\
{f^1_3} + {f^1_2}&{f^1_1}& 0\\
{f^1_3}&{f^1_2}&{f^1_1}
\end{pmatrix}.
\eas
Hence, 
\bas
J&=&\left(\left({f^1_3} + {f^1_2} + {f^1_1}\right)x_1\right)\otimes {x^1}+\left(\left({f^1_3} + {f^1_2} \right)x_1+{f^1_1} x_2\right) \otimes {x^2}+\left({f^1_3} x_1+{f^1_2} x_2+{f^1_1} x_3\right)\otimes {x^3}.
\\
&=&
\left(\left({f^1_3} + {f^1_2} + {f^1_1}\right)\cx{1}\right)\otimes (\cX{1}-	(\bX{1}+ \bX{2}))
\\&&
+\left(\left({f^1_3} + {f^1_2} \right)\cx{1}+{f^1_1} (\bx{2}+\cx{1})\right) \otimes \bX{2}
\\&&
+\left({f^1_3} \cx{1}+{f^1_2} (\bx{2}+\cx{1})+{f^1_1} (\bx{1}+\cx{1}))\right)\otimes \bX{1}
\\
&=&
\left({f^1_3} + {f^1_2} + {f^1_1}\right)\cx{1}\otimes  \cX{1}+  {f^1_2} \bX{1} \otimes \bx{2}+  {f^1_1} \bx{2} \otimes \bX{2}+ {f^1_1} \bx{1} \otimes \bX{1}
\\
&=&
\left({f^1_3} + {f^1_2} + {f^1_1}\right){\FC{1}{1}}+ {f^1_1}{\FB{1}{1}}
+{f^1_2}{\FB{2}{1}}
+{f^1_1}{\FB{2}{2}}=\mathsf{J}.
\eas
%One sees that the Lie algebra is semisimple 
The Jacobi-matrix is given by
\bas
\mathsf{J}&=&\begin{pmatrix}
{f^1_3} + {f^1_2} + {f^1_1}& 0& 0\\
 0&{f^1_1}& 0\\
 0&{f^1_2}&{f^1_1}
\end{pmatrix},
\eas
and the Jacobian is
\bas
|\mathsf{J}|&=&({f^1_1})^2
({f^1_3} + {f^1_2} + {f^1_1})
.
\eas
 \end{enumerate}
\end{Example}

	\section{Dual pair}\label{sec:dualpair}
 In this section we formulate a theorem, that has no immediate application to the remaining sections of this paper,
 but it could be of use if one sets out to classify the Lie subalgebras of \(\mathfrak{net}_{\CC,N}\).
 It was motivated by the observation that the dimension is invariant under permuting \(\BB\) and \(\CC\).
 
 %\issue{FM}{it is good to give some motivations for this  part}
	\begin{Definition}\label{def:dualpair}
		A dual pair in a Lie algebra $\mathfrak{g}$ is a pair of subalgebras $\mathfrak{g}_i, i=0,1$, such that $\mathcal{Z}(\mathfrak{g}_{i+1\mod{2}})=\mathfrak{g}_{i}$,  where the irreducible representations are formed by tensor products of irreducible representations of the respective elements in the pair.
	\end{Definition}\
	\begin{Theorem}\label{thm:dualpair}
		The subalgebras \(\langle\dBB\rangle\oplus\sFC{\CC}{}\) and \(\langle\dCC\rangle\oplus\sFB{}{\BB}\) form a dual pair,
		that is, they are the center of one another:
	\bas
		\mathcal{Z}(\langle\dBB\rangle\oplus{\sFC{\CC}{}})=\langle\dCC\rangle\oplus{\sFB{}{\BB}}\eas
		and \bas\mathcal{Z}(\langle\dCC\rangle\oplus{\sFB{}{\BB}})=\langle\dBB\rangle\oplus{\sFC{\CC}{}}.\eas
	\end{Theorem}
	\begin{proof}
		Let \bas
		\mathsf{A}=\sum_{\cc_1,\cc_2=1}^\CC \cFC{\cc_1}{\cc_2}  \FC{\cc_1}{\cc_2}+\sum_{\cc_1=1}^\CC \sum_{\bb_1=1}^{\BB} \cFA{\cc_1}{\bb_1} \FA{\cc_1}{\bb_1}+\sum_{\bb_1,\bb_2=1}^{\BB}\cFB{\bb_1}{\bb_2}\FB{\bb_1}{\bb_2}.
	\eas
		be a generic element in \(\mathfrak{net}_{\CC,N}\),
		and let \(\beta\dBB+\sum_{\cc_1,\cc_2=1}^\CC \cFH{\cc_1}{\cc_2} \FC{\cc_1}{\cc_2}\) be a generic element of \(\langle\dBB\rangle\oplus{\sFC{\CC}{}}\).
		We compute \([{\beta\dBB+\sum_{\cc_1,\cc_2=1}^\CC \cFH{\cc_1}{\cc_2} \FC{\cc_1}{\cc_2}},\mathsf{A}]\):
	\bas
		\lefteqn{[{\beta\dBB+\sum_{\cc_1,\cc_2=1}^\CC \cFH{\cc_1}{\cc_2}  \FC{\cc_1}{\cc_2}},\mathsf{A}]}&&
		\\&=&\sum_{\cc_1,\cc_2=1}^\CC \cFH{\cc_1}{\cc_2}\sum_{{\cc_3},{\cc_4}=1}^\CC \cFC{{\cc_3}}{{\cc_4}} [\FC{\cc_1}{\cc_2},\FC{{\cc_3}}{{\cc_4}}]+\sum_{\cc_1,\cc_2=1}^\CC \cFH{\cc_1}{\cc_2} \sum_{{\cc_3}=1}^\CC \sum_{{\bb_4}=1}^{\BB} \cFA{{\cc_3}}{{\bb_4}}[\FC{\cc_1}{\cc_2}, \FA{{\cc_3}}{{\bb_4}}]+\beta[\dBB,\sum_{{\cc_3}=1}^\CC \sum_{{\bb_4}=1}^{\BB} \cFA{{\cc_3}}{{\bb_4}}\FA{{\cc_3}}{{\bb_4}}]
	\\&=&
		\sum_{\cc_1,\cc_2=1}^\CC \cFH{\cc_1}{\cc_2} \sum_{{\cc_3},{\cc_4}=1}^\CC \cFC{{\cc_3}}{{\cc_4}} (\delta^{\cc_2}_{{\cc_3}} \FC{\cc_1}{{\cc_4}}-\delta^{{\cc_4}}_{\cc_1} \FC{{\cc_3}}{\cc_2})+\sum_{\cc_1,\cc_2=1}^\CC \cFH{\cc_1}{\cc_2} \sum_{{\cc_3}=1}^\CC \sum_{{\bb_4}=1}^{\BB} \cFA{{\cc_3}}{{\bb_4}}\delta^{\cc_2}\FA{\cc_1}{{\bb_4}}+\beta\sum_{{\cc_3}=1}^\CC \sum_{{\bb_4}=1}^{\BB} \cFA{{\cc_3}}{{\bb_4}}\FA{{\cc_3}}{{\bb_4}}
	\\&=&
		\sum_{{\cc_4},\cc_1,{\cc_2}=1}^\CC  \cFH{\cc_1}{{\cc_2}}\cFC{{\cc_2}}{{\cc_4}}  \FC{\cc_1}{{\cc_4}}- \cFC{{\cc_1}}{\cc_2}  \cFH{\cc_2}{{\cc_4}}  \FC{{\cc_1}}{{\cc_4}}+ \sum_{\cc_1,{\cc_2}=1}^\CC\sum_{{\bb_4}=1}^{\BB} \cFH{\cc_1}{{\cc_2}}  \cFA{{\cc_2}}{{\bb_4}}\FA{\cc_1}{{\bb_4}}+\beta\sum_{{\cc_3}=1}^\CC \sum_{{\bb_4}=1}^{\BB} \cFA{{\cc_3}}{{\bb_4}}\FA{{\cc_3}}{{\bb_4}}
	.
	\\&=&
		\sum_{{\cc_1},{\cc_2}=1}^\CC [\cFH{}{},\cFC{}{} ]_{\cc_2}^{\cc_1} \FC{{\cc_1}}{{\cc_2}}
		+ \sum_{\cc_1=1}^\CC\sum_{{\bb_1}=1}^{\BB}\left( (\cFH{}{}+\beta)\cdot \cFA{}{}\right)_{\bb_1}^{\cc_1} \FA{\cc_1}{{\bb_1}}
	.
	\eas
		For \(\mathsf{A}\) to be in \(\mathcal{Z}(\langle\dBB\rangle\oplus\hat{\sFC{\CC}{}})\), this expression should be zero for all \(\beta\dBB+\sum_{\cc_1,{\cc_2}=1}^\CC \cFH{\bb_1}{{\cc_2}} \FC{\bb_1}{{\cc_2}}\). 
		This implies that the \(\cFC{}{}\)-matrix has to commute with all \(\cFH{}{}\)-matrices, that is, it equals a multiple of the identity
		of  \(\sFC{\CC}{}\). In other words, \(\cFC{\cc_1}{\cc_2}=\alpha \delta^{\cc_1}_{\cc_2}\) and
		\bas
		\sum_{\cc_1,\cc_2=1}^\CC \cFC{\cc_1}{\cc_2}\FC{\cc_1}{\cc_2} &=&
		\sum_{\cc_1,\cc_2=1}^\CC\alpha \delta^{\cc_1}_{\cc_2}\FC{\cc_1}{\cc_2} 
		=\sum_{\cc_1=1}^\CC\alpha \FC{\cc_1}{\cc_1} =\alpha \dCC.
		\eas
		%and is therefore \(0\) in \(\hat{\sFC{\CC}{}}\).
			The vectors spanning \(\cFA{}{}\) have to be eigenvectors of all \(\cFH{}{}\)-matrices  for all eigenvalues \(-\beta\).
		This forces the \(\cFA{\cc}{{\bb}}, \cc=1,\ldots,\CC, \bb=1,\ldots,\BB\) to be zero.
	We conclude that \bas
		\mathcal{Z}(\langle\dBB\rangle\oplus{\sFC{\CC}{}})=\langle\dCC\rangle\oplus{\sFB{}{\BB}}.\eas
		The dual statement, in the sense of Theorem \ref{thm:BCdual},
		\bas\mathcal{Z}(\langle\dCC\rangle\oplus{\sFB{}{\BB}})=\langle\dBB\rangle\oplus{\sFC{\CC}{}},\eas has an analogous proof.
	\end{proof}
\section{Algorithms}\label{sec:alg}
\begin{algorithm}
\caption{Creating Hierarchy}\label{alg:ff}
\begin{algorithmic}[1]
\STATE
\COMMENT{This algorithm is not essential, it intends to bring, as far as possible, the representation of elements in \(\sFB{}{\BB}\) in lower diagonal form.
It leaves room for improvement or change of goal}
\STATE Let $q=0$.
\COMMENT{$q$ is the global cocolor counter, the index of \(\bu{}\).}
\FOR{$\cc=1$ \TO $\CC$} 
    \STATE Determine an equation $\dot{x}_{i_\cc}$ with color $\cc$ with the largest number of self-interactions.
    \COMMENT{If $i_\cc=n_\cc$ is a possible choice, make it, this empties the second for loop in the quick option.}
    \STATE Let $ x_{i_\cc}=\vv{0}{\cc}=\cu{\cc}, x^{i_\cc}=\VV{0}{\cc}=\cU{\cc}$.
    \COMMENT{This has the effect that we create zeros in the solvable part $\FA{}{}$, to be defined later.}
    \IF{you want to be quick}
         \FOR{$q_c=1$ \TO $i_\cc-1$} \STATE $q=q+1,x_{q_c}=\vv{q_c}{\cc}=\bu{q},x^{q_c}=\VV{q_c}{\cc}=\bU{q}$. \ENDFOR
         \FOR{$q_c=i_\cc+1$ \TO $n_\cc$} \STATE $q=q+1, x_{q_c}=\vv{q_c-1}{\cc}=\bu{q}, x^{q_c}=\VV{q_c-1}{\cc}=\bU{q}$. \ENDFOR
    \ELSE
        \FOR{$q_c=1$ \TO $n_\cc$} 
            \STATE Count for all $j_\cc\neq i_\cc$ the number of arguments $\varpi_j$ that $\dot{x}_{j_\cc}$ has not in common with $\dot{x}_{i_\cc}$ on any given position.
            \STATE Determine $k$, such that  $\varpi_k$ is minimal (or maximal, your choice may depend on which side of the diagonal you want to make zero as much as possible). 
            \IF{there is more than one $k$ such that $\varpi_k$ is extremal}
                \STATE Choose the $k$ with the highest position of the first $x_{i_\cc}$ among the arguments of $\dot{x}_k$.
            \ENDIF
            \STATE If that does not decide things, make an arbitrary choice of $k$.
            \STATE Let $q=q+1 , x_k=\vv{q_\cc}{\cc}=\bu{q}, x^k=\VV{q_\cc}{\cc}=\bU{q}$.
            \STATE Remove \(\dot{x}_k\) from the list of equations with color \(\cc\) to be checked on the following steps in this loop.
        \ENDFOR
    \ENDIF
\ENDFOR
\end{algorithmic}
\end{algorithm}
 \begin{algorithm}
 \caption{Exponential transformation}\label{alg:xbasis}
 \begin{algorithmic}[1]
	\STATE We now define  for \(1\leq k\leq \CC<N\) a color-preserving map (which preserves the colored network architecture) from \(\mathbb{R}^{n_k}\oplus\mathbb{R}^{\star n_k}\) to itself:
		\bas
			\begin{pmatrix}\xx{0}{k}\\	\xx{1}{k}\\\vdots\\\xx{\bar{n}_k}{k}\\\XX{0}{k}\\\XX{1}{k}\\\vdots\\\XX{\bar{n}_k}{k}\end{pmatrix}&=&
				\exp(-\sum_{\cc=1}^\CC\sum_{q=1}^{\bar{n}_\cc}\vv{0}{\cc}\otimes 	\VV{q}{\cc}\triangleright)
			\begin{pmatrix}\vv{0}{k}\\	\vv{1}{k}\\\vdots\\\vv{\bar{n}_k}{k}\\\VV{0}{k}\\\VV{1}{k}\\\vdots\\\VV{\bar{n}_k}{k}\end{pmatrix}
				=
			\begin{pmatrix}\vv{0}{k}\\\vv{1}{k}-\vv{0}{k}\\\vdots\\\vv{\bar{n}_k}{k}-\vv{0}{k}\\\VV{0}{k}+\sum_{q=1}^{\bar{n}_k}\VV{q}{k}\\\VV{1}{k}\\\vdots\\\VV{\bar{n}_k}{k}\end{pmatrix},
		\eas
		where the \(\VV{\cdot}{k}\) are a standard dual basis to the \(\vv{\cdot}{k}\).
		 \begin{Lemma}\label{lem:dualxyk}
		The \(\XX{\cdot}{k}\) are a standard dual basis to the \(\xx{\cdot}{k}\).
	\end{Lemma}
		\STATE \begin{proof}
	\bas
		\XX{0}{k}\cdot\xx{0}{k}&=&(\VV{0}{k}+\sum_{q=1}^{\bar{n}_k}\VV{q}{k})\cdot\vv{0}{k}=\VV{0}{k}\cdot\vv{0}{k}=1,\\
		\XX{0}{k}\cdot\xx{j}{k}&=&(\VV{0}{k}+\sum_{q=1}^{\bar{n}_k}\VV{q}{k})\cdot(\vv{j}{k}-\vv{0}{k})=\sum_{q=1}^{\bar{n}_k}\VV{q}{k})\cdot\vv{j}{k}-\VV{0}{k}\cdot\vv{0}{k}
		=\sum_{q=1}^{\bar{n}_k}\delta^q_j-1=1-1=0,\quad j=1,\cdots,\bar{n}_k\\
		\XX{i}{k}\cdot\xx{0}{k}&=&\VV{i}{k}\cdot\vv{0}{k}=0,\quad i=1,\cdots,\bar{n}_k\\
		\XX{i}{k}\cdot\xx{j}{k}&=&\VV{i}{k}\cdot(\vv{j}{k}-\vv{0}{k})=\delta^i_j,\quad i,j=1,\cdots,\bar{n}_k
	\eas
	\end{proof}
		\begin{Corollary}\label{cor:dualxy}
		The \(\XX{\cdot}{\cdot}\) are a standard dual basis to the \(\xx{\cdot}{\cdot}\).
	\end{Corollary}
		\STATE The inverse is given by
		\bas
			\begin {pmatrix}\vv{0}{k}\\\vv{1}{k}\\\vdots\\\vv{\bar{n}_k}{k}\\\VV{0}{k}\\\VV{1}{k}\\\vdots\\\VV{\bar{n}_k}{k}\end{pmatrix}
			&=&	\exp(\sum_{\cc=1}^\CC\xx{0}{\cc}\otimes 	\sum_{q=1}^{\bar{n}_\cc} \XX{q}{\cc}\triangleright)
			\begin {pmatrix}\xx{0}{k}\\\xx{1}{k}\\\vdots\\\xx{\bar{n}_k}{k}\\\XX{0}{k}\\\XX{1}{k}\\\vdots\\\XX{\bar{n}_k}{k}\end{pmatrix}
			=\begin {pmatrix}\xx{0}{k}\\\xx{1}{k}+\xx{0}{k}\\\vdots\\\xx{\bar{n}_k}{k}+\xx{0}{k}\\\XX{0}{k}-	\sum_{q=1}^{\bar{n}_k} \XX{q}{k}\\\XX{1}{k}\\\vdots\\\XX{\bar{n}_k}{k}\end{pmatrix}.
		\eas
The order of the definition of the transformation and its inverse is forced by the fact that we know beforehand that the
\(\VV{i}{k}\) form a dual basis to the \(\vv{j}{l}\).
\end{algorithmic}
\end{algorithm}
\begin{Remark}
    In \S \ref{sec:examples}, we refer to Algorithm \ref{alg:ff} all the time, but we remark here that in most cases we follow the quick choice.
%In the feedforward case, it is doing very well, cf. Example \ref{exm:dim5col1exm1}.
		For other types of structure, one may want to develop other algorithms to assign the \(x_i\) to \(\cu{},\bu{}\)-coordinates.
\end{Remark}
\newpage
\section{Examples}\label{sec:examples}
	In this section, we write out the relevant change of coordinates for given \(2\leq N \leq 8\) and \(\CC=1,2\).
	Higher dimensional and more colorful examples would be just as mechanical
	but run into typographical difficulties.  
	Then we introduce a vector field and compute its Jacobi-matrix in terms of \(\cx{},\bx{}\)-coordinates, and
	express the result in terms of \(\FA{}{},\FB{}{}\) and \(\FC{}{}\), illustrating how by simple means a concrete equation can be fit in the theoretical
	description of the network Lie algebra \(\mathfrak{net}_{\CC,N}\).
	This can then be used as a starting point for the bifurcation analysis.

	\subsection{Example with \(N=2,\protect\CC=1\)}
	\def\NN{2}
By applying Algorithm \ref{alg:xbasis} and Definition \ref{def:xy},
	the relation between the coordinates is given by
		\ba\label{eqn:2x1}
			\begin {pmatrix}\vv{0}{1}\\\vv{1}{1}\\\VV{0}{1}\\\VV{1}{1}\end{pmatrix}&=&
			\exp(\xx{0}{1}\otimes 	\XX{1}{1}\triangleright)
			\begin {pmatrix}\xx{0}{1}\\\xx{1}{1}\\\XX{0}{1}\\\XX{1}{1}\end{pmatrix}
			=\begin {pmatrix}\xx{0}{1}\\\xx{1}{1}+\xx{0}{1}\\\XX{0}{1}-	\XX{1}{1}\\\XX{1}{1}\end{pmatrix}
			=\begin {pmatrix}\cx{1}\\\bx{1}+\cx{1}\\\cX{1}-	\bX{1}\\\bX{1}\end{pmatrix}.
		\ea
\begin{Example}\label{exm:dim2col1exm1}
\bas
\dot{x}_1&=&f^1({x_1}, {x_1}, {x_2}, {x_2}),\\
\dot{x}_2&=&f^1({x_2}, {x_1}, {x_1}, {x_2}),\\
\eas
with picture
	\colorlet{co1}{black}
	\colorlet{co2}{orange}
	\colorlet{co3}{violet}
	\colorlet{co4}{blue}
	\colorlet{co5}{green}
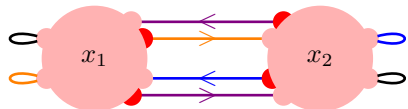
\begin{center}
	\colorlet{co1}{black}
	\colorlet{co2}{orange}
	\colorlet{co3}{violet}
	\colorlet{co4}{blue}
	\colorlet{co5}{green}
	\colorlet{c1}{red}
	\colorlet{c2}{red}
	\begin{tikzpicture}
	\node (n1) at (0,0) {1};
	\node (n2) at (3,0) {2};
	\node (z1) at (3,0) {};
	\node (z2) at (6,0) {};
	\foreach \i in {0,...,15} 
		\foreach \x in {1,...,\NN} 
			\foreach \k in {x}	
				\node (n\x\k\i) at ($ (n\x)!0.23!360*\i/16:(z\x) $) {};
		\draw[co2,line width=1pt] (n1x1) to[out=0,in=180] node [sloped] {$>$} (n2x7);
		\draw[co3,line width=1pt] (n1x14) to[out=0,in=180] node [sloped] {$>$} (n2x10);
		\draw[co3,line width=1pt] (n2x6) to[out=180,in=0] node [sloped] {$<$} (n1x2);
		\draw[co4,line width=1pt] (n2x9) to[out=180,in=0] node [sloped] {$<$} (n1x15);
		\foreach \i in 
			{n1x1,n1x14,n2x6,n2x9}
				\fill[c1] (\i) circle [radius=4pt];
		\foreach \i in 
			{n1x9,n1x7,n2x1,n2x7,n2x10,n1x2,n1x15,n2x15}
				\fill[c1!\int] (\i) circle [radius=4pt];
		\draw[co2,line width=1pt] (n1x9) to[loop left] node [sloped] {$$} ();
		\draw[co1,line width=1pt] (n1x7) to[loop left] node [sloped] {$$} ();
		\draw[co4,line width=1pt] (n2x1) to[loop right] node [sloped] {$$} ();
		\draw[co1,line width=1pt] (n2x15) to[loop right] node [sloped] {$$} ();
	\foreach \i in {1,...,\NN}{
		\fill[c\i!\int] (n\i) circle [radius=20pt];
		\draw[c\i!\int,line width=1pt]	(n\i) to[out=180,in=0] node [color=black] {$x_\i$}		(n\i);}
	\end{tikzpicture}
\end{center}	
The Jacobi-matrix is
\bas
J&=&\begin{pmatrix}
{f^1_2} + {f^1_1}&{f^1_4} + {f^1_3}\\
{f^1_3} + {f^1_2}&{f^1_4} + {f^1_1}
\end{pmatrix}.
\eas
with \(\tr{J}=2f^1_1+f^1_2+f^1_4\).
The relation between the coordinates (using Algorithm \ref{alg:ff}) is given by
\bas
\begin{pmatrix}
x_2\\
x_1\\
\end{pmatrix}=\begin{pmatrix}
\vv01
\\
\vv11
\\
\end{pmatrix}=\begin{pmatrix}
\cu1\\
\bu1
\\
\end{pmatrix}.
\eas
Using Algorithm \ref{alg:xbasis} we obtain the Jacobi matrix
\bas
\mathsf{J}&=&\begin{pmatrix}
{f^1_4} + {f^1_3} + {f^1_2} + {f^1_1}& 0\\
{f^1_3} + {f^1_2}&  {f^1_1}- {f^1_3} 
\end{pmatrix},
\eas
and the multipliers are 
\bas
\bigwedge^1&=&\begin{pmatrix}{f^1_4} + {f^1_3} + {f^1_2} + {f^1_1}\end{pmatrix},\\
\bigwedge^2&=& \begin{pmatrix}{f^1_1}-{f^1_3}\end{pmatrix}.
\eas
The Jacobian is
\bas
|\mathsf{J}|&=&|\bigwedge^1|\cdot|\bigwedge^2|=({f^1_4} + {f^1_3} + {f^1_2} + {f^1_1})
(  {f^1_1}- {f^1_3} ).
\eas
\iffalse
If we let \(\tau={f^1_4}  + {f^1_2} + 2{f^1_1} \), \(\lambda={f^1_4} + {f^1_2} + 2{f^1_3}  \) and \(\alpha={f^1_3} + {f^1_2}\),
then the Jacobi matrix can be written as
\bas
\begin{pmatrix}
\frac{\tau}{2}& 0\\
\alpha&  \frac{\tau}{2}
\end{pmatrix}
+\begin{pmatrix}
\frac{\lambda}{2}& 0\\
0& - \frac{\lambda}{2}
\end{pmatrix},
\eas
where the first matrix is in the solvable part and the second in the semisimple part of the Levi decomposition.
We are not going to carry out this splitting in the following examples, but it should be clear how to do this.
Notice that
\bas
[\begin{pmatrix}
\frac{\tau}{2}& 0\\
\alpha&  \frac{\tau}{2}
\end{pmatrix},\begin{pmatrix}
\frac{\lambda}{2}& 0\\
0& - \frac{\lambda}{2}
\end{pmatrix}]=
\begin{pmatrix}
\frac{\lambda\tau}{4}& 0\\
\frac{\alpha\lambda}{2}& - \frac{\lambda\tau}{4}
\end{pmatrix}-\begin{pmatrix}
\frac{\lambda\tau}{4}& 0\\
-\frac{\alpha\lambda}{2}& - \frac{\lambda\tau}{4}
\end{pmatrix}=\begin{pmatrix}0&0\\\alpha\lambda&0\end{pmatrix},
\eas
illustrating the fact that Levi decomposition is a semidirect product, with the semisimple part acting on the solvable part.
\fi
\end{Example}
\subsection{Example with \(N=3=2\oplus 1,\protect\CC=2\)}\label{ex:n3c2}
	\def\NN{3}
 By applying Algorithm \ref{alg:xbasis} and Definition \ref{def:xy},
	the relation between the coordinates is given by
		\bas
			\begin {pmatrix}\vv{0}{1}\\\vv{1}{1}\\\vv{0}{2}\\\VV{0}{1}\\\VV{1}{1}\\\VV{0}{2}\end{pmatrix}
			&=&
			\exp(\sum_{\cc=1}^2 \xx{0}{\cc}\otimes 	\sum_{q=1}^{\bar{n}_\cc} \XX{q}{\cc}\triangleright)
			\begin {pmatrix}\xx{0}{1}\\\xx{1}{1}\\\xx{0}{2}\\\XX{0}{1}\\\XX{1}{1}\\\XX{0}{2}\end{pmatrix}
			=\begin {pmatrix}\xx{0}{1}\\\xx{1}{1}+\xx{0}{1}\\\xx{0}{2}\\\XX{0}{1}-\XX{1}{1}\\\XX{1}{1}\\\XX{0}{2}\end{pmatrix}
			=\begin {pmatrix}\cx{1}\\\bx{1}+\cx{1}\\\cx{2}\\\cX{1}-\bX{1}\\\bX{1}\\\cX{2}\end{pmatrix}.
		\eas
\begin{Example}\label{exm:bas3x2}
	We write down the {\em basis form}
	\bas
	f&=&f^1(x_1,x_1,x_2,x_2;x_3)x^1+
	f^1(x_2,x_1,x_1,x_2;x_3)x^2+
	f^2(x_1,x_2;x_3)x^3,
	\eas
 short for
 \bas
	\dot{x}_1&=&f^1(x_1,x_1,x_2,x_2;x_3),\\
	\dot{x}_2&=&f^1(x_2,x_1,x_1,x_2;x_3),\\
	\dot{x}_3&=&f^2(x_1,x_2;x_3).
	\eas
Or, equivalently (cf. Notation \ref{not:pic}),
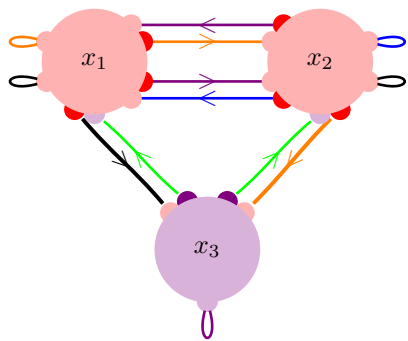
	\begin{center}
	\colorlet{co1}{black}
	\colorlet{co2}{orange}
	\colorlet{co3}{violet}
	\colorlet{co4}{blue}
	\colorlet{co5}{green}
	\colorlet{c1}{red}
	\colorlet{c2}{red}
	\colorlet{c3}{violet}
	\begin{tikzpicture}
	\node (n1) at (0,0) {1};
	\node (n2) at (3,0) {2};
	\node (n3) at (1.5,-2.5) {3};
	\node (z1) at (3,0) {};
	\node (z2) at (6,0) {};
	\node (z3) at (4.5,-2.5) {};
	\foreach \i in {0,...,15} 
		\foreach \x in {1,...,\NN} 
			\foreach \k in {x}	
				\node (n\x\k\i) at ($ (n\x)!0.23!360*\i/16:(z\x) $) {};
		\draw[co2,line width=1pt] (n1x1) to[out=0,in=180] node [sloped] {$>$} (n2x7);
		\draw[co3,line width=1pt] (n1x15) to[out=0,in=180] node [sloped] {$>$} (n2x9);
		\draw[co1,line width=1.5pt] (n1x11) to[out=310,in=130] node [sloped] {$>$} (n3x6);
		\draw[co5,line width=1pt] (n3x5) to[out=140,in=140+180] node [sloped] {$<$} (n1x12);
		\draw[co5,line width=1pt] (n3x3) to[out=40,in=40+180] node [sloped] {$>$} (n2x12);
		\draw[co3,line width=1pt] (n2x6) to[out=180,in=0] node [sloped] {$<$} (n1x2);
		\draw[co4,line width=1pt] (n2x10) to[out=180,in=0] node [sloped] {$<$} (n1x14);
		\draw[co2,line width=1.5pt] (n2x13) to[out=45+180,in=45] node [sloped] {$<$} (n3x2);
	\foreach \i in {n1x1,n1x11,n1x15,n2x6,n2x10,n2x13}
		\fill[c1] (\i) circle [radius=4pt];
	\foreach \i in {n3x3,n3x5}
		\fill[c3] (\i) circle [radius=4pt];
	\foreach \i in {n1x2,n1x9,n1x7,n1x14,n2x1,n2x15,n2x7,n2x9,n3x2,n3x6}
		\fill[c1!\int] (\i) circle [radius=4pt];
	\foreach \i in {n1x12,n2x12,n3x12}
		\fill[c3!\int] (\i) circle [radius=4pt];
	\draw[co3,line width=1pt] (n3x12) to[loop below] node [sloped] {$$} ();
	\draw[co4,line width=1pt] (n2x1) to[loop right] node [sloped] {$$} ();
	\draw[co1,line width=1pt] (n2x15) to[loop right] node [sloped] {$$} ();
	\draw[co2,line width=1pt] (n1x7) to[loop left] node [sloped] {$$} ();
	\draw[co1,line width=1pt] (n1x9) to[loop left] node [sloped] {$$} ();
	\foreach \i in {1,...,\NN}{
		\fill[c\i!\int] (n\i) circle [radius=20pt];
		\draw[c\i!\int,line width=1pt]	(n\i) to[out=180,in=0] node [color=black] {$x_\i$}		(n\i);}
\end{tikzpicture}
\end{center}	

	Using  Algorithm \ref{alg:ff}  we obtain the identification
	\bas
	\begin {pmatrix}x_1\\x_2\\x_3\\x^1\\x^2\\x^3\end{pmatrix}
	=\begin {pmatrix}\vv{0}{1}\\\vv{1}{1}\\\vv{2}{1}\\\VV{0}{1}\\\VV{1}{1}\\\VV{2}{1}\end{pmatrix}
	=\begin {pmatrix}\cu{1}\\\bu{1}\\\cu{2}\\\cU{1}\\\bU{1}\\\cU{2}\end{pmatrix}.
	\eas
	Applying the transformation we obtain the transformed basis form
 \bas
		f&=&f^1(\cx{1},\cx{1},\bx{1}+\cx{1},\bx{1}+\cx{1};\cx{2})\cX{1}-f^1(\bx{1}+\cx{1},\cx{1},\cx{1},\bx{1}+\cx{1};\cx{2})\cX{1}
\\&&+f^1(\bx{1}+\cx{1},\cx{1},\cx{1},\bx{1}+\cx{1};\cx{2})\bX{1}
		+f^2(\cx{1},\bx{1}+\cx{1};\cx{2})\cX{2}.
  \eas
 The differential equation is
	\bas
		\dcx{1}&=&f^1(\cx{1},\cx{1},\bx{1}+\cx{1},\bx{1}+\cx{1};\cx{2})-f^1(\bx{1}+\cx{1},\cx{1},\cx{1},\bx{1}+\cx{1};\cx{2}),\\
		\dbx{1}&=&f^1(\bx{1}+\cx{1},\cx{1},\cx{1},\bx{1}+\cx{1};\cx{2}),\\
		\dcx{2}&=&f^2(\cx{1},\bx{1}+\cx{1};\cx{2}).
	\eas
	The Jacobi-matrix is spanned by the basis of the \(\sigma\)s:
	\bas
	{J}&=&(f^1_1 \vv{0}{1} +f^1_2\vv{0}{1}+(f^1_3+f^1_4)\vv{1}{1}+f^1_5\vv{0}{2})\otimes\VV{0}{1}
	\\&&+(f^1_1 \vv{1}{1} +f^1_2\vv{0}{1}+f^1_3\vv{0}{1}+f^1_4\vv{1}{1}+f^1_5\vv{0}{2})\otimes\VV{1}{1}
	\\&&+(f^2_1 \vv{0}{1} +f^2_2\vv{1}{1}+f^1_3\vv{0}{2})\otimes\VV{0}{2}.
	\eas
	Then
	\bas
	\mathsf{J}&=&(f^1_1 \cx{1} +f^1_2\cx{1}+(f^1_3+f^1_4)(\bx{1}+\cx{1})+f^1_5\cx{2})\otimes(\cX{1}-\bX{1})
	\\&&+(f^1_1 (\bx{1}+\cx{1}) +f^1_2\cx{1}+f^1_3\cx{1}+f^1_4(\bx{1}+\cx{1})+f^1_5\cx{2})\otimes\bX{1}
	\\&&+(f^2_1 \cx{1} +f^2_2(\bx{1}+\cx{1})+f^2_3\cx{2})\otimes\cX{2}
	\\&=&((f^1_1+f^1_2+f^1_3+f^1_4) \cx{1} +(f^1_3+f^1_4)\bx{1}+f^1_5\cx{2})\otimes\cX{1}
	\\&&+(f^1_1 -f^1_3)\bx{1}\otimes\bX{1}
	\\&&+((f^2_1+f^2_2) \cx{1} +f^2_2\bx{1}+f^2_3\cx{2})\otimes\cX{2}
	\\&=&(f^1_1+f^1_2+f^1_3+f^1_4) \FC{1}{1} 
	+f^1_5\FC{2}{1}
	+(f^2_1+f^2_2) \FC{1}{2} 
	+f^2_3\FC{2}{2}
	+(f^1_1 -f^1_3)\FB{1}{1}
	+(f^1_3+f^1_4)\FA{1}{1}
	+f^2_2\FA{1}{2}.
	\eas
		The Jacobi-matrix is
	\bas
		\mathsf{J}&=&\begin{pmatrix}f^1_1+f^1_2+f^1_3+f^1_4 &f^2_1+f^2_2&0\\
			f^1_5&f^2_3&0\\
		f^1_3+f^1_4&f^2_2&f^1_1 -f^1_3
	\end{pmatrix},
	\eas
	and we see that the multipliers are
 \bas
 \bigwedge^1&=&\begin{pmatrix}f^1_1+f^1_2+f^1_3+f^1_4 &f^2_1+f^2_2\\
			f^1_5&f^2_3
		\end{pmatrix},\\
  \bigwedge^2&=&\begin{pmatrix}f^1_1 -f^1_3\end{pmatrix}.
  \eas
 Thus \(n_1=2\) and \(n_2=1\) and \(m_1=m_2=1\) and one has \(m_1n_1+m_2n_2=3=N\) and \(n_1^2+n_2^2=5\leq 6=N\CC\).
 Notice that in the layered feedforward case \cite{Gracht2021}, if the dynamics of color \(1\) do not depend on color \(2\), that is, with the feedback term \(f^1_5=0\), the eigenvalues are on the diagonal.
\end{Example}
	\def\NN{4}
	\subsection{Example with \(N=\NN=2\oplus 2,\protect\CC=2\)}\label{exm:4x2}
	By applying Algorithm \ref{alg:xbasis} and Definition \ref{def:xy},
	the relation between the coordinates is given by
	\bas
		\begin {pmatrix}\vv{0}{1}\\\vv{1}{1}\\\vv{0}{2}\\\vv{1}{2}\\\VV{0}{1}\\\VV{1}{1}\\\VV{0}{2}\\\VV{1}{2}\end{pmatrix}
		&=&	\exp(\sum_{\cc=1}^2\xx{0}{\cc}\otimes 	\XX{1}{\cc}\triangleright)
		\begin {pmatrix}\xx{0}{1}\\\xx{1}{1}\\\xx{0}{2}\\\xx{1}{2}\\\XX{0}{1}\\\XX{1}{1}\\\XX{0}{2}\\\XX{1}{2}\end{pmatrix}
		=\begin {pmatrix}\xx{0}{1}\\\xx{1}{1}+\xx{0}{1}\\\xx{0}{2}\\\xx{1}{2}+\xx{0}{2}\\\XX{0}{1}-\XX{1}{1}\\\XX{1}{1}\\\XX{0}{2}-\XX{1}{2}\\\XX{1}{2}\end{pmatrix}
		=\begin {pmatrix}\cx{1}\\\bx{1}+\cx{1}\\\cx{2}\\\bx{2}+\cx{2}\\\cX{1}-\bX{1}\\\bX{1}\\\cX{2}-\bX{2}\\\bX{2}\end{pmatrix}.
	\eas
	
	\begin{Example}\label{exm:MR4183886}
		Let
		\bas
		\dot{x}_1&=&f^1(x_1,x_2;x_3,x_4),\\
		\dot{x}_2&=&f^1(x_2,x_1;x_4,x_3),\\
		\dot{x}_3&=&f^2(x_3,x_4;x_1,x_2),\\
		\dot{x}_4&=&f^2(x_4,x_3;x_2,x_1),
		\eas
		with \(\Delta_1=\Delta_2=4\) and picture (cf. notation \ref{not:pic})
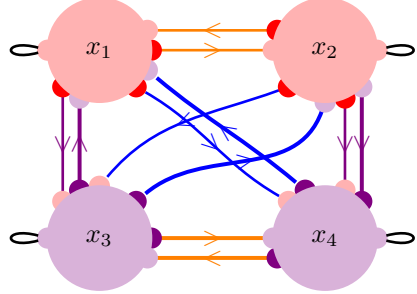
	\begin{center}
	\colorlet{co1}{black}
	\colorlet{co2}{orange}
	\colorlet{co3}{violet}
	\colorlet{co4}{blue}
	\colorlet{co5}{green}
	\colorlet{c1}{red}
	\colorlet{c2}{red}
	\colorlet{c3}{violet}
	\colorlet{c4}{violet}
	\begin{tikzpicture}
	\node (n1) at (0,0) {1};
	\node (n2) at (3,0) {2};
	\node (n3) at (0,-2.5) {3};
	\node (n4) at (3,-2.5) {4};
	\node (z1) at (3,0) {};
	\node (z2) at (6,0) {};
	\node (z3) at (3,-2.5) {};
	\node (z4) at (6,-2.5) {};
	\foreach \i in {0,...,15} 
		\foreach \x in {1,...,\NN} 
			\foreach \k in {x}	
				\node (n\x\k\i) at ($ (n\x)!0.23!360*\i/16:(z\x) $) {};
		 \draw[co2,line width=1pt]	(n1x0) to[out=0,in=180] node [sloped] {$>$}	(n2x8);
		\draw[co3,line width=1pt]	(n1x10) to[out=270,in=90] node [sloped] {$>$}	(n3x6);
		\draw[co4,line width=1pt]	(n1x14) to[out=330,in=150] node [sloped]{$>$}	(n4x6);
		\draw[co2,line width=1pt]	(n2x7) to[out=180,in=0] node [sloped] {$<$}	(n1x1);
		\draw[co4,line width=1pt]	(n2x10) to[out=200,in=50] node [sloped]{$<$}	(n3x4);
		\draw[co3,line width=1pt]	(n2x13) to[out=270,in=90] node [sloped]{$>$}	(n4x3);
		 \draw[co4,line width=1.5pt]	(n3x2) to[out=40,in=250] node [sloped]{$>$}	(n2x12);
		\draw[co3,line width=1.5pt]	(n3x5) to[out=90,in=270] node [sloped]{$>$}	(n1x11);
		\draw[co2,line width=1.5pt]	(n3x0) to[out=0,in=180] node [sloped] {$>$}	(n4x8);
        \draw[co3,line width=1.5pt]	(n4x2) to[out=90,in=270] node [sloped] {$>$}	(n2x14);
		\draw[co4,line width=1.5pt]	(n4x5) to[out=140,in=320] node [sloped] {$<$}	(n1x15);
		\draw[co2,line width=1.5pt]	(n4x9) to[out=180,in=0] node [sloped] {$<$}	(n3x15);
	\foreach \i in 
		{n1x0,n1x10,n1x14,n2x7,n2x10,n2x13}
			\fill[c1] (\i) circle [radius=4pt];
	\foreach \i in 
		{n3x2,n3x5,n3x0,n4x2,n4x5,n4x9}
			\fill[c3] (\i) circle [radius=4pt];
	\foreach \i in 
		{n1x8,n2x0,n1x1,n2x8,n3x4,n3x6,n4x3,n4x6}
			\fill[c1!\int] (\i) circle [radius=4pt];
	\foreach \i in 
		{n4x0,n3x8,n1x11,n1x15,n2x12,n2x14,n3x15,n4x8}
			\fill[c3!\int] (\i) circle [radius=4pt];
	\draw[co1,line width=1pt] (n2x0) to[loop right] node [sloped] {$$} ();
	\draw[co1,line width=1pt] (n4x0) to[loop right] node [sloped] {$$} ();
	\draw[co1,line width=1pt] (n1x8) to[loop left] node [sloped] {$$} ();
	\draw[co1,line width=1pt] (n3x8) to[loop left] node [sloped] {$$} ();
	\foreach \i in {1,...,\NN}{
		\fill[c\i!\int] (n\i) circle [radius=20pt];
		\draw[c\i!\int,line width=1pt]	(n\i) to[out=180,in=0] node [color=black] {$x_\i$}		(n\i);}
\end{tikzpicture}
\end{center}	

	Using  Algorithm \ref{alg:ff}  we obtain the identification
	\bas
	\begin {pmatrix}x_1\\x_2\\x_3\\x_4\\x^1\\x^2\\x^3\\x^4\end{pmatrix}
		=\begin {pmatrix}\vv{0}{1}\\\vv{1}{1}\\\vv{0}{2}\\\vv{1}{2}\\\VV{0}{1}\\\VV{1}{1}\\\VV{0}{2}\\\VV{1}{2}\end{pmatrix}
		=\begin {pmatrix}\cu{1}\\\bu{1}\\\cu{2}\\\bu{2}\\\cU{1}\\\bU{1}\\\cU{2}\\\bU{2}\end{pmatrix}.
	\eas
	Applying the transformation we obtain
	\bas
		\dcx{1}&=&f^1(\cx{1},\bx{1}+\cx{1},\cx{2},\bx{2}+\cx{2})-f^1(\bx{1}+\cx{1},\cx{1},\bx{2}+\cx{2},\cx{2}),\\
		\dbx{1}&=&f^1(\bx{1}+\cx{1},\cx{1},\bx{2}+\cx{2},\cx{2}),\\
		\dcx{2}&=&f^2(\cx{2},\bx{2}+\cx{2},\cx{1},\bx{1}+\cx{1})-f^2(\bx{2}+\cx{2},\cx{2},\bx{1}+\cx{1},\cx{1}),\\
		\dbx{2}&=&f^2(\bx{2}+\cx{2},\cx{2},\bx{1}+\cx{1},\cx{1}).
	\eas
	The Jacobi-matrix in tensor notation is
	\bas
		{J}&=&(f^1_1\vv{0}{1}+f^1_2\vv{1}{1}+f^1_3\vv{0}{2}+f^1_4\vv{1}{2})\otimes\VV{0}{1}
		+(f^1_1\vv{1}{1}+f^1_2\vv{0}{1}+f^1_3\vv{1}{2}+f^1_4\vv{0}{2})\otimes\VV{1}{1}
		\\&&+(f^2_1\vv{0}{2}+f^2_2\vv{1}{2}+f^2_3\vv{0}{1}+f^2_4\vv{1}{1})\otimes\VV{0}{2}
		+(f^2_1\vv{1}{2}+f^2_2\vv{0}{2}+f^2_3\vv{1}{1}+f^2_4\vv{0}{1})\otimes\VV{1}{2}
		\\&=&(f^1_1\cx{1}+f^1_2(\bx{1}+\cx{1})+f^1_3\cx{2}+f^1_4(\bx{2}+\cx{2}))\otimes\cX{1}
		+(f^1_1\bx{1}-f^1_2\bx{1}+f^1_3\bx{2}-f^1_4\bx{2})\otimes\bX{1}
		\\&&+(f^2_1\cx{2}+f^2_2(\bx{2}+\cx{2})+f^2_3\cx{1}+f^2_4(\bx{1}+\cx{1}))\otimes\cX{2}
		+(f^2_1\bx{2}-f^2_2\bx{2}+f^2_3\bx{1}-f^2_4\bx{1})\otimes\bX{2}
		\\&=&f^1_1\FC{1}{1}+f^1_2(\FA{1}{1}+\FC{1}{1})+f^1_3\FC{2}{1}+f^1_4(\FA{2}{1}+\FC{2}{1})
		+f^1_1\FB{1}{1}-f^1_2\FB{1}{1}+f^1_3\FB{2}{1}-f^1_4\FB{2}{1}
		\\&&+(f^2_1\FC{2}{2}+f^2_2(\FA{2}{2}+\FC{2}{2})+f^2_3\FC{1}{2}+f^2_4(\FA{1}{2}+\FC{1}{2})
		+f^2_1\FB{2}{2}-f^2_2\FB{2}{2}+f^2_3\FB{1}{2}-f^2_4\FB{1}{2}
		\\&=&
		(f^1_1 +f^1_2)\FC{1}{1}
		+(f^2_3 +f^2_4)\FC{1}{2}
		+(f^1_3 +f^1_4)\FC{2}{1}
		+(f^2_1 +f^2_2)\FC{2}{2}
		\\&&+(f^1_1 -f^1_2)\FB{1}{1}
		+(f^2_3 -f^2_4)\FB{1}{2}
		+(f^1_3 -f^1_4)\FB{2}{1}
		+(f^2_1 -f^2_2)\FB{2}{2}
		\\&&+f^1_2\FA{1}{1}
		+f^2_4\FA{1}{2}
		+f^1_4\FA{2}{1}
		+f^2_2\FA{2}{2}=\mathsf{J}.
		\eas
		Then the matrix is
		\bas
		\mathsf{J}&=&\begin{pmatrix}
			f^1_1 +f^1_2&f^2_3 +f^2_4&0&0\\
			f^1_3 +f^1_4&f^2_1 +f^2_2&0&0\\
			f^1_2&f^2_4&f^1_1 -f^1_2&f^2_3 -f^2_4\\
			f^1_4&f^2_2&f^1_3 -f^1_4&f^2_1 -f^2_2
		\end{pmatrix}.
		\eas
	%Created by Form4.3 on Wed May  8 11:22:59 2024
The multipliers are
\bas
\bigwedge^1&=&\begin{pmatrix}
			f^1_1 +f^1_2&f^2_3 +f^2_4\\
			f^1_3 +f^1_4&f^2_1 +f^2_2
  \end{pmatrix},\\
  \bigwedge^2&=&\begin{pmatrix}
  f^1_1 -f^1_2&f^2_3 -f^2_4\\
			f^1_3 -f^1_4&f^2_1 -f^2_2
		\end{pmatrix}.
  \eas

Thus \(n_1=n_2=2\) and \(m_1=m_2=1\) and one has \(m_1n_1+m_2n_2=4=N\) and \(n_1^2+n_2^2=8=\CC N\).
The Jacobian is
\bas
|{\mathsf{J}}|&=&|{\bigwedge^1}| \cdot |{\bigwedge^2}|
\eas
and \(\tr{\mathsf{J}}=\tr{\bigwedge^1}+\tr{\bigwedge^2}=2f^1_1+2f^2_1\).
\end{Example}
\begin{Example}\label{exm:dim4col2exm4}
%Created by Form4.3 on Thu Jun 20 10:50:28 2024
\bas
\dot{x}_{1}&=&{f^{{1}}}({x_{1}},{x_{2}},{x_{3}},{x_{4}}),\\
\dot{x}_{2}&=&{f^{{1}}}({x_{2}},{x_{1}},{x_{4}},{x_{3}}),\\
\dot{x}_{3}&=&{f^{{2}}}({x_{3}},{x_{4}}),\\
\dot{x}_{4}&=&{f^{{2}}}({x_{4}},{x_{3}}).\\
\eas
The Jacobi-matrix is
\bas
J&=&\begin{pmatrix}
{f^{{1}}_{{1}}}&{f^{{1}}_{{2}}}&{f^{{1}}_{{3}}}&{f^{{1}}_{{4}}}\\
{f^{{1}}_{{2}}}&{f^{{1}}_{{1}}}&{f^{{1}}_{{4}}}&{f^{{1}}_{{3}}}\\
 0& 0&{f^{{2}}_{{1}}}&{f^{{2}}_{{2}}}\\
 0& 0&{f^{{2}}_{{2}}}&{f^{{2}}_{{1}}}
\end{pmatrix}.
\eas
Trace =\(2{f^{{1}}_{{1}}} + 2{f^{{2}}_{{1}}}\).

	\begin{center}
	\colorlet{co1}{black}
	\colorlet{co2}{orange}
	\colorlet{co3}{violet}
	\colorlet{co4}{blue}
	\colorlet{co5}{green}
	\colorlet{c1}{red}
	\colorlet{c2}{red}
	\colorlet{c3}{violet}
	\colorlet{c4}{violet}
	\begin{tikzpicture}
	\node (n1) at (0,0) {1};
	\node (n2) at (3,0) {2};
	\node (n3) at (0,-2.5) {3};
	\node (n4) at (3,-2.5) {4};
	\node (z1) at (3,0) {};
	\node (z2) at (6,0) {};
	\node (z3) at (3,-2.5) {};
	\node (z4) at (6,-2.5) {};
	\foreach \i in {0,...,15} 
		\foreach \x in {1,...,\NN} 
			\foreach \k in {x}	
				\node (n\x\k\i) at ($ (n\x)!0.23!360*\i/16:(z\x) $) {};
		 \draw[co2,line width=1pt]	(n1x0) to[out=0,in=180] node [sloped] {$>$}	(n2x8);
		\draw[co3,line width=1pt]	(n1x10) to[out=270,in=90] node [sloped] {$<$}	(n3x6);
		\draw[co4,line width=1pt]	(n1x14) to[out=330,in=150] node [sloped]{$<$}	(n4x6);
		\draw[co2,line width=1pt]	(n2x7) to[out=180,in=0] node [sloped] {$>$}	(n1x1);
		\draw[co4,line width=1pt]	(n2x10) to[out=200,in=50] node [sloped]{$>$}	(n3x4);
		\draw[co3,line width=1pt]	(n2x13) to[out=270,in=90] node [sloped]{$<$}	(n4x3);
		%%\draw[co4,line width=1.5pt]	(n3x2) to[out=40,in=250] node [sloped]{$>$}	(n2x12);
		%%\draw[co3,line width=1.5pt]	(n3x5) to[out=90,in=270] node [sloped]{$>$}	(n1x11);
		\draw[co2,line width=1.5pt]	(n3x0) to[out=0,in=180] node [sloped] {$>$}	(n4x8);
       %%\draw[co3,line width=1.5pt]	(n4x2) to[out=90,in=270] node [sloped] {$>$}	(n2x14);
		%%\draw[co4,line width=1.5pt]	(n4x5) to[out=140,in=320] node [sloped] {$<$}	(n1x15);
		\draw[co2,line width=1.5pt]	(n4x9) to[out=180,in=0] node [sloped] {$<$}	(n3x15);
	\foreach \i in 
		{n1x0,n1x10,n1x14,n2x7,n2x10,n2x13}
			\fill[c1] (\i) circle [radius=4pt];
	\foreach \i in 
		{n3x0,n4x9}
			\fill[c3] (\i) circle [radius=4pt];
	\foreach \i in 
		{n1x8,n2x0,n1x1,n2x8,n3x4,n3x6,n4x3,n4x6}
			\fill[c1!\int] (\i) circle [radius=4pt];
	\foreach \i in 
		{n4x0,n3x8,n3x15,n4x8}
			\fill[c3!\int] (\i) circle [radius=4pt];
	\draw[co1,line width=1pt] (n2x0) to[loop right] node [sloped] {$$} ();
	\draw[co1,line width=1pt] (n4x0) to[loop right] node [sloped] {$$} ();
	\draw[co1,line width=1pt] (n1x8) to[loop left] node [sloped] {$$} ();
	\draw[co1,line width=1pt] (n3x8) to[loop left] node [sloped] {$$} ();
	\foreach \i in {1,...,\NN}{
		\fill[c\i!\int] (n\i) circle [radius=20pt];
		\draw[c\i!\int,line width=1pt]	(n\i) to[out=180,in=0] node [color=black] {$x_\i$}		(n\i);}
\end{tikzpicture}
\end{center}	
\bas
\mathsf{J}&=&
\begin{pmatrix}
{f^{{1}}_{{1}}} + {f^{{1}}_{{2}}}& 0&0&0\\
{f^{{1}}_{{3}}} + {f^{{1}}_{{4}}}&{f^{{2}}_{{1}}} + {f^{{2}}_{{2}}}&0&0\\
{f^{{1}}_{{2}}}& 0&{f^{{1}}_{{1}}} - {f^{{1}}_{{2}}}& 0\\
{f^{{1}}_{{4}}}&{f^{{2}}_{{2}}}&{f^{{1}}_{{3}}} - {f^{{1}}_{{4}}}&{f^{{2}}_{{1}}} - {f^{{2}}_{{2}}}
\end{pmatrix}.
\eas
\end{Example}

\subsection{Example with $N=5,\protect\CC=1$}\label{sec:5x1}
	\def\NN{5}
We take our example from \cite{Gracht2021}:
\bas
\dot{x}_1&=&f^1({x_1}, {x_2}, {x_3}, {x_4}, {x_5}), \\
\dot{x}_2&=&f^1({x_2}, {x_5}, {x_4}, {x_5}, {x_5}), \\
\dot{x}_3&=&f^1({x_3}, {x_4}, {x_5}, {x_5}, {x_5}), \\
\dot{x}_4&=&f^1({x_4}, {x_5}, {x_5}, {x_5}, {x_5}), \\
\dot{x}_5&=&f^1({x_5}, {x_5}, {x_5}, {x_5}, {x_5}), 
\eas
with (equivalent) picture (see Notation \ref{not:pic} 
for the meaning of all the different cells, arrows, and colors):
	\colorlet{co1}{black}
	\colorlet{co2}{orange}
	\colorlet{co3}{violet}
	\colorlet{co4}{blue}
	\colorlet{co5}{green}
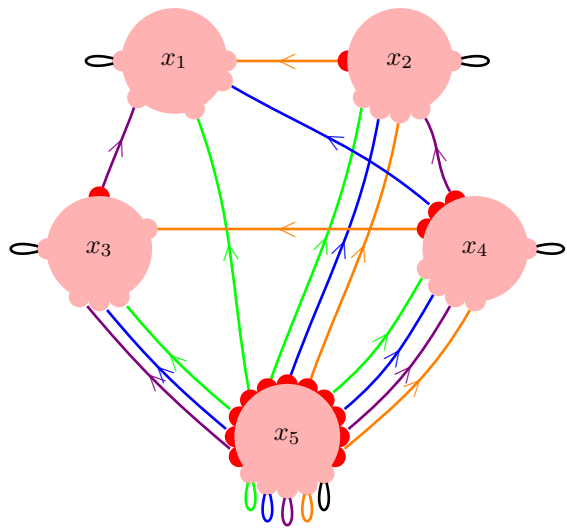
	\begin{center}
	\colorlet{co1}{black}
	\colorlet{co2}{orange}
	\colorlet{co3}{violet}
	\colorlet{co4}{blue}
	\colorlet{co5}{green}
	\colorlet{c1}{red}
	\colorlet{c2}{red}
	\colorlet{c3}{red}
	\colorlet{c4}{red}
	\colorlet{c5}{red}
	\begin{tikzpicture}
	\node (n1) at (0, 0) {1};
	\node (n2) at (3, 0) {2};
	\node (n3) at (-1, -2.5) {3};
	\node (n4) at (4, -2.5) {4};
	\node (n5) at (1.5, -5) {5};
	\node (z1) at (3, 0) {};
	\node (z2) at (6, 0) {};
	\node (z3) at (2, -2.5) {};
	\node (z4) at (7, -2.5) {};
	\node (z5) at (4.5, -5) {};
	\foreach \i in {0,...,15} 
		\foreach \x in {1,...,\NN} 
			\foreach \k in {x}	
				\node (n\x\k\i) at ($ (n\x)!0.23!360*\i/16:(z\x) $) {};
		\draw[co2,line width=1pt]	(n5x15) to[out=40,in=240] node [sloped] {$>$}	(n4x12);
		\draw[co3,line width=1pt]	(n5x0) to[out=40,in=240] node [sloped] {$>$}	(n4x11);
		\draw[co4,line width=1pt]	(n5x1) to[out=40,in=240] node [sloped] {$>$}	(n4x10);
		\draw[co5,line width=1pt]	(n5x2) to[out=40,in=240] node [sloped] {$>$}	(n4x9);
		\draw[co3,line width=1pt]	(n5x9) to[out=140,in=310] node [sloped] {$<$}	(n3x11);
		\draw[co4,line width=1pt]	(n5x8) to[out=140,in=310] node [sloped] {$<$}	(n3x12);
		\draw[co5,line width=1pt]	(n5x7) to[out=140,in=310] node [sloped] {$<$}	(n3x13);
		\draw[co2,line width=1pt]	(n5x3) to[out=70,in=260] node [sloped] {$>$}	(n2x12);
		\draw[co4,line width=1pt]	(n5x4) to[out=70,in=260] node [sloped] {$>$}	(n2x11);
		\draw[co5,line width=1pt]	(n5x5) to[out=70,in=260] node [sloped] {$>$}	(n2x10);
		\draw[co5,line width=1pt]	(n5x6) to[out=100,in=290] node [sloped] {$<$}	(n1x13);
		\draw[co2,line width=1pt]	(n4x7) to[out=180,in=0] node [sloped] {$<$}	(n3x1);
		\draw[co4,line width=1pt]	(n4x6) to[out=140,in=330] node [sloped] {$<$}	(n1x15);
		\draw[co3,line width=1pt]	(n4x5) to[out=120,in=300] node [sloped] {$<$}	(n2x13);
		\draw[co3,line width=1pt]	(n3x4) to[out=70,in=250] node [sloped] {$>$}	(n1x10);
		\draw[co2,line width=1pt]	(n2x8) to[out=180,in=0] node [sloped] {$<$}	(n1x0);
	\foreach \i in {n5x3,n5x4,n5x5,n5x6,n5x15,n5x0,n5x1,n5x2,n5x7,n5x8,n5x9,n4x6,n4x7,n4x5,n3x4,n2x8}
			\fill[c1] (\i) circle [radius=4pt];
	\foreach \i in {n3x8,n5x10,n5x11,n5x12,n5x13,n5x14,n1x8,n2x0,n4x0,n4x9,n4x10,n4x11,n4x12,n3x11,n3x12,n3x13,n2x10,n2x11,n2x12,n1x13,n3x1,n1x15,n2x13,n1x10,n1x0}
			\fill[c1!\int] (\i) circle [radius=4pt];
	\draw[co5,line width=1pt] (n5x10) to[loop below] node [sloped] {$$} ();
	\draw[co4,line width=1pt] (n5x11) to[loop below] node [sloped] {$$} ();
	\draw[co3,line width=1pt] (n5x12) to[loop below] node [sloped] {$$} ();
	\draw[co2,line width=1pt] (n5x13) to[loop below] node [sloped] {$$} ();
	\draw[co1,line width=1pt] (n5x14) to[loop below] node [sloped] {$$} ();
	\draw[co1,line width=1pt] (n2x0) to[loop right] node [sloped] {$$} ();
	\draw[co1,line width=1pt] (n4x0) to[loop right] node [sloped] {$$} ();
	\draw[co1,line width=1pt] (n1x8) to[loop left] node [sloped] {$$} ();
	\draw[co1,line width=1pt] (n3x8) to[loop left] node [sloped] {$$} ();
	\foreach \i in {1,...,\NN}{
		\fill[c\i!\int] (n\i) circle [radius=20pt];
		\draw[c\i!\int,line width=1pt]	(n\i) to[out=180,in=0] node [color=black] {$x_\i$}		(n\i);}
\end{tikzpicture}
\end{center}

 By applying Algorithm \ref{alg:xbasis} and Definition \ref{def:xy},
	the relation between the coordinates is given by
	\bas
		\begin {pmatrix}\cu{1}\\\bu{1}\\\bu{2}\\\bu{3}\\\bu{4}\\\cU{1}\\\bU{1}\\\bU{2}\\\bU{3}\\\bU{4}\end{pmatrix}
		=\begin {pmatrix}\vv{0}{1}\\\vv{1}{1}\\\vv{2}{1}\\\vv{3}{1}\\\vv{4}{1}\\\VV{0}{1}\\\VV{1}{1}\\\VV{2}{1}\\\VV{3}{1}\\\VV{4}{1}\end{pmatrix}
		&=&	\exp(\xx{0}{1}\otimes 	\sum_{q=1}^{4} \XX{q}{1}\triangleright)
		\begin {pmatrix}\xx{0}{1}\\\xx{1}{1}\\\xx{2}{1}\\\xx{3}{1}\\\xx{4}{1}\\\XX{0}{1}\\\XX{1}{1}\\\XX{2}{1}\\\XX{3}{1}\\\XX{4}{1}\end{pmatrix}
			=\begin {pmatrix}\xx{0}{1}\\\xx{1}{1}+\xx{0}{1}\\\xx{2}{1}+\xx{0}{1}\\\xx{3}{1}+\xx{0}{1}\\\xx{4}{1}+\xx{0}{1}\\\XX{0}{1}-	\sum_{q=1}^{4} \XX{q}{1}\\
			\XX{1}{1}\\\XX{2}{1}\\\XX{3}{1}\\\XX{4}{1}\end{pmatrix}
			=\begin {pmatrix}\cx{1}\\\bx{1}+\cx{1}\\\bx{2}+\cx{1}\\\bx{3}+\cx{1}\\\bx{4}+\cx{1}\\\cX{1}-	\sum_{q=1}^{4} \bX{q}\\\bX{1}\\\bX{2}\\\bX{3}\\\bX{4}\end{pmatrix}.
	\eas
\begin{Example}\label{exm:dim5col1exm1a}
For the introduction and the picture, see Example \ref{exm:dim5col1exm1}.
The Jacobi-matrix is
\bas
J&=&\begin{pmatrix}
{f^{{1}}_{{1}}}&{f^{{1}}_{{2}}}&{f^{{1}}_{{3}}}&{f^{{1}}_{{4}}}&{f^{{1}}_{{5}}}\\
 0&{f^{{1}}_{{1}}}& 0&{f^{{1}}_{{3}}}&{f^{{1}}_{{2}}} + {f^{{1}}_{{4}}} + {f^{{1}}_{{5}}}\\
 0& 0&{f^{{1}}_{{1}}}&{f^{{1}}_{{2}}}&{f^{{1}}_{{3}}} + {f^{{1}}_{{4}}} + {f^{{1}}_{{5}}}\\
 0& 0& 0&{f^{{1}}_{{1}}}&{f^{{1}}_{{2}}} + {f^{{1}}_{{3}}} + {f^{{1}}_{{4}}} + 
      {f^{{1}}_{{5}}}\\
 0& 0& 0& 0&{f^{{1}}_{{1}}} + {f^{{1}}_{{2}}} + {f^{{1}}_{{3}}} + 
      {f^{{1}}_{{4}}} + {f^{{1}}_{{5}}}
\end{pmatrix}
.
\eas
The relation between the coordinates (using Algorithm \ref{alg:ff}) is given by
\bas
\begin{pmatrix}
x_5\\
x_1\\
x_2\\
x_3\\
x_4
\end{pmatrix}=\begin{pmatrix}
\vv01
\\
\vv11
\\
\vv21
\\
\vv31
\\
\vv41
\end{pmatrix}=\begin{pmatrix}
\cu1\\
\bu1
\\
\bu2
\\
\bu3
\\
\bu4
\end{pmatrix}.
\eas
	Using Algorithm \ref{alg:xbasis} we obtain
the Jacobi-matrix
\bas
\mathsf{J}&=&
\begin{pmatrix}
{f^{{1}}_{{1}}} + {f^{{1}}_{{2}}} + {f^{{1}}_{{3}}} + 
      {f^{{1}}_{{4}}} + {f^{{1}}_{{5}}}&0&0&0&0\\
 0&{f^{{1}}_{{1}}}& 0& 0& 0\\
 0&{f^{{1}}_{{2}}}&{f^{{1}}_{{1}}}& 0& 0\\
 0&{f^{{1}}_{{3}}}& 0&{f^{{1}}_{{1}}}& 0\\
 0&{f^{{1}}_{{4}}}&{f^{{1}}_{{3}}}&{f^{{1}}_{{2}}}&{f^{{1}}_{{1}}}
\end{pmatrix}
,
\eas
with multipliers
\bas
\bigwedge^1&=&\begin{pmatrix}{f^{{1}}_{{1}}} + {f^{{1}}_{{2}}} + {f^{{1}}_{{3}}} + 
      {f^{{1}}_{{4}}} + {f^{{1}}_{{5}}}\end{pmatrix},\\
\bigwedge^2&=&\begin{pmatrix}
{f^1_1}& 0\\
 {f^1_2}&{f^1_1}
\end{pmatrix},
\eas
with \(n_1=1, n_2=2\) and \(m_1=1, m_2=2\), satisfying \(m_1 n_1+m_2n_2=\NN=N\) and \(n_1^2+n_2^2=\NN=N\).
The Jacobian is
\bas
|\mathsf{J}|&=&|\bigwedge^1|^{m_1}\cdot|\bigwedge^2|^{m_2}=({f^1_1})^4
({f^{{1}}_{{1}}} + {f^{{1}}_{{2}}} + {f^{{1}}_{{3}}} + 
      {f^{{1}}_{{4}}} + {f^{{1}}_{{5}}}).
\eas
%The Jordan-Chevalley decomposition is obvious.
\end{Example}

\subsection{Example with \(N=8,\protect\CC=1\)}
%\issue{Jan}{Give more explanation about the example}
\begin{Example}\label{exm:dim8col1exm1}
In \cite{MR4183886} one color (homogeneous) networks are studied in which every cell has at most one input from every other cell (itself included). This is called asymmetry.
The restriction to one color is made to avoid notational complexity.
The asymmetry condition is not important, which is a good thing,
since the given example does not satisfy the condition, but does produce the results predicted by theorem 1.2 in loc. cit.
This allows the application of representation theory to find matrices \(\bigwedge^i\), such that the Jacobian is a product of
the determinants of these matrices (which are therefore called multipliers). This gives a way to analyze the spectrum of the linearized system and its degeneracies.
We take \cite[Example 4.9]{MR4183886} to apply our decomposition method and then look at the result in the light of the multipliers.
\bas
\dot{x}_1&=&f^1({x_1}, {x_2}, {x_3}, {x_4}, {x_5}, {x_6}, {x_7}, {x_8}),\\
\dot{x}_2&=&f^1({x_2}, {x_6}, {x_4}, {x_8}, {x_2}, {x_6}, {x_7}, {x_8}),\\
\dot{x}_3&=&f^1({x_3},{x_5},{x_7},{x_3},{x_8},{x_6},{x_7},{x_8}),\\
\dot{x}_4&=&f^1({x_4},{x_2},{x_7},{x_4},{x_8},{x_6},{x_7},{x_8}),\\
\dot{x}_5&=&f^1({x_5},{x_6},{x_3},{x_8},{x_5},{x_6},{x_7},{x_8}),\\
\dot{x}_6&=&f^1({x_6},{x_6},{x_8},{x_8},{x_6},{x_6},{x_7},{x_8}),\\
\dot{x}_7&=&f^1({x_7},{x_8},{x_7},{x_7},{x_8},{x_6},{x_7},{x_8}),\\
\dot{x}_8&=&f^1({x_8},{x_6},{x_7},{x_8},{x_8},{x_6},{x_7},{x_8}).
\eas
The Jacobi-matrix is
\bas
J&=&\begin{pmatrix}
{f^{{1}}_{{1}}}&{f^{{1}}_{{2}}}&{f^{{1}}_{{3}}}&{f^{{1}}_{{4}}}&{f^{{1}}_{{5}}}&{f^{{1}}_{{6}}}&{f^{{1}}_{{7}}}&{f^{{1}}_{{8}}}\\
 0&{f^{{1}}_{{1}}} + {f^{{1}}_{{5}}}& 0&{f^{{1}}_{{3}}}& 0&{f^{{1}}_{{2}}} + {f^{{1}}_{{6}}}&{f^{{1}}_{{7}}}&{f^{{1}}_{{4}}} + {f^{{1}}_{{8}}}\\
 0& 0&{f^{{1}}_{{1}}} + {f^{{1}}_{{4}}}& 0&{f^{{1}}_{{2}}}&{f^{{1}}_{{6}}}&{f^{{1}}_{{3}}} + {f^{{1}}_{{7}}}&{f^{{1}}_{{5}}} + {f^{{1}}_{{8}}}\\
 0&{f^{{1}}_{{2}}}& 0&{f^{{1}}_{{1}}} + {f^{{1}}_{{4}}}& 0&{f^{{1}}_{{6}}}&{f^{{1}}_{{3}}} + {f^{{1}}_{{7}}}&{f^{{1}}_{{5}}} + {f^{{1}}_{{8}}}\\
 0& 0&{f^{{1}}_{{3}}}& 0&{f^{{1}}_{{1}}} + {f^{{1}}_{{5}}}&{f^{{1}}_{{2}}} + {f^{{1}}_{{6}}}&{f^{{1}}_{{7}}}&{f^{{1}}_{{4}}} + {f^{{1}}_{{8}}}\\
 0& 0& 0& 0& 0&{f^{{1}}_{{1}}} + {f^{{1}}_{{2}}} + {f^{{1}}_{{5}}} + 
      {f^{{1}}_{{6}}}&{f^{{1}}_{{7}}}&{f^{{1}}_{{3}}} + {f^{{1}}_{{4}}} + {f^{{1}}_{{8}}}\\
 0& 0& 0& 0& 0&{f^{{1}}_{{6}}}&{f^{{1}}_{{1}}} + {f^{{1}}_{{3}}} + {f^{{1}}_{{4}}} + 
      {f^{{1}}_{{7}}}&{f^{{1}}_{{2}}} + {f^{{1}}_{{5}}} + {f^{{1}}_{{8}}}\\
 0& 0& 0& 0& 0&{f^{{1}}_{{2}}} + {f^{{1}}_{{6}}}&{f^{{1}}_{{3}}} + {f^{{1}}_{{7}}}&{f^{{1}}_{{1}}} + {f^{{1}}_{{4}}} + {f^{{1}}_{{5}}} + 
      {f^{{1}}_{{8}}}
\end{pmatrix}.
\eas
Its trace is
\bas
\tr{J}&=&8{f^1_1}+ {f^1_2}+ {f^1_3} +4{f^1_4} +4{f^1_5} +{f^1_6}  +{f^1_7}  +{f^1_8}  
\\&=&({f^1_1}+ {f^1_2}+ {f^1_3} +{f^1_4} +{f^1_5} +{f^1_6}  +{f^1_7}  +{f^1_8}  )
+{f^1_1} +3({f^1_4} +{f^1_5}+2f^1_1),
\eas
suggesting the existence of two one-dimensional blocks \(\bigwedge^{1,2}\) with multiplicity \(m_{1,2}=1\) and one irreducible  two-dimensional block \(\bigwedge^3\) with
multiplicity \(m_3=3\), so that \(m_1n_1+m_2n_2+m_3n_3=8\) and \(n_1^2+n_2^2+n_3^2\leq 8\),
cf. \cite[Proposition 6.8]{MR4183886}.
The relation between the coordinates (using Algorithm \ref{alg:ff}) is given by
\bas
\begin{pmatrix}
x_8\\
x_1\\
x_2\\
x_3\\
x_4\\
x_5\\
x_6\\
x_7\\
\end{pmatrix}=\begin{pmatrix}
\vv0{}
\\
\vv1{}
\\
\vv2{}
\\
\vv3{}
\\
\vv4{}
\\
\vv5{}
\\
\vv6{}
\\
\vv7{}
\\
\end{pmatrix}=\begin{pmatrix}
\cu1\\
\bu1
\\
\bu2
\\
\bu3
\\
\bu4
\\
\bu5
\\
\bu6
\\
\bu7
\\
\end{pmatrix}.
\eas
Using Algorithm \ref{alg:xbasis} we obtain the Jacobi-matrix
\bas
\mathsf{J}&=&
\begin{pmatrix}
{f^{{1}}_{{1}}} + {f^{{1}}_{{2}}} + {f^{{1}}_{{3}}} + 
      {f^{{1}}_{{4}}} + {f^{{1}}_{{5}}} + {f^{{1}}_{{6}}} + 
      {f^{{1}}_{{7}}} + {f^{{1}}_{{8}}}&0&0&0&0&0&0&0\\
 0&{f^{{1}}_{{1}}}& 0& 0& 0& 0& 0& 0\\
 0&{f^{{1}}_{{2}}}&{f^{{1}}_{{1}}} + {f^{{1}}_{{5}}}& 0&{f^{{1}}_{{2}}}& 0& 0& 0\\
 0&{f^{{1}}_{{3}}}& 0&{f^{{1}}_{{1}}} + {f^{{1}}_{{4}}}& 0&{f^{{1}}_{{3}}}& 0& 0\\
 0&{f^{{1}}_{{4}}}&{f^{{1}}_{{3}}}& 0&{f^{{1}}_{{1}}} + {f^{{1}}_{{4}}}& 0& 0& 0\\
 0&{f^{{1}}_{{5}}}& 0&{f^{{1}}_{{2}}}& 0&{f^{{1}}_{{1}}} + {f^{{1}}_{{5}}}& 0& 0\\
{f^{{1}}_{{2}}} + {f^{{1}}_{{6}}}& - {f^{{1}}_{{2}}}& 0& - {f^{{1}}_{{2}}}& - {f^{{1}}_{{2}}}& 0&{f^{{1}}_{{1}}} + {f^{{1}}_{{5}}}& - {f^{{1}}_{{2}}}\\
{f^{{1}}_{{3}}} + {f^{{1}}_{{7}}}& - {f^{{1}}_{{3}}}& - {f^{{1}}_{{3}}}& 0& 0& - {f^{{1}}_{{3}}}& - {f^{{1}}_{{3}}}&{f^{{1}}_{{1}}} + {f^{{1}}_{{4}}}
\end{pmatrix}.
\eas

Using some elementary (colored) linear algebra (changing the numbering of the \(\bu{}\) in the process), we obtain
\bas
\mathsf{J}&=&
\begin{pmatrix}
{f^1_1} + {f^1_2} + {f^1_3} + {f^1_4} + {f^1_5} + {f^1_6} + {f^1_7} + {f^1_8}&0&0&0&0&0&0&0\\
 0&{f^1_1}& 0& 0& 0& 0& 0& 0\\
 0& 0&{f^1_1} + {f^1_4}&{f^1_3}& 0& 0& 0& 0\\
 0& 0&{f^1_2}&{f^1_1} + {f^1_5}& 0& 0& 0& 0\\
 0& 0& 0& 0&{f^1_1} + {f^1_4}&{f^1_3}& 0& 0\\
 0& 0& 0& 0&{f^1_2}&{f^1_1} + {f^1_5}& 0& 0\\
-{f^1_3} - {f^1_7}& 0& 0& 0& 0& 0&{f^1_1} + {f^1_4}&  {f^1_3}\\
{f^1_2} + {f^1_6}& 0& 0& 0& 0& 0&  {f^1_2}&{f^1_1} + {f^1_5}
\end{pmatrix}.
\eas
We have obtained the results predicted by the representation theory but in a completely elementary fashion.
This cleaning up of the Jacobi-matrix is not yet fully automated by us; it would be interesting to have an algorithm.
One potential problem is that trace decomposition generally allows for some choice, so there is no a priori knowledge
of the trace of the multipliers.
\begin{Remark}
One might, optimistically, think that such a nice diagonal result is always possible, but this result is so nice because of the 
characteristics of the example. If one makes slight changes to the starting equation, the result may not be as nice.
\end{Remark}

We recognize (in the notation of \cite[Example 4.9]{MR4183886})  the network multipliers
\bas
\bigwedge^1&=&\begin{pmatrix}{f^1_1} + {f^1_2} + {f^1_3} + {f^1_4} + {f^1_5} + {f^1_6} + {f^1_7} + {f^1_8}\end{pmatrix},\\
\bigwedge^2&=&\begin{pmatrix} f^1_1\end{pmatrix},\\
\bigwedge^3&=&\begin{pmatrix} f^1_1+f^1_4&f^1_3\\f^1_2&f^1_1+f^1_5\end{pmatrix},
\eas
with \(n_1=n_2=1,n_3=2\) and \(m_1=m_2=1,m_3=3\), with \(m_1n_1+m_2n_2+m_3n_3=8=N\) and \(n_1^2+n_2^2+n_3^2=6\leq 8=N\),
cf. \cite[Proposition 6.8]{MR4183886}.
Notice that 
\bas
\tr{\mathsf{J}}&=&{f^1_8} + {f^1_7} + {f^1_6} +3 {f^1_5} +3 {f^1_4} + {f^1_3} + 3{f^1_2} +8{f^1_1}
\\&=&m_1\tr{\bigwedge^1}+m_2\tr{\bigwedge^2}+m_3\tr{\bigwedge^3}.
\eas
and
\bas
|{\mathsf{J}}|&=&({f^1_8} + {f^1_7} + {f^1_6} + {f^1_5} + {f^1_4} + {f^1_3} + {f^1_2} + {f^1_1})f^1_1
\begin{vmatrix}{f^1_5} + {f^1_1}& - {f^1_2}\\
- {f^1_3}&{f^1_4} + {f^1_1}\end{vmatrix}^3
\\&=&|{\bigwedge^1}|^{m_1}|{\bigwedge^2}|^{m_2}|{\bigwedge^3}|^{m_3}.
\eas
%\issue{JS}{What this shows, is that the transformation is a kind of approximate Fourier transform.
%I think we should try and make it into a Fourier transform}
We now leave the linear algebra and ask for all the nilpotents in this system. These have to satisfy
\bas
{\bigwedge^1}&=&\begin{vmatrix}{f^1_1} + {f^1_2} + {f^1_3} + {f^1_4} + {f^1_5} + {f^1_6} + {f^1_7} + {f^1_8}\end{vmatrix}=0,\\
{\bigwedge^2}&=&\begin{vmatrix}f^1_1\end{vmatrix}=0,\\
\tr{\bigwedge^3}&=&f^1_4+f^1_5+2f^1_1=0,\\
|{\bigwedge^3}|&=&\begin{vmatrix} f^1_1+f^1_4&f^1_2\\f^1_3&f^1_1+f^1_5\end{vmatrix}=0.
\eas
This leads to
\bas
f^1_1&=&0,\\
f^1_5&=&-f^1_4,\\
{f^1_8} &=&- {f^1_7} - {f^1_6}  - {f^1_3} - {f^1_2},\\
({f^1_4})^2+f^1_2f^1_3&=&0.
\eas
This reduces \(\bigwedge^3\) to
\bas
\begin{pmatrix} f^1_4&f^1_2\\f^1_3&-f^1_4\end{pmatrix}&\equiv &\begin{pmatrix} 0&1\\0&0\end{pmatrix},
\eas
using \(P=\begin{pmatrix} f^1_2 &0\\-f^1_4 &1 \end{pmatrix}\) as (assuming \(f^1_2\neq 0\), invertible; if \(f^1_2=0\), we assume \(f^1_3\neq 0\), similar computation; if both are zero, the multiplier itself is zero.) intertwiner,
since 
\bas
\begin{pmatrix} f^1_4&f^1_2\\f^1_3&-f^1_4\end{pmatrix}\begin{pmatrix} f^1_2 &0\\-f^1_4 &1 \end{pmatrix}
&=& \begin{pmatrix} 0&f^1_2\\0&-f^1_4\end{pmatrix}=
\begin{pmatrix} f^1_2 &0\\-f^1_4 &1 \end{pmatrix}\begin{pmatrix} 0&1\\0&0\end{pmatrix}.
\eas

\end{Example}

\usetikzlibrary{arrows, positioning, shapes.geometric}

\subsection{\(2\)D Coupled Oscillators with Two Masses}\label{sec:2osc}
%Consider a 2D system with two masses \( m_i, i=1,2\),  connected by springs in both the horizontal {\color{red} and vertical directions}. Each mass has one degree of freedom, position \( x_i \) and velocity \( y_i \). The system can be described as follows:
%\issue{JS}{I would think that in this model each mass can move in one dimension (it has one degree of freedom), where it has a position and a velocity, so where are the springs in vertical directions?}
The following analysis is meant  to illustrate the results of this paper,
not as the best way to study the equation, since the semigroupoid in this case is an abelian group, \(\mathbb{Z}/2\),
so Fourier analysis is much more effective.

Consider the following system:

%\subsection*{Equations of Motion}
\bas
m \ddot{x}_1 &=& -(k + \kappa)x_1 + \kappa x_2,
\\
m \ddot{x}_2 &=& -(k + \kappa)x_2 + \kappa x_1.
\eas
Equivalently,
\bas
\dot{x}_1&=&y_1,
\\
\dot{y}_1&=&\frac{1}{m}\left( -(k + \kappa)x_1 + \kappa x_2\right),
\\
\dot{x}_2&=&y_2,
\\
\dot{y}_2&=&\frac{1}{m}\left( -(k + \kappa)x_2 + \kappa x_1\right).
\eas
This system of equations describes the motion of two coupled oscillators, where $x_1$ and $x_2$ represent the one-dimensional displacements of two masses, and $m$ is the mass of each oscillator. The parameters $k$ and $\kappa$ represent the stiffness of the springs connecting the oscillators.  Set,
\bas
\dot{X}_1&=&
\begin{bmatrix}
\dot{x}_1 \\
\dot{y}_1
\end{bmatrix}
=
\begin{bmatrix}
0 & 1 \\
-\frac{k + \kappa}{m} &0
\end{bmatrix}
\begin{bmatrix}
x_1 \\
y_1
\end{bmatrix}+
\begin{bmatrix}
0 & 0 \\
\frac{\kappa}{m} & 0
\end{bmatrix}
\begin{bmatrix}
x_2 \\
y_2
\end{bmatrix}=M X_1+N X_2,
\eas
and 
\bas
\dot{X}_2=\begin{bmatrix}
\dot{x}_2 \\
\dot{y}_2
\end{bmatrix}
=
\begin{bmatrix}
0 & 1 \\
-\frac{k + \kappa}{m} &0
\end{bmatrix}
\begin{bmatrix}
x_2 \\
y_2
\end{bmatrix}+
\begin{bmatrix}
0 & 0 \\
\frac{\kappa}{m} & 0
\end{bmatrix}
\begin{bmatrix}
x_1 \\
y_1
\end{bmatrix}= M X_2+N X_1.
\eas
Now, we have the following system which is one color network system
\bas
\dot{X}_1&=&M X_1+N X_2,
\\
\dot{X}_2&=& M X_2+N X_1.
\eas
\iffalse
\bas
\begin{bmatrix}
\dot{X}_1 \\
\dot{X}_2
\end{bmatrix}
=
\begin{bmatrix}
M & N \\
N &M
\end{bmatrix}
\begin{bmatrix}
X_1 \\
X_2
\end{bmatrix}
\eas
	\fi
 \def\NN{2}

By applying Algorithm \ref{alg:xbasis} and Definition \ref{def:xy},
	the relation between the coordinates is given by
		\ba\label{eqn:2x1S}
			\begin {pmatrix}\vv{0}{1}\\\vv{1}{1}\\\VV{0}{1}\\\VV{1}{1}\end{pmatrix}&=&
			\exp(\xx{0}{1}\otimes 	\XX{1}{1}\triangleright)
			\begin {pmatrix}\xx{0}{1}\\\xx{1}{1}\\\XX{0}{1}\\\XX{1}{1}\end{pmatrix}
			=\begin {pmatrix}\xx{0}{1}\\\xx{1}{1}+\xx{0}{1}\\\XX{0}{1}-	\XX{1}{1}\\\XX{1}{1}\end{pmatrix}
			=\begin {pmatrix}\cx{1}\\\bx{1}+\cx{1}\\\cX{1}-	\bX{1}\\\bX{1}\end{pmatrix}.
		\ea

\bas
\dot{X}_1&=&f^1({X_1}, {X_2}),\\
\dot{X}_2&=&f^1({X_2}, {X_1}),\\
\eas
with picture
%\issue{JS}{Homework}
	\colorlet{co1}{black}
	\colorlet{co2}{orange}
	\colorlet{co3}{violet}
	\colorlet{co4}{blue}
	\colorlet{co5}{green}
\begin{center}
	\colorlet{co1}{black}
	\colorlet{co2}{orange}
	\colorlet{co3}{violet}
	\colorlet{co4}{blue}
	\colorlet{co5}{green}
	\colorlet{c1}{red}
	\colorlet{c2}{red}
	\begin{tikzpicture}
	\node (n1) at (0,0) {1};
	\node (n2) at (3,0) {2};
	\node (z1) at (3,0) {};
	\node (z2) at (6,0) {};
	\foreach \i in {0,...,15} 
		\foreach \x in {1,...,\NN} 
			\foreach \k in {x}	
				\node (n\x\k\i) at ($ (n\x)!0.23!360*\i/16:(z\x) $) {};
%		\draw[co2,line width=1pt] (n1x1) to[out=0,in=180] node [sloped] {$>$} (n2x7);
		\draw[co3,line width=1pt] (n1x14) to[out=0,in=180] node [sloped] {$>$} (n2x10);
		\draw[co3,line width=1pt] (n2x6) to[out=180,in=0] node [sloped] {$<$} (n1x2);
%		\draw[co4,line width=1pt] (n2x9) to[out=180,in=0] node [sloped] {$<$} (n1x15);
		\foreach \i in 
			{n1x14,n2x6}
				\fill[c1] (\i) circle [radius=4pt];
		\foreach \i in 
                {n1x7,n2x10,n1x2,n2x15}
\fill[c1!\int] (\i) circle [radius=4pt];
		\draw[co1,line width=1pt] (n1x7) to[loop left] node [sloped] {$$} ();
		\draw[co1,line width=1pt] (n2x15) to[loop right] node [sloped] {$$} ();
	\foreach \i in {1,...,\NN}{
		\fill[c\i!\int] (n\i) circle [radius=20pt];
		\draw[c\i!\int,line width=1pt]	(n\i) to[out=180,in=0] node [color=black] {$X_\i$}		(n\i);}
	\end{tikzpicture}
\end{center}	
The Jacobi-matrix is
\bas
J&=&\begin{pmatrix}
M& N\\
N &M
\end{pmatrix}.
\eas
%with \(\tr{J}=2M.\) 
The relation between the coordinates (using Algorithm \ref{alg:ff}) is given by
\bas
\begin{pmatrix}
X_1\\
X_2\\
\end{pmatrix}=\begin{pmatrix}
\vv01
\\
\vv11
\\
\end{pmatrix}=\begin{pmatrix}
\cu1\\
\bu1
\\
\end{pmatrix}.
\eas
Using Algorithm \ref{alg:xbasis} we obtain the Jacobi matrix
\bas
\mathsf{J}&=&\begin{pmatrix}
M + N& 0\\
N & M-N 
\end{pmatrix},
\eas
and the `multipliers' (Cf. Remark \ref{rem:vector}) are 
\bas
\bigwedge^1&=&\begin{pmatrix}M+N\end{pmatrix},\\
\bigwedge^2&=& \begin{pmatrix}M-N\end{pmatrix}.
\eas
The Jacobian is
\bas
|\mathsf{J}|&=&|\bigwedge^1|\cdot|\bigwedge^2|=|M+N|
| M-N |=-\frac{k}{m}\frac{k-2\kappa}{m}.
\eas
%If we let \(\tau= 2M \), \(\lambda= 2N  \) and \(\alpha=N \),
\iffalse
The Jacobi matrix can be written as
\bas
\begin{pmatrix}
M& 0\\
0& M
\end{pmatrix}
+\begin{pmatrix}
N& 0\\
N& -N
\end{pmatrix},
\eas
where the first matrix is in the solvable part and the second in the semisimple part of the Levi decomposition.
We are not going to carry out this splitting in the following examples, but it should be clear how to do this.
Notice that
\bas
[\begin{pmatrix}
M& 0\\
0& M
\end{pmatrix},\begin{pmatrix}
N& 0\\
N& - N
\end{pmatrix}]=
\begin{pmatrix}
[M,N]&0\\
[M,N]&-[M,N]
\end{pmatrix}.
\eas
This should illustrate the fact that Levi decomposition is a semidirect product, with the semisimple part acting on the solvable part.
but seems to do the opposite.
\issue{JS}{I do not understand this Levi decomposition at all, what does Maple have to say about the 2D example?}
\fi
\section{Concluding remarks}\label{sec:conclusions}
%\issue{FM-Jan}{It should be something interesting or nu}
The algorithmic approach presented in this paper is but a first step towards the classification of bifurcations of colored networks.
Obvious areas of further research are the computation of Jordan-Chevalley decompositions with many variables, the generalization of the Jacobson-Morozov construction
to the situation of colored networks, the construction of versal deformations of organizing centers, and the nonlinear network normal form description.

In our Examples, we see tantalizing connections with other approaches using representation theory.
The semigroup(oid) based approach to colored network theory that started with \cite{rink2015coupled} has, after a decade of research, opened up many interesting 
possibilities and connections between different approaches.
\bibliography{CN}

\newcommand{\etalchar}[1]{$^{#1}$}
\begin{thebibliography}{SEGHC{\etalchar{+}}15}

\bibitem[And04]{MR2189632}
Ian~M. Anderson.
\newblock Maple packages and {J}ava applets for classification problems in geometry and algebra.
\newblock In {\em Foundations of computational mathematics: {M}inneapolis, 2002}, volume 312 of {\em London Math. Soc. Lecture Note Ser.}, pages 193--206. Cambridge Univ. Press, Cambridge, 2004.

\bibitem[BBC{\etalchar{+}}14]{boccaletti2014structure}
Stefano Boccaletti, Ginestra Bianconi, Regino Criado, Charo~I Del~Genio, Jes{\'u}s G{\'o}mez-Gardenes, Miguel Romance, Irene Sendina-Nadal, Zhen Wang, and Massimiliano Zanin.
\newblock The structure and dynamics of multilayer networks.
\newblock {\em Physics reports}, 544(1):1--122, 2014.

\bibitem[DN21]{MR4183886}
Lee DeVille and Eddie Nijholt.
\newblock Circulant type formulas for the eigenvalues of linear network maps.
\newblock {\em Linear Algebra Appl.}, 610:379--439, 2021.

\bibitem[Ger63]{MR0161898}
M.~Gerstenhaber.
\newblock The cohomology structure of an associative ring.
\newblock {\em Ann. of Math. (2)}, 78:267--288, 1963.

\bibitem[GGP{\etalchar{+}}20]{MR4059374}
Punit Gandhi, Martin Golubitsky, Claire Postlethwaite, Ian Stewart, and Yangyang Wang.
\newblock Bifurcations on fully inhomogeneous networks.
\newblock {\em SIAM J. Appl. Dyn. Syst.}, 19(1):366--411, 2020.

\bibitem[GL09]{MR2481276}
Martin Golubitsky and Reiner Lauterbach.
\newblock Bifurcations from synchrony in homogeneous networks: linear theory.
\newblock {\em SIAM J. Appl. Dyn. Syst.}, 8(1):40--75, 2009.

\bibitem[GS17]{MR3606590}
Martin Golubitsky and Ian Stewart.
\newblock Coordinate changes for network dynamics.
\newblock {\em Dyn. Syst.}, 32(1):80--116, 2017.

\bibitem[Hum12]{humphreys2012introduction}
James~E Humphreys.
\newblock {\em Introduction to Lie algebras and representation theory}, volume~9.
\newblock Springer Science \& Business Media, 2012.

\bibitem[IPB{\etalchar{+}}06]{in2006complex}
Visarath In, Antonio Palacios, Adi~R Bulsara, Patrick Longhini, Andy Kho, Joseph~D Neff, Salvatore Baglio, and Bruno Ando.
\newblock Complex behavior in driven unidirectionally coupled overdamped duffing elements.
\newblock {\em Physical Review E}, 73(6):066121, 2006.

\bibitem[KUVV13]{Form00}
J.~Kuipers, T.~Ueda, J.A.M. Vermaseren, and J.~Vollinga.
\newblock {FORM} version 4.0.
\newblock {\em Computer Physics Communications}, 184(5):1453--1467, 2013.

\bibitem[Lev60]{MR0123464}
Eugen\'{\i}o~Elia Levi.
\newblock {\em Opere}.
\newblock A cura dell'Unione Matematica Italiana e col contributo del Consiglio Nazionale delle Ricerche. 2 Vols. Edizioniremonese, Rome, 1960.

\bibitem[MS19]{MR3953028}
Fahimeh Mokhtari and Jan~A. Sanders.
\newblock Versal normal form for nonsemisimple singularities.
\newblock {\em J. Differential Equations}, 267(5):3083--3113, 2019.

\bibitem[RS13]{MR3071397}
Bob~W. Rink and Jan~A. Sanders.
\newblock Amplified {H}opf bifurcations in feed-forward networks.
\newblock {\em SIAM J. Appl. Dyn. Syst.}, 12(2):1135--1157, 2013.

\bibitem[RS15]{rink2015coupled}
Bob Rink and Jan Sanders.
\newblock Coupled cell networks: semigroups, {L}ie algebras and normal forms.
\newblock {\em Transactions of the American Mathematical Society}, 367(5):3509--3548, 2015.

\bibitem[SEGHC{\etalchar{+}}15]{sevilla2015enhancing}
R~Sevilla-Escoboza, Ricardo Gutierrez, G~Huerta-Cuellar, S~Boccaletti, J~G{\'o}mez-Garde{\~n}es, A~Arenas, and JM~Buld{\'u}.
\newblock Enhancing the stability of the synchronization of multivariable coupled oscillators.
\newblock {\em Physical Review E}, 92(3):032804, 2015.

\bibitem[Ste16]{MR3525092}
Benjamin Steinberg.
\newblock {\em Representation theory of finite monoids}.
\newblock Universitext. Springer, Cham, 2016.

\bibitem[SVM07]{SVM2007}
J.~A. Sanders, F.~Verhulst, and J.~Murdock.
\newblock {\em Averaging {M}ethods in {N}onlinear {D}ynamical {S}ystems}, volume~59 of {\em Applied Mathematical Sciences}.
\newblock Springer, New York, second edition, 2007.

\bibitem[TW12]{MR3059183}
G.~Thompson and Z.~Wick.
\newblock Subalgebras of {${\mathfrak{gl}}(3,\mathbb{R})$}.
\newblock {\em Extracta Math.}, 27(2):201--230, 2012.

\bibitem[VDGNR22]{Gracht2021}
S{\"o}ren Von Der~Gracht, Eddie Nijholt, and Bob Rink.
\newblock Amplified steady state bifurcations in feedforward networks.
\newblock {\em Nonlinearity}, 35(4):2073, 2022.

\bibitem[Wu02]{wu2002synchronization}
Chai~Wah Wu.
\newblock {\em Synchronization in Coupled Chaotic Circuits \& Systems}, volume~41.
\newblock World Scientific, 2002.

\end{thebibliography}
% \tableofcontents
\end{document}